\documentclass[]{arXivclass} %

\DoubleSpacedXI %

\usepackage{endnotes}

\usepackage{natbib}
 \bibpunct[, ]{(}{)}{,}{a}{}{,}%
 \def\bibfont{\small}%

\usepackage{multibib}
\newcites{EC}{References}

\TheoremsNumberedThrough     %
\ECRepeatTheorems

\EquationsNumberedThrough    %

\usepackage[hypertexnames=false]{hyperref}
\usepackage{mathtools, nccmath}
\usepackage{braket} %
\usepackage{siunitx}
\usepackage{enumitem}
\usepackage{tikz}
\setlist[enumerate]{align=left}
\usepackage{amsmath,amssymb}
\usepackage{mathabx}
\usepackage{amsmath}
\usepackage{cleveref}
\usepackage{xcolor}
\usepackage[short]{optidef}
\DeclareDocumentEnvironment{miniinf}{D||{\defaultProblemFormat} O{\defaultConstraintFormat} D<>{} m m m m}{%
\ifthenelse{\equal{#3}{b}}{%
    \ifthenelse{\equal{#1}{s}}%
    {\setFormatShort{inf}{#4}\BaseMiniStar{#2}{#4}{#5}{#7}{inf}{#3}}%
    {\setFormatLong{infimum}{#4}\BaseMiniStar{#2}{#4}{#5}{#7}{infimum}{#3}}%
}{%
    \ifthenelse{\equal{#1}{s}}%
    {\setFormatShort{inf}{#4}\BaseMini{#2}{#4}{#5}{#6}{#7}{inf}}%
    {\setFormatLong{infimum}{#4}\BaseMini{#2}{#4}{#5}{#6}{#7}{infimum}}%
}%
}%
{\endBaseMini\toggletrue{bodyCon}}
                        
\usepackage{booktabs}
\usepackage{graphicx}
\usepackage[]{subcaption}
\usepackage{multirow}
\usepackage[tworuled,linesnumbered]{algorithm2e}
\usepackage{makecell}
\usetikzlibrary{arrows.meta}

\newcommand{\abs}[1]{\left\lvert#1\right\rvert}
\allowdisplaybreaks

\begin{document}

\RUNAUTHOR{You and Liu}

\RUNTITLE{Robustness Quantification of Tail-trimming Effects for In-CVaR Based Nonlinear Regression} 

\newcommand{\myfulltitle}{Robustness Quantification of Tail-trimming Effects for Interval CVaR Based Nonlinear Regression}%
\TITLE{\myfulltitle}

\ARTICLEAUTHORS{%
\AUTHOR{Yulei You}
\AFF{Department of Industrial Engineering, Tsinghua University, Beijing 100084, China, \EMAIL{youyl23@mails.tsinghua.edu.cn}} %
\AUTHOR{Junyi Liu}
\AFF{Department of Industrial Engineering, Tsinghua University, Beijing 100084, China, \EMAIL{junyiliu@tsinghua.edu.cn}}
} %

\ABSTRACT{%
Tail trimming is a widely adopted strategy for achieving robustness against outlier contamination in statistical learning. However, a thorough understanding of the extent to which tail trimming is effective for robust nonlinear regression remains largely lacking. In this paper, we investigate robust nonlinear regression through the lens of the interval conditional value-at-risk (In-CVaR), which guards against outliers by trimming extreme losses. We first characterize the trimming effect of the In-CVaR based estimator under distributional perturbation. Specifically, we establish that the In-CVaR based estimator is qualitatively robust with respect to the Prokhorov metric if and only if the largest portion of losses is trimmed. Furthermore, under sample contamination, we derive a probably approximately correct (PAC) bound for the In-CVaR based estimator, revealing that the preliminary covariate filtering commonly performed in the literature is in fact unnecessary for In-CVaR based regression. The nonconvexity of the In-CVaR objective, coupled with the unboundedness of the covariate, response, and parameter spaces, poses substantial technical challenges that necessitate new treatments of both qualitative robustness and PAC analysis. Overall, this study elucidates to what extent right-tail trimming of the loss via the In-CVaR risk measure is essential and effective for robust nonlinear regression.  These theoretical findings are further supported by numerical experiments.
}%


\maketitle
\vspace{-1\baselineskip}
\section{Introduction}\label{sec:Intro}

Robust regression is a prominent topic for mitigating the impact of distortion for
data-driven estimation in various applications \citep{bertsimas2018characterization, chen2018robust, bottmer2022sparse}. As illustrated in Figure 5 of \cite{blanchet2025distributionally}, distortion may occur either before or after sampling. Under \textit{distributional perturbation}, observations are drawn from a population distribution that differs slightly from the nominal one. Under \textit{sample contamination}, observations are drawn from the nominal distribution and subsequently altered by an adversary.

It is well known that the classical least squares estimator is severely affected by outliers. In contrast, by trimming the extreme tails of the empirical losses, the trimming approaches replace the sample average of losses with ranked summation in the objective for minimization in order to enhance the robustness of estimators; see \cite{hu2023rank} for a detailed discussion of rank based statistical learning. This trimming idea can be traced back to the least median of squares (LMS) and the least trimmed squares (LTS) proposed by \cite{rousseeuw1984least} for linear regression.  Generalizing the trimming idea, 
\cite{liu2022risk}  propose interval conditional value-at-risk (In-CVaR) for statistical learning, in which the In-CVaR risk measure is defined as the average over an interval between two specific quantiles, namely the $\alpha$-value-at-risk (VaR) and the $\beta$-VaR with $0 \leq \alpha < \beta \leq 1$. {It reduces to the LTS formulation under the $\ell_2$ loss when $\alpha=0$ and $\beta<1$.
When $\beta=1$, In-CVaR recovers conditional value-at-risk (CVaR), which has shown promising out-of-sample performance in statistical learning \citep{takeda2009robust, curi2020adaptive}. 
In-CVaR thus provides a unified control of the outlier resistance and the out-of-sample performance.} Recent works on trimmed statistical learning include classification \citep{tsyurmasto2013support, fujiwara2017dc}, regression \citep{oliveira2025trimmed},  multi-class learning \citep{hu2020learning,hu2022sum}, and partial AUC optimization \citep{yao2022large}. \cite{jiang2024distributionally} relate In-CVaR with distributionally favorable optimization, providing new insight into the outlier resistance of In-CVaR. 

There is a series of works on  computational methods for trimmed statistical estimators, including deterministic difference-of-convex algorithms \citep{fujiwara2017dc, hu2020learning},  mixed-integer programming approaches \citep{jiang2024distributionally}, proximal-gradient methods   \citep{aravkin2020trimmed} under the smoothness condition, and stochastic difference-of-convex algorithms \citep{liu2022risk} for In-CVaR based nonlinear statistical learning allowing for nondifferentiability. Despite these modeling and computational advances of In-CVaR based statistical learning, as well as the promising practical performance in mitigating the effects of outliers,  the fundamental statistical robustness properties of In-CVaR based estimators remain largely unexplored.  In particular, it is unclear how the trimming parameters $(\alpha,\beta)$ affect the robustness of the resulting statistical estimator, and how In-CVaR based estimators behave under finite-sample contamination. This paper addresses these questions for the first time by studying In-CVaR based regression under two types of distortion, respectively: distributional perturbation before sampling, and sample contamination.

Under distributional perturbation, the trimming effect of In-CVaR risk measure is firstly analyzed by \cite{cont2010robustness} and extended by \cite{embrechts2018quantile} for risk-sharing problems. An interesting phenomenon shown by the former reference is that the historical In-CVaR risk measure is qualitatively robust with respect to the  Lévy metric if and only if  both
tails of the distribution are trimmed ($\alpha \neq 0, \beta \neq 1$). In particular, the qualitative robustness is characterized via topological continuity; that is, 
an estimator is qualitatively robust if a small change in the data-generating distribution results only in a small change of the estimator.  This motivates the first part of the present paper: whether trimming tails by In-CVaR plays the necessary and sufficient roles for In-CVaR based statistical learning and particularly, under which metric, the qualitative robustness could hold.

The qualitative robustness of In-CVaR based statistical learning introduces several challenges. First, the In-CVaR objective is generally nonconvex and may yield set-valued estimators, whereas existing qualitative robustness analyses primarily concern the laws of single-valued empirical estimators under nominal and perturbed distributions \citep{cont2010robustness,guo2023statistical,zhang2024statistical}. We  develop a new qualitative robustness framework for set-valued estimators arising from nonconvex learning objectives.  Second, for practical considerations, our regression setting allows unbounded covariates and responses, and we also allow an unbounded parameter space, whereas \citet{shapiro2021lectures,zhang2024statistical,laguel2021superquantiles} assume bounded parameter spaces.  Since a loss may be neither bounded nor globally Lipschitz, we establish a new uniform convergence result of the In-CVaR objective, which largely extends the uniform convergence analysis of CVaR  by \citet{laguel2021superquantiles} requiring the loss function to be bounded and globally Lipschitz.

To the best of our knowledge, we are the first to establish qualitative robustness for  In-CVaR based statistical estimators. Our analysis illustrates several novel findings. First, we show that an In-CVaR based regression estimator is qualitatively robust under the Prokhorov metric, which metrizes the weak topology, {\bf if and only if} $\beta \neq 1$, illustrating the essentialness of trimming the tails. This indicates that the empirical risk minimization (ERM),
which is shown by  \cite{zhang2024statistical} to be qualitatively robust for kernel learning under the $\psi$-weak topology, 
would not be qualitatively robust against distributional perturbation under the Prokhorov metric.  While trimming the left tail could strengthen the out-of-sample performance, the left-tail trimming is independent of the qualitative robustness of the In-CVaR based regression estimator against the distribution perturbation. Second,  within an unbounded parameter space, we show that the parameter solution sets remain uniformly contained in a compact set under distributional perturbation under appropriate conditions. This allows the establishment of the uniform convergence result for the parameter solution sets even under the unbounded covariates, response, and parameter spaces.

As discussed in \cite{blanchet2025distributionally}, besides distributional perturbation, sample contamination is another type of distortion when observations are drawn from the nominal distribution, but subsequently altered by an adversary. Robustness under sample contamination is traditionally assessed through the breakdown point (BP) introduced by \cite{hampel1968contributions}. Related works include \cite{rousseeuw1984least} for LMS and LTS estimators under linear regression, \cite{hossjer1994rank} for In-CVaR based linear regression, and \cite{kanamori2017breakdown} for  In-CVaR based classification. As noted in \cite{loh2025theoretical}, the BP criterion provides only a coarse characterization of robustness.

 Recent work has therefore shifted toward obtaining probably approximately correct (PAC) guarantees under adversarial sample contamination, including  trimmed linear regression by \citet{bhatia2015robust,bhatia2017consistent} under response contamination, and   \citet{pensia2025robust} under contamination of both covariates and responses.  The analysis of \citet{pensia2025robust}  relies on the strong stability induced by preliminary covariate filtering, in order to attain the uniform control of empirical means and covariances over sufficiently large subsets of samples. However, such stability arguments under covariate filtering are closely tied to the linear regression structure, which thus limits its generalization to nonlinear regression settings. \citet{oliveira2025trimmed} present a novel idea of the trimmed minimaximization estimator which leads to a better  PAC guarantee, whereas the computation of such a trimmed minimaximization estimator remains highly challenging even for linear regression.

So our second major part of the present paper is motivated to address whether the  In-CVaR based nonlinear regression estimator can achieve a good PAC bound under sample contamination without resorting to preliminary covariate filtering or complex minimaximizing optimization. {To develop a tight PAC bound, we devise a new decomposition technique for quantifying the gap between the empirical In-CVaR objectives at a candidate parameter and the true parameter. 
This new decomposition technique, together with a clipping strategy, enables us to directly control the gap by the empirical-process deviation of the  excess losses $
d_{\BFtheta}(\boldsymbol{X},Y)
\triangleq
\bigl(Y-f(\boldsymbol{X},\BFtheta)\bigr)^2
-
\bigl(Y-f(\boldsymbol{X},\BFtheta^*)\bigr)^2
$.   We show that, under mild assumptions  weaker than those in the previous literature,  the empirical In-CVaR objective at the true parameter is smaller  than at any parameter outside a prescribed neighborhood with  high probability and uniformly over adversarial contamination. This shows that a good PAC bound holds for the empirical In-CVaR estimator under nonlinear regression, without preliminary covariate filtering.}
 We summarize our main contributions  as follows.

1. \textit{Robustness under distributional perturbation.} 
 We analyze  \textit{qualitative robustness} against distributional perturbation for In-CVaR based regression in terms of the Prokhorov metric. 
 We show that for a class of regression models, the In-CVaR based estimator remains uniformly contained in compact regions under distributional perturbation and  is qualitatively robust with respect to the  Prokhorov metric if $\beta \neq 1$ (Corollary \ref{thm:sufficient_qrao}). 
 We further show that the In-CVaR based estimator  is consistent (Theorem  \ref{thm:consistency}) given that the empirical estimators lie within a compact set. 
 Based on the consistency property, we establish that the In-CVaR based estimator is qualitatively robust in terms of the Prokhorov metric only if the largest portion of losses is trimmed (Theorem \ref{thm: qrao_necessary}). These findings reveal that trimming the tails is essential for the robustness of statistical learning estimators under distributional perturbation.

2. \textit{Robustness under sample contamination.} 
We establish a \textit{PAC} error bound for the In-CVaR based nonlinear regression estimator under adversarial contamination. Our analysis  reveals that a good PAC bound holds for In-CVaR based estimators without the need of  preliminary covariate filtering customary in the  literature. Our analysis  covers a wide range of  nonlinear models including polynomial, logarithmic, exponential and power regression.  Specializing to linear regression, without the preliminary covariate filtering,  our PAC bound has the same dependence on the sample size, contamination level, and confidence parameter as in \cite{pensia2025robust}, except for a less favorable dependence on the covariate dimension. This gap can be mitigated by the postprocessing procedure as in \cite{pensia2025robust} yielding the same near-optimal error bound. {Supporting our theory of the PAC bound, the numerical experiments show that In-CVaR based estimators exhibit strong robustness against sample contamination, and covariate filtering offers only limited benefits.}

The remainder of this paper is organized as follows.
Section \ref{sec:qrao}  investigates the qualitative robustness of the In-CVaR based estimator, establishing stability and consistency properties, and providing both sufficient and necessary conditions for qualitative robustness under distributional perturbation.
Section \ref{sec:PAC} studies the PAC bound for empirical In-CVaR based regression.
Section \ref{sec:exp} conducts numerical experiments on contaminated synthetic data and real data to illustrate and validate our theoretical
findings. 
Section \ref{sec:conclusion} concludes the paper with discussions and future research directions.

\section{Qualitative Statistical Robustness of In-CVaR Based Estimators}
\label{sec:qrao}

In this section, we investigate the qualitative statistical robustness of the In-CVaR based estimator under distributional perturbation. Echoing the insightful robustness property of the In-CVaR risk measure by  \cite{cont2010robustness}, in this section we show that the In-CVaR based statistical estimator possesses qualitative robustness if and only if the extreme losses are trimmed. Since In-CVaR based statistical estimators may not be unique under the generally nonconvex and nondifferentiable In-CVaR objective, it necessitates the new analysis and insights on set-valued estimators for characterizing the trimming effect in terms of the qualitative robustness, and moreover for being applicable to various nonlinear regression models with the minimal assumptions on the loss and regression functions.  Section \ref{sec:qualitative}  introduces the notation and mathematical formulation, followed by the mathematical definition of qualitative robustness. We show the qualitative robustness holds when $\beta \neq 1$  in Section \ref{sec:sufficient_qrao}. In Section \ref{sec:consistency}, we establish the consistency of In-CVaR based estimators, which provides a useful tool that is applied in Section \ref{sec:necessary_qrao} for establishing  that $\beta \neq 1$ is necessary for the qualitative robustness of the In-CVaR based estimators.

\subsection{Qualitative robustness concept}
\label{sec:qualitative}

We begin by introducing the notation  used throughout the paper. 
Let $\mathbb{N}_+$ denote the set of positive integers, $\mathbb{R}^p$ the $p$-dimensional Euclidean space, $\mathbb{R}^p_+$ the nonnegative orthant, and $\mathbb{R}^{p\times q}$ the space of real-valued $p \times q$ matrices. For $n\in \mathbb{N}_+$, let $[n] \triangleq \{1, \cdots, n\}$.  For a subscripted vector $\BFtheta_\bullet \in \mathbb{R}^p$ and $i \in [p]$, let $\BFtheta_{\bullet,i}$ denote its $i$-th component. For $i \in [p]$, let $\BFe_i \in \mathbb{R}^p$ denote the standard basis vector with 1 in the $i$-th coordinate and 0 elsewhere. For any vector $\boldsymbol{x} = (x_1, \cdots , x_p) \in \mathbb{R}^p$, let $\Vert \boldsymbol{x} \Vert$ denote its Euclidean norm, and the open ball centered at $\boldsymbol{x}$ with radius $\varepsilon > 0$ is given by $\mathcal{B}(\boldsymbol{x}, \varepsilon) \triangleq \{\BFy \in \mathbb{R}^p \mid  \Vert \boldsymbol{x} - \BFy \Vert <\varepsilon\}$. For $p \in \mathbb{N}_+$ and $\alpha \geq 0$, let $\alpha\mathcal{B}^p \triangleq \{\boldsymbol{x} \in \mathbb{R}^p: \|\boldsymbol{x}\| \leq \alpha\} $ denote the $\alpha$ Euclidean ball in $\mathbb{R}^p$. Given a set $A \subseteq \mathbb{R}^p$, we define its $\varepsilon$-neighborhood by $\mathcal{B}(A, \varepsilon) \triangleq \bigcup_{\boldsymbol{x} \in A} \mathcal{B}(\boldsymbol{x}, \varepsilon),$
the distance from a point $\boldsymbol{x} \in \mathbb{R}^p$ to $A$ by $\operatorname{dist}(\boldsymbol{x}, A) \triangleq \inf_{\BFy \in A} \Vert \boldsymbol{x} - \BFy \Vert$ and the diameter of $A$ by $\operatorname{diam}(A) \triangleq \sup_{\boldsymbol{x}, \BFy\in A}\Vert \boldsymbol{x} - \BFy\Vert$. We adopt the deviation $\mathbb{D}(A, B) \triangleq \sup_{\boldsymbol{x} \in A}\operatorname{dist}(\boldsymbol{x}, B)$ to quantify the distance between the two sets $A$ and $B$, following \cite{shapiro2021lectures}. 
 For a finite set $A$, we use $| A |$ to denote its cardinality.
 For a matrix $\boldsymbol{A} \in \mathbb{R}^{p\times p}$, let $\lambda_{\min}(\boldsymbol{A})$ denote its smallest eigenvalue.
All analysis is conducted on a Polish probability space $(\Omega, \mathcal{F}, \mathbb{P})$. For a probability distribution $D$, let $\mathbb{E}_D$ denote expectation with respect to $D$. For $k > 0$ and a measurable function $g$,  let
$
\| g\|_{L_k(D)} \triangleq \left(\mathbb{E}_D\left[|g|^k\right]\right)^{1/k}.$We use standard asymptotic notation in the regimes $n\to\infty$ and $\varepsilon\downarrow0$. For nonnegative functions $f,g:D\to[0,\infty)$, where $D=\mathbb{N}_+$ or $D=[0,1]$, we write $f=O(g)$ if $f\leq Cg$ for some $C>0$ and all sufficiently large $n$ or sufficiently small positive $\varepsilon$, respectively. When $D=\mathbb{N}_+$, we write $f=o(g)$ if $g(n)>0$ for all sufficiently large $n$ and $\lim_{n\to\infty}f(n)/g(n)=0$. Finally, $f=\Omega(g)$ if $g=O(f)$, and $\lesssim$ and $\gtrsim$ denote inequalities up to positive multiplicative constants.

Below we introduce the In-CVaR based estimator. Let $\mathcal{X} \subseteq \mathbb{R}^p$ and $\mathcal{Y} \subseteq \mathbb{R}$ denote covariate and response spaces, respectively, and $\Theta \subseteq \mathbb{R}^q$ denote the parameter space for a regression model $f(\bullet, \bullet): \mathcal{X} \times \Theta \to \mathcal{Y}$. Let $\mathcal{L} : [0, \infty) \rightarrow [0, \infty) $ be a loss function.
Given a parameter $\BFtheta \in \Theta$, the loss for an input-output pair $(\boldsymbol{x}, y)$ is measured by $\mathcal{L}(\abs{f(\boldsymbol{x}, \BFtheta) - y})$. 
Given a distribution $D$ supported on $\mathcal{X} \times \mathcal{Y}$, let $(\boldsymbol{X}_D, Y_D)$ denote a random vector distributed according to $D$. Then, with a scalar $\alpha \in [0,1]$, the $\alpha$-VaR of the loss is defined as 
\[
\operatorname{VaR}_\alpha\left(\mathcal{L}\left(\,\abs{f(\boldsymbol{X}_{D}, \BFtheta) - Y_{D}}\right)\right) 
\triangleq  \inf\left\{t: \mathbb{P}\left(\mathcal{L}\left(\,\abs{f(\boldsymbol{X}_{D}, \BFtheta) - Y_{D}}\right) \leq t\right)\geq \alpha\right\}.
\]
With $\alpha \in [0,1)$, the $\alpha$-CVaR of the loss is given by
\[
 \operatorname{CVaR}_\alpha\left(\mathcal{L}\left(\,\abs{f(\boldsymbol{X}_{D}, \BFtheta) - Y_{D}}\right)\right) 
\triangleq \frac{1}{1-\alpha}\int_\alpha^1 \operatorname{VaR}_\gamma\left(\mathcal{L}\left(\,\abs{f(\boldsymbol{X}_{D}, \BFtheta) - Y_{D}}\right)\right)\rm{d}\gamma.
\]
For $0 \leq \alpha < \beta \leq 1$, the $(\alpha,\beta)$-In-CVaR of the loss is defined as
\[
\operatorname{In-CVaR}_\alpha^\beta\left(\mathcal{L}\left(\,\abs{f(\boldsymbol{X}_{D}, \BFtheta) - Y_{D}}\right)\right) 
\triangleq  \frac{1}{\beta - \alpha}\int_\alpha^\beta\operatorname{VaR}_\gamma\left(\mathcal{L}\left(\,\abs{f(\boldsymbol{X}_{D}, \BFtheta) - Y_{D}}\right)\right)\rm{d}\gamma.  
\]
The In-CVaR based regression problem seeks $\BFtheta \in \Theta$  that minimizes the $(\alpha, \beta)$-In-CVaR of the losses under  a distribution $D$, i.e.,
\begin{equation}
    \label{eq:problem_statement}
    \min_{\BFtheta \in \Theta} \operatorname{In-CVaR}_\alpha^\beta\left(\mathcal{L}\left(\,\abs{f(\boldsymbol{X}_{D}, \BFtheta) - Y_{D}}\right)\right),
\end{equation}
the optimal solution set of which is denoted as ${\mathcal{S}}_\alpha^\beta(D,f)$.
Problem \eqref{eq:problem_statement} unifies several classical regression formulations. It reduces to expectation minimization when $\alpha = 0$ and $\beta = 1$, recovers the LTS formulation \citep{rousseeuw1984least, aravkin2020trimmed} when $\alpha = 0$ and $\beta < 1$, and yields the LMS formulation \citep{rousseeuw1984least} in the limit as $\beta =1/2$ and $\alpha \uparrow 1/2$.

With a loss function $\mathcal{L}$ and a regression function $f$, define the class of admissible distributions by $
\mathcal{D} \triangleq \{D: \mathcal{L}\left(\,\abs{f(\boldsymbol{X}_{D}, \BFtheta) - Y_{D}}\right)\, \text{is integrable for all }\, \BFtheta \in \Theta\}.$
With the nominal distribution $D_0\in \mathcal{D}$ and a perturbed distribution $D \in\mathcal{D}$ around $D_0$, 
{suppose that we obtain independent observations $ \{(\boldsymbol{x}_{D}^{(i)},y_{D}^{(i)})\}_{i=1}^n $ identically distributed from $D$,}
with empirical distribution denoted by $D^{(n)}$.
Then the set of In-CVaR based estimators under $D^{(n)}$ and the regression function $f$ is
\[
    \hat{\mathcal{S}}_\alpha^\beta(D^{(n)},f) = \argmin_{\BFtheta \in \Theta} \, \operatorname{In-CVaR}_{\alpha}^{\beta} \left( \mathcal{L}\left(\, \abs{f(\boldsymbol{X}_{D^{(n)}}, \BFtheta) - Y_{D^{(n)}}} \right) \right).
\]

To address whether the observed empirical estimator $ \hat{\mathcal{S}}_\alpha^\beta(D^{(n)}, f)$ reliably approximates  the true estimator ${\mathcal{S}}_\alpha^\beta(D_0 , f)$ with sufficiently large samples, we follow the concept of qualitative statistical robustness, which concerns topological continuity, introduced in \cite{hampel1971general}, \citet[p.41]{huber2009robust}.  The definition of qualitative robustness is presented below, which concerns the  convergence of the deviation of the empirical estimator w.p.1  as sample size increases and the perturbed distribution $D$ is close to the nominal distribution $D_0$.  

\begin{definition}[Qualitative Robustness of Statistical Estimator]
\label{def_qualitative_robust}

  Let ${{T}}(D, f)$ and $\hat{{T}}(D^{(n)}, f)$ represent the statistical functional and its sample estimator under the distribution $D$ and regression model $f$.  For the nominal distribution $D_0 \in \mathcal{D}$, the estimator $\hat{{T}}(\bullet, f)
$ is said to be qualitatively robust  at $D_0$ with respect to the metric $d$ if, for every $\varepsilon > 0$, there exists $\delta > 0$ such that, for all perturbed distributions $D \in \mathcal{D}$ with $d(D, D_0) \leq \delta$,
\[
\mathbb{D}\big( \hat{{T}}(D^{(n)} , f), {{T}}(D_0 , f)\big) \leq \varepsilon \, \, \text{with probability 1 (w.p.1) as} \,\, n \rightarrow \infty.
\] 
\end{definition}

Definition \ref{def_qualitative_robust}  applies to possibly set-valued estimators without specifying a selection rule. To quantify the discrepancy between two distributions, it is common to choose the metric $d$ in order to induce the weak topology  \citep{hampel1971general, huber2009robust, cont2010robustness} or the $\psi$-weak topology \citep{kratschmer2012qualitative, kratschmer2014comparative, kratschmer2017domains, guo2023statistical, zhang2024statistical} on probability measures.  We adopt the Prokhorov metric, an extension of the Lévy metric to the multivariate space that induces the weak topology. For the sake of convenience, let $\mathcal{Z} \triangleq \mathcal{X} \times \mathcal{Y}$ and $\boldsymbol{Z} \triangleq (\boldsymbol{X},Y)$.

\begin{definition}[Prokhorov metric]
    
    \label{def:Prok} 
     The Prokhorov metric between two distributions $D_1, D_2 \in \mathcal{D}$ is defined as
    \[
    d_{\text{Prok}}(D_1, D_2) \triangleq \inf_{\varepsilon > 0} \left\{\varepsilon  : \mathbb{P}( \boldsymbol{Z}_{D_1} \in A) \leq \mathbb{P} ( \boldsymbol{Z}_{D_2} \in \mathcal{B}(A, \varepsilon))+\varepsilon  \text { for all } A \in \mathfrak{B} \right\},
    \]
    where $\mathfrak{B}$ denotes the Borel $\sigma$-algebra on $\mathcal{Z}$. The $\delta$-Prokhorov ball centered at $D_0 \in \mathcal{D}$ is defined as
    \[
    \mathcal{B}_{\text{Prok}}(D_0, \delta) \triangleq \{D \in \mathcal D: d_{\text{Prok}}(D, D_0) < \delta\}.
    \]
\end{definition}

It is worth mentioning that  \citet[Definition 3.1]{zhang2024statistical} characterize the robustness via the deviation between the laws of two empirical estimators under $D_0$ and $D$ in terms of the Prokhorov metric with the sample size being sufficiently large. {Definition \ref{def_qualitative_robust} is different from  theirs; one major difference is that  they allow the two empirical estimators to be far apart with an arbitrarily small probability, whereas we adopt another type in which the empirical estimator under $D$ converges to the true estimator w.p.1 as $n \to \infty$. Our definition strengthens the almost sure requirement by allowing $n$ to increase asymptotically to infinity. }

We make the following blanket assumptions for the rest of this section.
\begin{enumerate}[
    label=\textbf{Assumption~(A\arabic*):}~,
    ref=(A\arabic*),
    start=1,
    leftmargin=0pt,
    labelindent=0pt,
    itemindent=0pt,
    labelwidth=0pt,
    labelsep=0pt,
    listparindent=0pt,
    align=left
    ]
    \item 
    \label{ass:a1}
    The loss function $\mathcal L:\mathbb R_+\to\mathbb R_+$ is continuous and
    non-decreasing, satisfying $\mathcal L(0)=0$ and
    $\lim_{t\to\infty}\mathcal L(t)=\infty$.
    The regression function $f$ is continuous on $\mathcal X\times\Theta$.
    \item 
\label{ass:a2}
    The sample space $\mathcal{Z}$ and the parameter space $\Theta$ are closed. For the nominal distribution $D_0$ and the regression function $f$, given $0 \leq \alpha < \beta \leq 1$, the true solution set ${\mathcal{S}}_\alpha^\beta(D_0,f)$ is nonempty and contained in a compact subset of $\Theta$.
\item 
\label{ass:a3}
    For every distribution $D \in \mathcal{D}$ and every compact subset $C \subseteq \Theta$, there exists a measurable function  $\psi: \mathcal{X} \times \mathcal{Y} \to \mathbb{R}_+$ such that $\psi(\boldsymbol{X}_D, Y_D)$ is integrable and
    \[
    \mathcal{L}(\,\abs{f(\boldsymbol{X}_D, \BFtheta) - Y_D}) \leq \psi(\boldsymbol{X}_D, Y_D), \quad \forall \BFtheta \in C.
    \]
\end{enumerate}

Assumptions \ref{ass:a1}-\ref{ass:a3} are standard assumptions for qualitative robustness analysis under perturbation. 
Conditions similar to Assumptions \ref{ass:a1} and \ref{ass:a3} are also adopted in \cite{zhang2024statistical} for kernel learning estimators.

In the rest of this section, the main goal is to derive the insightful result:

\textit{Under appropriate assumptions, the In-CVaR based estimator is qualitatively robust if and only if $\beta \neq 1$, equivalently, if and only if the largest losses are trimmed.}

Before presenting the detailed analysis on the qualitative robustness, we sketch the overall idea here. 
To prove qualitative robustness when $\beta \neq 1$, we first show that the population solution sets $\mathcal{S}_\alpha^\beta(D,f)$ are contained in a common compact set for all $D$ near $D_0$ under appropriate conditions. Then restricting to that compact set, we establish the uniform convergence of the population In-CVaR objectives under distributional perturbation in Theorem~\ref{thm:stability}, and then the stability of the In-CVaR based estimator at $D_0$ follows when $\beta\neq1$. Based on the stability property,  by utilizing the almost-sure weak convergence of the empirical distributions $D^{(n)}$ to $D$, we thus obtain that the In-CVaR based estimator is qualitatively robust with respect to the Prokhorov metric if $\beta \neq 1$, presented in Corollary \ref{thm:sufficient_qrao}.
Notice that our analysis does not require the boundedness of the  parameter space, which is in contrast to the boundedness assumption in   \citet{zhang2024statistical}.

Regarding the reverse direction, the analysis contains two parts. In Section \ref{sec:consistency}, we analyze the uniform convergence of the empirical In-CVaR objective under a weaker condition than \citet[Theorem 1]{laguel2021superquantiles}. Specifically, that reference requires that the loss function is almost surely bounded and globally Lipschitz continuous, which are easily violated in ordinary least squares regression with unbounded $(\boldsymbol{X}_D, Y_D)$ and the $\ell_2$ loss. Therefore, we develop a new uniform convergence analysis for the In-CVaR objective within a compact parameter space. Subsequently, we establish the consistency result of In-CVaR in Theorem \ref{thm:consistency}, similarly to  the analysis in  \citet{shapiro2021lectures}, but by accommodating a much more general class of regression models. Then in Section \ref{sec:necessary_qrao}, based on the consistency property, we show a key result in Theorem~\ref{thm: qrao_necessary} that for qualitatively robust estimators, the tail must be trimmed ($\beta \neq 1$) via a constructive contradiction.

This claim leads to two interesting observations: first, the ERM regression estimator, which is shown to be qualitatively robust for kernel learning under the $\psi$-weak topology in \cite{zhang2024statistical},
 is not robust under the Prokhorov metric; second, while it is well known that CVaR risk measure by trimming the left tail could strengthen the out-of-sample performance for robust statistical learning \citep{takeda2009robust, curi2020adaptive}, { the left-tail trimming is independent of the qualitative robustness of the In-CVaR based regression estimator against the distribution perturbation}.

\subsection{Qualitative robustness under tail trimming  ($\beta \neq 1$)}
\label{sec:sufficient_qrao}

In this subsection, we show the qualitative robustness of the In-CVaR based estimator with $\beta \neq 1$. Notice we do not assume any compactness of the parameter space. So we first show that, for several typical regression models including linear, generalized linear and one-dimensional piecewise affine regression,  the 
solution sets ${\mathcal{S}}_\alpha^\beta(D, f)$ remain within a compact set for any distribution $D$ in a sufficiently small neighborhood of $D_0$. For $p \in \mathbb{N}_+$, let $\mathcal{H}_{p}$ denote the set of all hyperplanes passing through the origin  in $\mathbb{R}^{p}$. All proofs are deferred to Appendix \ref{app:compact}.

\begin{proposition}

\label{prop:compactness_linear}
Assume that Assumptions \ref{ass:a1} and \ref{ass:a2} hold and $\beta \neq 1$. Let $\mathcal{X}\subseteq\mathbb{R}^p$ and
$\Theta\subseteq\mathbb{R}^p$, and consider the linear regression model
$
f(\boldsymbol{x},\BFtheta)=\BFtheta^\top\boldsymbol{x}.
$
For nominal distribution $D_0$, suppose that
\begin{equation}
\label{eq:linear_hyperplane}
\lim _{\tau \rightarrow 0^{+}} \sup _{H \in \mathcal{H}_{p}} \mathbb{P}\left(\operatorname{dist}\left(\boldsymbol{X}_{D_0}, H\right)<\tau\right) < \beta.
\end{equation}
Then there exist $\bar \delta > 0$ and a compact subset $C \subseteq \Theta$ such that $\emptyset \neq  {\mathcal{S}}_\alpha^\beta(D,f) \subseteq C$ for all {\rm $D \in  \mathcal{B}_{\text{Prok}}(D_0,\bar \delta)$}.
\end{proposition}

\begin{remark}
The probabilistic condition \eqref{eq:linear_hyperplane} requires that $\boldsymbol{X}_{D_0}$ does not concentrate near any origin-passing hyperplane in $\mathbb{R}^{p}$.  One example is the mixture distribution  satisfying
\[
D_0 = (\beta - \eta ) D_{0,1} + \left(1 - \beta + \eta\right) D_{0,2}, \quad 0 < \eta < \beta,
\]
where $D_{0,1}$ and  $D_{0,2}$ are two distributions with the latter one having a bounded density; see the verification in Proposition \ref{lem:probabilistic_requirement}, Appendix \ref{app:compact}. 
\end{remark}

\begin{remark}
Proposition \ref{prop:compactness_linear} also applies to regression models that are linear in their parameters, including polynomial and logarithmic regression. Specifically, consider the regression model $f(\boldsymbol{x},\BFtheta)=\BFtheta^\top \phi(\boldsymbol{x}),$
where $\BFtheta \in \Theta \subseteq \mathbb{R}^q$ and $\phi:\mathcal{X}\to\mathbb{R}^{q}$ is continuous. Then condition \eqref{eq:linear_hyperplane} becomes
\[
\lim_{\tau\rightarrow 0^{+}}
\sup_{H\in\mathcal{H}_{q}}
\mathbb{P}\left(
\operatorname{dist}\left(
\phi(\boldsymbol{X}_{D_0}),H
\right)<\tau
\right)
<\beta.
\]
\end{remark}

We next consider models of the form $f(\boldsymbol{x}, \BFtheta) = g(\BFtheta^\top \boldsymbol{x})$, where $g$ is continuous and nondecreasing. The case of a continuous nonincreasing $g$ can be handled analogously.

\begin{proposition}
    
    \label{prop:compact_glm}
    Assume that Assumptions \ref{ass:a1} and \ref{ass:a2} hold and $\beta \neq 1$. Let $\mathcal{X}\subseteq\mathbb{R}^p$ and
$\Theta\subseteq\mathbb{R}^p$, and let $D_0$ be the nominal distribution. Consider $f(\boldsymbol{x}, \BFtheta) = g(\BFtheta^\top \boldsymbol{x})$, where $g: \mathbb{R} \rightarrow \mathcal{Y}$  is continuous and nondecreasing. Define 
\[ m \triangleq \lim_{t \rightarrow -\infty } g(t), \quad M \triangleq \lim_{t \rightarrow  \infty} g(t) \quad \text{and} \quad 
    \eta \triangleq \lim _{\tau \rightarrow 0^{+}} \sup _{H \in \mathcal{H}_{p}} \mathbb{P}\left(\operatorname{dist}\left(\boldsymbol{X}_{D_0}, H\right)<\tau\right),\]
    where  $m$ and $M$ may take extended real values.
    Suppose $\eta < \beta$ and 
    {
    \rm
    \begin{equation}
    \label{eq:compo}
    \begin{aligned}
    \begin{array}{ll}
     & \inf_{\BFtheta \in \Theta} \operatorname{In-CVaR}_{\alpha}^{\beta} \left(\mathcal{L}\left(\,\abs{g\left(\BFtheta^\top \boldsymbol{X}_{D_0}\right) -Y_{D_0}}\right)\right)  
     \\
     & \qquad < \displaystyle\min\left\{1, \frac{\beta - \eta}{\beta - \alpha}\right\}\operatorname{In-CVaR}_{\max\{\alpha - \eta, 0\}}^{\beta  - \eta }\left(\min\left\{\mathcal{L}\left(\,\abs{m -Y_{D_0}}\right), \mathcal{L}\left(\,\abs{M -Y_{D_0}}\right)\right\}\right).
     \end{array} 
    \end{aligned}
    \end{equation}
    }
    Then there exist $\bar \delta > 0$ and a compact subset $C \subseteq \Theta$ such that $\emptyset \neq  {\mathcal{S}}_\alpha^\beta(D,f) \subseteq C$ for all {\rm $D \in  \mathcal{B}_{\text{Prok}}(D_0,\bar \delta)$}.
\end{proposition}

\begin{remark}
Proposition \ref{prop:compact_glm} extends to regression models of the form $f(\boldsymbol{x}, \BFtheta) = g(\BFtheta^\top \phi(\boldsymbol{x}))$, where $\BFtheta \in \Theta \subseteq \mathbb{R}^q$, $\phi: \mathcal{X} \to \mathbb{R}^q$ is continuous and $g: \mathbb{R} \rightarrow \mathcal{Y}$ satisfies the conditions of Proposition \ref{prop:compact_glm}. This covers exponential regression and power regression.
\end{remark}

We admit that inequality \eqref{eq:compo} may not be easily verified, but when the nominal response follows a well-specified regression model with additive noise $Y_{D_0}=f(\boldsymbol{X}_{D_0},\BFtheta^*)+\xi_{D_0}$
for some $\BFtheta^*\in\Theta$, a sufficient condition for \eqref{eq:compo} that is amenable to verification is 
 \begin{equation}
    \label{eq:compo_2}
    \begin{aligned}
    \begin{array}{ll}
     & \operatorname{In-CVaR}_{\alpha}^{\beta} \left(\mathcal{L}\left(\,\abs{\xi_{D_0}}\right)\right)  \\
     & \qquad \displaystyle <
     \min\left\{1, \frac{\beta-\eta}{\beta - \alpha}\right\}\operatorname{In-CVaR}_{\max\{\alpha - \eta,0 \}}^{\beta -\eta }\left(\min\left\{\mathcal{L}\left(\,\abs{m -Y_{D_0}}\right), \mathcal{L}\left(\,\abs{M -Y_{D_0}}\right)\right\}\right).
     \end{array} 
    \end{aligned}
    \end{equation}
When $\eta=0$, condition~\eqref{eq:compo_2} further simplifies to
\[
\operatorname{In-CVaR}_{\alpha}^{\beta}
\left(\mathcal{L}(\abs{\xi_{D_0}})\right)
<
\operatorname{In-CVaR}_{\alpha}^{\beta}
\left(
\min\left\{
\mathcal{L}(\abs{m-Y_{D_0}}),
\mathcal{L}(\abs{M-Y_{D_0}})
\right\}
\right).
\]

Next, we consider piecewise affine regression with a one-dimensional covariate.

\begin{proposition}
\label{prop:compact_pa}
Assume that Assumptions \ref{ass:a1} and \ref{ass:a2} hold and $\beta \neq 1$. Let $I, J \geq 1$, and consider 
\[
f(x,\BFtheta)
=
\max_{i\in[I]}\{a_i x+c_i\}
-
\max_{j\in[J]\cup\{0\}}\{b_j x+d_j\},
\qquad (b_0,d_0)=(0,0),
\]
where
$\BFtheta=((a_i,c_i)_{i=1}^I,(b_j,d_j)_{j=1}^J)
\in\Theta\subseteq\mathbb{R}^{2(I+J)}$. For nominal distribution $D_0\in\mathcal D$, suppose that the following conditions
hold:
\begin{enumerate}[label=(\alph*)]
    \item
    $\displaystyle
    \lim_{h\downarrow0}\sup_{x\in\mathbb R}
    \mathbb P\bigl(X_{D_0}\in[x-h,x+h]\bigr)=0$;

    \item
    there exists $\BFtheta^*\in\Theta$ such that
    $Y_{D_0}=f(X_{D_0},\BFtheta^*)+\xi_{D_0}$.
    Moreover, $f(\cdot,\BFtheta^*)$ has exactly $I+J$ nondegenerate
    affine pieces: there exist
    $-\infty=t_0<t_1<\cdots<t_{I+J-1}<t_{I+J}=\infty$
    such that, for $x\in(t_{k-1},t_k)$,
    $
    f(x,\BFtheta^*)
    =(a_{i(k)}^*-b_{j(k)}^*)x
     +(c_{i(k)}^*-d_{j(k)}^*),
    $
    where $i(k)\in[I]$ and $j(k)\in[J]\cup\{0\}$, and $(a_{i(k)}^* - b_{j(k)}^*, c_{i(k)}^*-d_{j(k)}^*) \neq (a_{i(k+1)}^* - b_{j(k+1)}^*, c_{i(k+1)}^*-d_{j(k+1)}^*)$ for all $k \in [I+J-1]$;

    \item
    $\displaystyle
    \beta>
    1-\min_{k\in[I+J]}
    \mathbb P\bigl(X_{D_0}\in(t_{k-1},t_k)\bigr)$;

    \item
    $\mathcal L(t)$ is strictly increasing.
\end{enumerate}
For every
$
\nu\in\left(
\max\left\{
\beta,\,
2-\beta-\min_{k\in[I+J]}
\mathbb P\bigl(X_{D_0}\in(t_{k-1},t_k)\bigr)
\right\},\,1
\right],
$
there exists $\varepsilon>0$ such that, if
\[
\mathbb P\bigl(\mathcal L(|\xi_{D_0}|)\leq\varepsilon\bigr)
\geq \nu,
\]
then there exist $\bar\delta>0$ and a compact set $C\subseteq\Theta$
such that $\emptyset \neq  {\mathcal{S}}_\alpha^\beta(D,f) \subseteq C$ for all
$D\in\mathcal B_{\mathrm{Prok}}(D_0,\bar\delta)$.
\end{proposition}

We also present Example \ref{ex:pa} in Appendix \ref{app:compact} showing that ${\mathcal{S}}_\alpha^\beta(D_0,f)$ may not be contained in a compact set (i.e., Assumption \ref{ass:a2} fails) if one of conditions (a)–(c) is omitted.
Based on the compactness conditions established above, the following theorem shows the stability of In-CVaR based estimators  under distributional perturbation  when $\beta \neq 1$.

\begin{theorem}[Stability of In-CVaR based estimators]

\label{thm:stability}
    Assume that Assumptions \ref{ass:a1} and \ref{ass:a2} hold and $\beta \neq 1$. Suppose that there exist a compact set $C\subseteq\Theta$ and a constant $\bar{\delta}>0$ such that, for every $D\in\mathcal{B}_{\mathrm{Prok}}(D_0,\bar{\delta})$, $\emptyset \neq  {\mathcal{S}}_\alpha^\beta(D,f) \subseteq C$. 
    Then, as {\rm $\displaystyle d_{\text{Prok}} (D, D_0) \rightarrow 0$}, 
    {
    \rm
    \[
    \displaystyle  \operatorname{In-CVaR}_\alpha^\beta \left(\mathcal{L}\left(\,\abs{f(\boldsymbol{X}_{D}, \BFtheta) - Y_{D}}\right)\right) \to \operatorname{In-CVaR}_\alpha^\beta \left(\mathcal{L}\left(\,\abs{f(\boldsymbol{X}_{D_0}, \BFtheta) - Y_{D_0}}\right)\right)
    \]
    }
    uniformly on $C$ and $\mathbb{D}\big({\mathcal{S}}_\alpha^\beta(D,f), {\mathcal{S}}_\alpha^\beta(D_0,f)\big) \rightarrow 0.$

\end{theorem}

\begin{proof}{Proof of Theorem \ref{thm:stability}.}
    Consider an arbitrary sequence $\{D_k\}_{k=1}^\infty$ such that $d_{\operatorname{Prok}}(D_k,D_0)\to 0$. By  \citet[Theorem 2.14]{huber2009robust} and Skorokhod's representation theorem, we can construct $(\boldsymbol{X}_{D_k}, Y_{D_k})\sim D_k$ and $ (\boldsymbol{X}_{D_0}, Y_{D_0})\sim D_0$ such that $(\boldsymbol{X}_{D_k}, Y_{D_k}) \to (\boldsymbol{X}_{D_0}, Y_{D_0})$ almost surely. Define 
    \[
    R_k \triangleq \sup_{\BFtheta \in C} \left|\mathcal{L}\left(\,\abs{f(\boldsymbol{X}_{D_k}, \BFtheta) - Y_{D_k}}\right) - \mathcal{L}\left(\,\abs{f(\boldsymbol{X}_{D_0}, \BFtheta) - Y_{D_0}}\right)\right|.
    \]
    Assumption \ref{ass:a1} and compactness of $C$ imply that $R_k \to 0$ almost surely. Choose $\kappa > 0$ such that $\beta + 2\kappa < 1$. Under the nominal distribution $D_0$, there exists a compact set $K \subseteq \mathcal{Z}$ with 
    $
    \mathbb{P}\left((\boldsymbol{X}_{D_0}, Y_{D_0}) \in K \right) > \beta + 2\kappa.
    $
    Continuity of the loss on $K \times C$ ensures that
    \[
    M \triangleq \max_{(x,y) \in K, \BFtheta \in C}\mathcal{L}\left(\,\abs{f(x, \BFtheta) - y}\right) + 1 < \infty.
    \]
    For all sufficiently large $k$, 
    $\mathbb{P}\left(R_k > 1\right) < \kappa$. Thus, for every $\BFtheta \in C$,
    \begin{equation}
    \label{eq:sta_probability}\mathbb{P}\left(\mathcal{L}\left(\,\abs{f(\boldsymbol{X}_{D_k}, \BFtheta) - Y_{D_k}}\right) \leq M\right) \geq \mathbb{P}\left((\boldsymbol{X}_{D_0}, Y_{D_0}) \in K, R_k \leq 1\right)> \beta
    \end{equation}
    and $\mathbb{P}\left(\mathcal{L}\left(\,\abs{f(\boldsymbol{X}_{D_0}, \BFtheta) - Y_{D_0}}\right) \leq M\right) > \beta$.
    For $\BFtheta \in \Theta$ and $D \in \mathcal{D}$, according to \citet[Theorem 1]{rockafellar2000optimization},
    we have
    \begin{equation}  
    \label{eq:sta_rockafellar}
    \begin{aligned}
    & \operatorname{In-CVaR}_\alpha^\beta \left(\mathcal{L}\left(\,\abs{f(\boldsymbol{X}_{D}, \BFtheta) - Y_{D}}\right)\right) \\
    & \qquad = \frac{1}{\beta - \alpha}\left(\sup_{t \geq 0} \{\beta t - \mathbb{E}\left[\max\left\{t - \mathcal{L}\left(\,\abs{f(\boldsymbol{X}_{D}, \BFtheta) - Y_{D}}\right), 0\right\}\right]\} \right. \\
    & \left. \qquad \quad - \sup_{t \geq 0}\{\alpha t - \mathbb{E}\left[\max\left\{t - \mathcal{L}\left(\,\abs{f(\boldsymbol{X}_{D}, \BFtheta) - Y_{D}}\right), 0\right\}\right]\}\right).
    \end{aligned}
    \end{equation}
    By \eqref{eq:sta_probability}, for $D = D_0$ and $D= D_k$,  both suprema in \eqref{eq:sta_rockafellar} can be restricted to $[0,M]$, uniformly over $\BFtheta \in C$, for sufficiently large $k$.
    Using $|\sup_{t} \phi(t) - \sup_t \psi(t)| \leq \sup_{t} |\phi(t) -\psi(t)|$ and $1$-Lipschitz continuity of $u \to \max\{t - u, 0\}$, we have
    \[
    \begin{aligned}
    & \sup_{\BFtheta \in C}\left|\operatorname{In-CVaR}_\alpha^\beta \left(\mathcal{L}\left(\,\abs{f(\boldsymbol{X}_{D_k}, \BFtheta) - Y_{D_k}}\right)\right) - \operatorname{In-CVaR}_\alpha^\beta \left(\mathcal{L}\left(\,\abs{f(\boldsymbol{X}_{D_0}, \BFtheta) - Y_{D_0}}\right)\right)\right|\\
    & \qquad \leq \frac{1}{\beta- \alpha} \sup_{\BFtheta \in C}\left|\sup_{t \in [0,M]} \left\{\beta t - \mathbb{E}\left[\max\left\{t - \mathcal{L}\left(\,\abs{f(\boldsymbol{X}_{D_k}, \BFtheta) - Y_{D_k}}\right), 0\right\}\right]\right\} \right.\\
    & \qquad \quad \left.- \sup_{t \in [0,M]} \left\{\beta t - \mathbb{E}\left[\max\left\{t - \mathcal{L}\left(\,\abs{f(\boldsymbol{X}_{D_0}, \BFtheta) - Y_{D_0}}\right), 0\right\}\right]\right\}\right|\\
    & \qquad \quad + \frac{1}{\beta- \alpha} \sup_{\BFtheta \in C}\left|\sup_{t \in [0,M]} \left\{\alpha t - \mathbb{E}\left[\max\left\{t - \mathcal{L}\left(\,\abs{f(\boldsymbol{X}_{D_k}, \BFtheta) - Y_{D_k}}\right), 0\right\}\right]\right\} \right.\\
    & \qquad \quad \left.- \sup_{t\in [0,M]} \left\{\alpha t - \mathbb{E}\left[\max\left\{t - \mathcal{L}\left(\,\abs{f(\boldsymbol{X}_{D_0}, \BFtheta) - Y_{D_0}}\right), 0\right\}\right]\right\}\right|\\
    & \qquad \leq \frac{2}{\beta - \alpha} \mathbb{E} \left[\min\{R_k, M\}\right] \to 0,
    \end{aligned}
    \]
    which proves uniform convergence on $C$ since the sequence $\{D_k\}_{k=1}^\infty$ is arbitrary.
    
    To establish stability of the solution sets, first note that uniform convergence implies convergence of the optimal values: 
    \begin{equation}
    \label{eq:thm_stability(b)1}
    \displaystyle \inf_{\BFtheta \in C} \operatorname{In-CVaR}_\alpha^\beta \left(\mathcal{L}(\,\abs{f(\boldsymbol{X}_{D}, \BFtheta) - Y_{D}})\right) \to \inf_{\BFtheta \in C} \operatorname{In-CVaR}_\alpha^\beta \left(\mathcal{L}\left(\,\abs{f(\boldsymbol{X}_{D_0}, \BFtheta) - Y_{D_0}}\right)\right)
    \end{equation}
    as $d_{\text{Prok}}(D, D_0) \rightarrow 0$.
    Assume, by contradiction, that there exists a sequence of distributions $\{D_k\}_{k=1}^\infty$ with $ d_{\text{Prok}}(D_k, D_0) \to 0$ and ${\BFtheta}_{D_k} \in {\mathcal{S}}_\alpha^\beta(D_k,f)$ such that $\operatorname{dist}\big(\BFtheta_{D_k}, {\mathcal{S}}_\alpha^\beta(D_0,f)\big) \geq \varepsilon$ for some $\varepsilon>0$. Since $C$ is compact, by passing to a subsequence if necessary, we assume  that $\BFtheta_{D_k}$ tends to a point $\BFtheta^* \in C$. It follows that $\BFtheta^* \notin {\mathcal{S}}_\alpha^\beta(D_0,f)$, and hence 
    \begin{equation}
    \label{eq:thm_stability(b)2}
    \operatorname{In-CVaR}_\alpha^\beta \left(\mathcal{L}\left(\,\abs{f(\boldsymbol{X}_{D_0}, \BFtheta^*) - Y_{D_0}}\right)\right) > \inf_{\BFtheta \in C} \operatorname{In-CVaR}_\alpha^\beta \left(\mathcal{L}\left(\,\abs{f(\boldsymbol{X}_{D_0}, \BFtheta) - Y_{D_0}}\right)\right).
    \end{equation}
    Moreover, 
    \[
    \displaystyle \operatorname{In-CVaR}_\alpha^\beta \left(\mathcal{L}\left(\,\abs{f(\boldsymbol{X}_{D_k}, \BFtheta_{D_k}) - Y_{D_k}}\right)\right) = \inf_{\BFtheta \in C} \operatorname{In-CVaR}_\alpha^\beta \left(\mathcal{L}\left(\,\abs{f(\boldsymbol{X}_{D_k}, \BFtheta) - Y_{D_k}}\right)\right) 
    \]
    and
    \begin{equation}
    \label{eq:stability1}
    \begin{aligned}
        & \displaystyle \operatorname{In-CVaR}_\alpha^\beta \left(\mathcal{L}\left(\,\abs{f(\boldsymbol{X}_{D_k}, \BFtheta_{D_k}) - Y_{D_k}}\right)\right) - \operatorname{In-CVaR}_\alpha^\beta \left(\mathcal{L}\left(\,\abs{f(\boldsymbol{X}_{D_0}, \BFtheta^*) - Y_{D_0}}\right)\right) \\
        =   &  \left[ \displaystyle \operatorname{In-CVaR}_\alpha^\beta \left(\mathcal{L}\left(\,\abs{f(\boldsymbol{X}_{D_k}, \BFtheta_{D_k}) - Y_{D_k}}\right)\right) - \operatorname{In-CVaR}_\alpha^\beta \left(\mathcal{L}\left(\,\abs{f(\boldsymbol{X}_{D_0}, \BFtheta_{D_k}) - Y_{D_0}}\right)\right)  \right]\\
        & + \left[\operatorname{In-CVaR}_\alpha^\beta \left(\mathcal{L}\left(\,\abs{f(\boldsymbol{X}_{D_0}, \BFtheta_{D_k}) - Y_{D_0}}\right)\right) - \operatorname{In-CVaR}_\alpha^\beta \left(\mathcal{L}\left(\,\abs{f(\boldsymbol{X}_{D_0}, \BFtheta^*) - Y_{D_0}}\right)\right)\right].
    \end{aligned}
    \end{equation}
    The first term  in the right-hand side of \eqref{eq:stability1} tends to zero by uniform convergence, and the second by continuity. Therefore, 
    \[
    \displaystyle \operatorname{In-CVaR}_\alpha^\beta \left(\mathcal{L}\big(\,\abs{f(\boldsymbol{X}_{D_k}, \BFtheta_{D_k}) - Y_{D_k}}\big)\right) \rightarrow \operatorname{In-CVaR}_\alpha^\beta \left(\mathcal{L}\big(\,\abs{f(\boldsymbol{X}_{D_0}, \BFtheta^*) - Y_{D_0}}\big)\right),
    \]
    which yields a contradiction to \eqref{eq:thm_stability(b)1} and \eqref{eq:thm_stability(b)2}.
\Halmos \end{proof}

Building on Theorem \ref{thm:stability}, we obtain the following corollary, which shows that the In-CVaR based estimator is qualitatively robust at $D_0$ with $\beta \neq 1$. 

\begin{corollary}[Qualitative robustness of In-CVaR based estimators]
    \label{thm:sufficient_qrao}
     Assume that Assumptions \ref{ass:a1} and \ref{ass:a2} hold and $\beta \neq 1$. Under the same conditions presented in Theorem \ref{thm:stability}, the In-CVaR based estimator $ \hat{\mathcal{S}}_\alpha^\beta(\bullet, f)$ is qualitatively robust at $D_0$ with respect to the Prokhorov metric.
\end{corollary}

\begin{proof}{Proof of Corollary \ref{thm:sufficient_qrao}.}

By Theorem \ref{thm:stability}, for any $\varepsilon > 0$, there exists $\delta \in (0, \bar \delta/{2})$ such that $\mathbb{D}\big( {\mathcal{S}}_\alpha^\beta(D^\prime , f), {\mathcal{S}}_\alpha^\beta(D_0 , f)\big) \leq \varepsilon$ for all $D^\prime \in \mathcal{B}_{\text{Prok}}(D_0, 2\delta)$. For any $D \in \mathcal{B}_{\text{Prok}}(D_0, \delta)$, by Varadarajan's theorem \citep[Theorem 11.4.1]{dudley2018real} and the fact that the Prokhorov metric metrizes weak convergence \citep[Theorem 2.14]{huber2009robust},  we have 
$D^{(n)} \in \mathcal{B}_{\text{Prok}}(D_0, 2\delta)$ w.p.1 for all sufficiently large $n$. Therefore, $\mathbb{D}\big( \hat{\mathcal{S}}_\alpha^\beta(D^{(n)}, f), {\mathcal{S}}_\alpha^\beta(D_0 , f)\big) \leq \varepsilon$ w.p.1 for all sufficiently large $n$, which proves the qualitative robustness of  $ \hat{\mathcal{S}}_\alpha^\beta(\bullet, f)$.
\Halmos\end{proof}

\subsection{Consistency of In-CVaR based estimator}
\label{sec:consistency}

For any  $D \in \mathcal{D}$, a statistical estimator $ \hat{T}(\bullet, f)$ is said to be consistent at $D$ if 
\[
\mathbb{D}\big( \hat{{T}}(D^{(n)} , f), {{T}}(D , f)\big) \rightarrow 0\,\, \text{w.p.1 as} \,\, n \rightarrow \infty.
\]
Theorem \ref{thm:consistency} establishes uniform convergence of the empirical In-CVaR objective over a compact set without requiring the almost-sure boundedness or global Lipschitz continuity assumptions imposed in \cite{laguel2021superquantiles}. It further extends the SAA consistency result of \cite{shapiro2021lectures} to the In-CVaR setting under the assumption that the empirical estimators remain within a compact set.

\begin{theorem}[Consistency of In-CVaR based estimators]
    
    \label{thm:consistency}
    Assume that Assumptions \ref{ass:a1} and \ref{ass:a3} hold and let $D \in \mathcal{D}$. Then the following hold:
    \begin{enumerate}[label=(\alph*)]
        \item {\rm $ \operatorname{In-CVaR}_\alpha^\beta \left(\mathcal{L}(\,\abs{f(\boldsymbol{X}_D, \bullet) - Y_D})\right)$} is continuous on $\Theta$;
        \item for any  compact set {\rm $C \subseteq \Theta$, $ \operatorname{In-CVaR}_\alpha^\beta \left(\mathcal{L}\left(\,\abs{f(\boldsymbol{X}_{D^{(n)}}, \BFtheta) - Y_{D^{(n)}}}\right)\right)$} converges to {\rm $ \operatorname{In-CVaR}_\alpha^\beta \left(\mathcal{L}(\,\abs{f(\boldsymbol{X}_D, \BFtheta) - Y_D})\right)$} w.p.1, as $n \to \infty$, uniformly on $C$; 
        \item  suppose that there exists a compact set $C \subseteq \Theta$ such that, 
        w.p.1 for all sufficiently large $n$, the set $\hat{\mathcal{S}}_\alpha^\beta(D^{(n)}, f)$ is nonempty and contained in $C$. Then $\hat{\mathcal{S}}_\alpha^\beta(\bullet, f)$ is consistent at $D$.
    \end{enumerate}
\end{theorem}

\begin{proof}{Proof of Theorem \ref{thm:consistency}.}
    For part (a), fix $\BFtheta \in \Theta$ and any sequence $\{\BFtheta_k\}_{k=1}^\infty \subseteq \Theta$ with $\BFtheta_k \to \BFtheta$. By the triangle inequality for integrals, we have 
    \begin{equation}
    \label{eq:consistency1}
    \begin{aligned}
        & \abs{\displaystyle  \operatorname{In-CVaR}_\alpha^\beta \left(\mathcal{L}(\,\abs{f(\boldsymbol{X}_D, \BFtheta_k) - Y_D})\right)-\operatorname{In-CVaR}_\alpha^\beta \left(\mathcal{L}(\,\abs{f(\boldsymbol{X}_D, \BFtheta) - Y_D})\right) }\\
        & \leq \displaystyle\frac{1}{\beta - \alpha} \int_0^1
        \abs{\operatorname{VaR}_\gamma
        \left(\mathcal{L}(\,\abs{f(\boldsymbol{X}_{D}, \BFtheta_k) - Y_{D}})\right) - \operatorname{VaR}_\gamma \left(\mathcal{L}(\,\abs{f(\boldsymbol{X}_{D}, \BFtheta) - Y_{D}})\right)} \mathrm{d} \gamma\\
        & \leq \frac{1}{\beta-\alpha} \mathbb{E} \left[\,\abs{\mathcal{L}(\,\abs{f(\boldsymbol{X}_D, \BFtheta_k) - Y_D}) - \mathcal{L}(\,\abs{f(\boldsymbol{X}_D, \BFtheta) - Y_D})}\right],  
    \end{aligned}
    \end{equation}
    where the last inequality follows from \citet[Section 1.2]{panaretos2019statistical}.
    Under Assumptions \ref{ass:a1} and \ref{ass:a3}, 
    by the Lebesgue dominated convergence theorem, 
    \[
    \displaystyle\lim_{k\rightarrow \infty} \mathbb{E}\left[\,\abs{\mathcal{L}(\,\abs{f(\boldsymbol{X}_D, \BFtheta_k) - Y_D}) - \mathcal{L}(\,\abs{f(\boldsymbol{X}_D, \BFtheta) - Y_D})}\right] = 0,
    \]
    which implies $\operatorname{In-CVaR}_\alpha^\beta \left(\mathcal{L}(\,\abs{f(\boldsymbol{X}_D, \BFtheta_k) - Y_D})\right) \rightarrow \operatorname{In-CVaR}_\alpha^\beta \left(\mathcal{L}(\,\abs{f(\boldsymbol{X}_D, \BFtheta) - Y_D})\right)$.

    For part (b), fix $\bar{\BFtheta} \in C$ and any sequence $\{\gamma_k\}_{k = 1}^ \infty$ of positive numbers converging to zero. Let $V_k \triangleq C \cap \mathcal{B}(\bar \BFtheta, \gamma_k)$ and 
    \[  \Delta_k(\boldsymbol{x},y)\triangleq\sup _{\BFtheta \in V_k}\abs{\mathcal{L}\left(\,\abs{f(\boldsymbol{x}, \BFtheta) - y}) - \mathcal{L}(\,\abs{f(\boldsymbol{x}, \bar \BFtheta) - y}\right)}.
    \]
    By Assumption \ref{ass:a1}, $\lim_{k \to \infty} \Delta_k(\boldsymbol{x}, y) = 0$ for every $(\boldsymbol{x},y) \in \mathcal{X} \times \mathcal{Y}$. Moreover, by Assumption \ref{ass:a3}, $\left\{\Delta_k(\boldsymbol{X}_D, Y_D)\right\}_{k=1}^\infty$ are dominated by an integrable function. Therefore, by the Lebesgue dominated convergence theorem,
    \[
    \lim_{k \rightarrow \infty} \mathbb{E}\left[\Delta_k(\boldsymbol{X}_D, Y_D)\right] = \mathbb{E}\left[\lim_{k \rightarrow \infty}\Delta_k(\boldsymbol{X}_D, Y_D)\right] = 0.
    \]
    Moreover, following the same argument as in \eqref{eq:consistency1}, we have
    \[
    \begin{aligned}
        & \sup_{\BFtheta \in V_k}\left|\operatorname{In-CVaR}_\alpha^\beta \left(\mathcal{L}\left(\,\abs{f(\boldsymbol{X}_{D^{(n)}}, \BFtheta) - Y_{D^{(n)}}}\right)\right) - \operatorname{In-CVaR}_\alpha^\beta \left(\mathcal{L}\left(\,\abs{f(\boldsymbol{X}_{D^{(n)}}, \bar \BFtheta) - Y_{D^{(n)}}}\right)\right)\right|\\
        & \qquad \leq  \frac{1}{(\beta-\alpha)n} \sum_{i=1}^n \Delta_{k}(\boldsymbol{x}_{D}^{(i)}, y_{D}^{(i)}).      
    \end{aligned}
    \]
    By the strong law of large numbers, $\sum_{i=1}^n \Delta_{k}(\boldsymbol{x}_{D}^{(i)}, y_{D}^{(i)})/n$ converges w.p.1 to $\mathbb{E}[\Delta_k(\boldsymbol{X}_D, Y_D)]$ as $n \rightarrow \infty$. Therefore, for any $\varepsilon > 0$, there exist a neighborhood $W$ of $\bar\BFtheta$ and  $\bar n_W \in \mathbb{N}_+$ such that w.p.1 for all $n \geq \bar n_W$, 
    \[
    \begin{aligned}
    & \sup_{\BFtheta \in W \cap C} \left|\operatorname{In-CVaR}_\alpha^\beta \left(\mathcal{L}\left(\,\abs{f(\boldsymbol{X}_{D^{(n)}}, \BFtheta) - Y_{D^{(n)}}}\right)\right) - \operatorname{In-CVaR}_\alpha^\beta \left(\mathcal{L}\left(\,\abs{f(\boldsymbol{X}_{D^{(n)}}, \bar \BFtheta) - Y_{D^{(n)}}}\right)\right)\right| < \varepsilon.
    \end{aligned}
    \]
    Since $C$ is compact, there exist finitely many points  $\BFtheta_1, \cdots, \BFtheta_m \in C$ and corresponding neighborhoods $W_1, \cdots, W_m$ covering $C$ such that w.p.1 for all $n \geq \bar n\triangleq \max\{\bar n_{W_1}, \cdots, \bar n_{W_m}\}$, 
    \begin{equation}
    \label{eq:consis1}
    \begin{aligned}
        & \sup_{\BFtheta\in W_j\cap C}  \left|\displaystyle  \operatorname{In-CVaR}_\alpha^\beta \left(\mathcal{L}\left(\,\abs{f(\boldsymbol{X}_{D^{(n)}}, \BFtheta) - Y_{D^{(n)}}}\right)\right) - \operatorname{In-CVaR}_\alpha^\beta \left(\mathcal{L}\left(\,\abs{f(\boldsymbol{X}_{D^{(n)}}, \BFtheta_j) - Y_{D^{(n)}}}\right)\right)  \right| < \varepsilon.
    \end{aligned}
    \end{equation}
    By part (a), these neighborhoods can be chosen such that
    \begin{equation}
    \label{eq:consis2}
    \displaystyle\sup_{\BFtheta\in W_j\cap C}  \abs{\displaystyle  \operatorname{In-CVaR}_\alpha^\beta \left(\mathcal{L}(\,\abs{f(\boldsymbol{X}_D, \BFtheta) - Y_D})\right)- \displaystyle  \operatorname{In-CVaR}_\alpha^\beta \left(\mathcal{L}(\,\abs{f(\boldsymbol{X}_D, \BFtheta_j) - Y_D})\right)  } < \varepsilon.
    \end{equation}
    According to \citet[Example 2.11]{cont2010robustness}, 
    there exists $n^*$ such that w.p.1 for each $\BFtheta_j$ and $n\geq n^*$, 
    \begin{equation}
    \label{eq:consis3}
    \displaystyle \abs{\displaystyle  \operatorname{In-CVaR}_\alpha^\beta \left(\mathcal{L}\left(\,\abs{f(\boldsymbol{X}_{D^{(n)}}, \BFtheta_j) - Y_{D^{(n)}}}\right)\right)- \displaystyle  \operatorname{In-CVaR}_\alpha^\beta \left(\mathcal{L}(\,\abs{f(\boldsymbol{X}_D, \BFtheta_j) - Y_D})\right)  } < \varepsilon.
    \end{equation}
    Combining \eqref{eq:consis1}-\eqref{eq:consis3}, we conclude that w.p.1 for all $n \geq \max\{\bar{n}, n^*\}$,
    {
    \[
    \begin{aligned}
    & \displaystyle\sup_{\BFtheta\in C}  \Bigg|\displaystyle  \operatorname{In-CVaR}_\alpha^\beta \left(\mathcal{L}\left(\,\abs{f(\boldsymbol{X}_{D^{(n)}}, \BFtheta) - Y_{D^{(n)}}}\right)\right)  -   \operatorname{In-CVaR}_\alpha^\beta \left(\mathcal{L}(\,\abs{f(\boldsymbol{X}_D, \BFtheta) - Y_D})\right)  \Bigg| < 3 \varepsilon.
    \end{aligned}
    \]
    }
Since $\varepsilon > 0 $ is arbitrary, part (b) follows. 

For part (c), suppose, for contradiction, that $\mathbb{D}\left(\hat{\mathcal{S}}_\alpha^\beta(D^{(n)}, f), \mathcal{S}_\alpha^\beta(D, f)\right) \nrightarrow 0$. Since $C$ is compact, by passing to a subsequence if necessary, we can assume that there exists $\hat\BFtheta^{(n)}
\in
\hat{\mathcal{S}}_{\alpha}^{\beta}(D^{(n)},f)$ such that $\operatorname{dist} \left(\hat\BFtheta^{(n)}, \mathcal{S}_\alpha^\beta(D, f)\right) \geq \varepsilon_0$ for some $\varepsilon_0 > 0$ and that there exists $\bar\BFtheta\in C$
 such that
$
\hat\BFtheta^{(n)}\to\bar\BFtheta \in C.
$
It follows that $\bar\BFtheta \notin \mathcal{S}_\alpha^\beta(D, f)$.
By \eqref{eq:consis1} and \eqref{eq:consis3}, we obtain that w.p.1,
\[\operatorname{In\text{-}CVaR}_{\alpha}^{\beta}
\left(
\mathcal{L}\left(
|f(\boldsymbol{X}_{D^{(n)}},\hat\BFtheta^{(n)})-Y_{D^{(n)}}|
\right)
\right) \to \operatorname{In\text{-}CVaR}_{\alpha}^{\beta}
\left(
\mathcal{L}\left(
|f(\boldsymbol{X}_D,\bar\BFtheta)-Y_D|
\right)
\right).
\]
By optimality of $\hat\BFtheta^{(n)}$, for any fixed $\BFtheta \in \Theta$, we have 
\begin{equation}
\label{eq:consis4}
\operatorname{In\text{-}CVaR}_{\alpha}^{\beta}
\left(
\mathcal{L}\left(
|f(\boldsymbol{X}_{D^{(n)}},\hat\BFtheta^{(n)})-Y_{D^{(n)}}|
\right)
\right) \leq \operatorname{In\text{-}CVaR}_{\alpha}^{\beta}
\left(
\mathcal{L}\left(
|f(\boldsymbol{X}_{D^{(n)}}, \BFtheta)-Y_{D^{(n)}}|
\right)
\right).
\end{equation}
Applying part (b) to the singleton $\{\BFtheta\}$, we have w.p.1,
\[
\operatorname{In\text{-}CVaR}_{\alpha}^{\beta}
\left(
\mathcal{L}\left(
|f(\boldsymbol{X}_{D^{(n)}}, \BFtheta)-Y_{D^{(n)}}|
\right)
\right) \to \operatorname{In\text{-}CVaR}_{\alpha}^{\beta}
\left(
\mathcal{L}\left(
|f(\boldsymbol{X}_{D}, \BFtheta)-Y_{D}|
\right)
\right).
\]
Therefore, by taking the limits of \eqref{eq:consis4}, we have
\[
\operatorname{In\text{-}CVaR}_{\alpha}^{\beta}
\left(
\mathcal{L}\left(
|f(\boldsymbol{X}_D,\bar\BFtheta)-Y_D|
\right)
\right) \leq \operatorname{In\text{-}CVaR}_{\alpha}^{\beta}
\left(
\mathcal{L}\left(
|f(\boldsymbol{X}_{D}, \BFtheta)-Y_{D}|
\right)
\right),
\]
which implies that $\bar \BFtheta \in \mathcal{S}_\alpha^\beta(D, f)$ and yields a contradiction. This proves part (c).
\Halmos\end{proof}

    Indeed, in the analysis of Theorem \ref{thm:consistency}, we can further relax the continuity assumption on the regression function in Assumption \ref{ass:a1}: the proof remains valid with $f(\boldsymbol{x}, \bullet)$ being continuous for almost every $\boldsymbol{x} \in \mathcal{X}$.

\subsection{Qualitative robustness implies tail trimming ($\beta \neq 1$)}
\label{sec:necessary_qrao}

In this subsection, we show that tail trimming with  $\beta \neq 1$ is necessary for the  qualitative robustness under the following additional assumption. 
\begin{enumerate}[
    label=\textbf{Assumption~(B):}~,
    ref=(B),
    leftmargin=0pt,
    labelindent=0pt,
    itemindent=0pt,
    labelwidth=0pt,
    labelsep=0pt,
    listparindent=0pt,
    align=left
    ]
    \item \label{ass:b3}
    For the regression function $f: \mathcal{X} \times \Theta \rightarrow \mathcal{Y}$, for every $M>0$, there exists $\tilde \BFtheta \in \Theta$ such that
    \[
    \sup_{\boldsymbol{x} \in \mathcal{X}}\inf_{\{\BFtheta \in \Theta: \|\BFtheta\| \leq M\}} \left|f(\boldsymbol{x}, \tilde \BFtheta) - f(\boldsymbol{x},  \BFtheta)\right| = \infty.
    \]
\end{enumerate}

Assumption \ref{ass:b3} requires that  one can find a parameter whose prediction can be arbitrarily separated from those generated by parameters within the bounded parameter region. We verify in Appendix \ref{app:assumption_b3} that this assumption holds for linear, polynomial, exponential, logarithmic, power and piecewise affine regression functions under appropriate covariate and parameter spaces. Under Assumption~\ref{ass:a2}, qualitative robustness implies that for each perturbed distribution $D$ close to the nominal distribution $D_0$ and sufficiently large sample size, the empirical solution sets must be contained in a common compact set w.p.1. Using this fact, we obtain the following result by a constructive counterexample showing that qualitative robustness implies tail trimming.

\begin{theorem}
    \label{thm: qrao_necessary}
    Assume that Assumptions \ref{ass:a1}--\ref{ass:a3} hold. If $ \hat{\mathcal{S}}_\alpha^\beta(\bullet, f)$ is qualitatively robust at $D_0$ with respect to the Prokhorov metric, then there exists $\bar\delta > 0$ such that 
    {\rm
    \begin{equation}
    \label{eq:necessary}
    \sup_{D \in \mathcal{B}_{\text{Prok}}(D_0, \bar\delta)} \inf_{{\BFtheta} \in {\mathcal{S}}_\alpha^\beta(D, f)} \| {\BFtheta}
    \| < \infty.
    \end{equation}
    }
    If, in addition, Assumption \ref{ass:b3} holds, then $\beta \neq 1$. 
\end{theorem}

\begin{proof}{Proof of Theorem \ref{thm: qrao_necessary}.}

By the  qualitative robustness of $ \hat{\mathcal{S}}_\alpha^\beta(\bullet, f)$ at $D_0$ with respect to the Prokhorov metric, the empirical solution sets must be contained in a common compact set w.p.1. By Theorem \ref{thm:consistency} (c), there exists $\bar\delta > 0$ such that $\hat{\mathcal{S}}_\alpha^\beta(\bullet, f)$ is consistent at every $D \in \mathcal{B}_{\text{Prok}}(D_0, \bar\delta)$, meaning that  there exists $\bar\delta > 0$ such that, for every fixed $D \in \mathcal{B}_{\text{Prok}}(D_0, \bar\delta)$, w.p.1, 
\[
\limsup_{n \to \infty} \mathbb{D}\left( \hat{\mathcal{S}}_\alpha^\beta(D^{(n)}, f),  {\mathcal{S}}_\alpha^\beta(D_0, f)\right) \leq 1 \text{ and }
\limsup_{n \to \infty}\mathbb{D}\left(\hat{\mathcal{S}}_\alpha^\beta(D^{(n)}, f), {\mathcal{S}}_\alpha^\beta(D, f)\right) = 0.
\]
For any  $\hat{\BFtheta}_\alpha^\beta(D^{(n)}, f) \in \hat{\mathcal{S}}_\alpha^\beta(D^{(n)}, f)$, by the triangle inequality, w.p.1, 
\[
\begin{aligned}
 \inf_{{\BFtheta} \in {\mathcal{S}}_\alpha^\beta(D, f)} \| {\BFtheta}
    \|  \leq  \limsup_{n \to \infty} \left(\| \hat{\BFtheta}_\alpha^\beta(D^{(n)}, f)\|  + \operatorname{dist}\left(\hat{\BFtheta}_\alpha^\beta(D^{(n)}, f), \mathcal{S}_\alpha^\beta(D,f)\right)\right) \leq \sup_{\BFtheta \in \mathcal{S}_\alpha^\beta (D_0, f)} \|\BFtheta\| + 1,
\end{aligned}
\]
which proves \eqref{eq:necessary} since the bound is deterministic and independent of $D$. 

We next show the second statement under Assumption \ref{ass:b3} by contradiction.  Suppose $\beta=1$. If \eqref{eq:necessary} holds for some $\bar\delta > 0$, then there exists $M > 0$ such that 
\begin{equation}
\label{eq:necessary1}
\sup_{D \in \mathcal{B}_{\text{Prok}}(D_0, \bar \delta)} \inf_{{\BFtheta} \in {\mathcal{S}}_\alpha^1(D, f)} \| {\BFtheta}
    \| \leq M - 1.
\end{equation}
Fix $\delta \in (0, \min\{1-\alpha, \bar\delta\})$. By Assumption \ref{ass:b3}, there exist $\tilde \BFtheta \in \Theta$ and $\boldsymbol{x}^* \in \mathcal{X}$ such that
    \begin{equation}
    \label{eq:upper_infinity}
    \inf_{\{\BFtheta \in \Theta: \|\BFtheta\| \leq M\}} \mathcal{L}\left(\,\abs{f(\boldsymbol{x}^*, \BFtheta) - f(\boldsymbol{x}^*, \tilde\BFtheta)}\right) > \frac{1-\delta}{\delta}\mathbb{E}_{D_0}\left[\mathcal{L}\left(\,\abs{f(\boldsymbol{X}_{D_0},  \tilde\BFtheta ) -Y_{D_0}}\right)\right].
    \end{equation}
    Let $G$ be a degenerate distribution supported at $(\boldsymbol{x}^*, f(\boldsymbol{x}^*,\tilde \BFtheta))$ and set $D = (1-\delta) D_0 + \delta G$. Then $D \in \mathcal{B}_{\text{Prok}}(D_0,  \bar\delta)$ by Strassen's Theorem \cite[Theorem 2.13]{huber2009robust}. By \eqref{eq:necessary1}, there exists $  {\BFtheta}_\alpha^1(D, f) \in {\mathcal{S}}_\alpha^1(D,f)$ with $\|{\BFtheta}_\alpha^1(D, f)\| \leq M$. Then, 
    \begin{equation}
    \label{eq:necessary2}
    \begin{aligned}
        & \operatorname{In-CVaR}_{\alpha}^{1} \left(\mathcal{L}\left(\,\abs{f(\boldsymbol{X}_{D},  \BFtheta_\alpha^1(D, f) ) -Y_{D}}\right)\right) 
        \geq  \frac{\delta}{1-\alpha}\mathcal{L}\left(\,\abs{f(\boldsymbol{x}^*, \BFtheta_\alpha^1(D, f)) - f(\boldsymbol{x}^*, \tilde\BFtheta) }\right)\\
        & \qquad \geq \frac{\delta}{1-\alpha} \inf_{\{\BFtheta \in \Theta: \|\BFtheta\| \leq M\}} \mathcal{L}\left(\,\abs{f(\boldsymbol{x}^*, \tilde\BFtheta) - f(\boldsymbol{x}^*, \BFtheta)}\right). 
    \end{aligned}        
    \end{equation}
    By construction, $\mathcal{L}\left(\,\abs{f(\boldsymbol{X}_{G}, \tilde\BFtheta)-Y_{G}}\right) = 0$ w.p.1. By the inequality \eqref{eq:upper_infinity}, we have
    \[
    \begin{aligned}
    & \operatorname{In-CVaR}_{\alpha}^{1} \left(\mathcal{L}\left(\,\abs{f(\boldsymbol{X}_{D},  \tilde\BFtheta ) -Y_{D}}\right)\right)
    \leq  \frac{1}{1 - \alpha} \mathbb{E}\left[\mathcal{L}\left(\,\abs{f(\boldsymbol{X}_{D},  \tilde\BFtheta ) -Y_{D}}\right)\right]\\
     & \qquad = \frac{1-\delta}{1-\alpha} \mathbb{E}\left[\mathcal{L}\left(\,\abs{f(\boldsymbol{X}_{D_0},  \tilde\BFtheta ) -Y_{D_0}}\right)\right]  < \frac{\delta}{1-\alpha} \inf_{\{\BFtheta \in \Theta: \|\BFtheta\| \leq M\}} \mathcal{L}\left(\,\abs{f(\boldsymbol{x}^*, \tilde\BFtheta) - f(\boldsymbol{x}^*, \BFtheta)}\right),
    \end{aligned}
    \]
    which creates a contradiction to \eqref{eq:necessary2}. This completes the proof.
\Halmos\end{proof}

Combining Corollary \ref{thm:sufficient_qrao} and Theorem  \ref{thm: qrao_necessary}, 
we conclude this section that the In-CVaR based estimator is qualitatively robust with respect to the Prokhorov metric if and only if $\beta \neq 1$, under Assumptions \ref{ass:a1}-\ref{ass:a3}, \ref{ass:b3}, and the local compactness of the  
solution set ${\mathcal{S}}_\alpha^\beta(\bullet, f)$ at $D_0$.  
This characterization highlights an important distinction from existing robustness results for ERM based kernel learning: \cite{zhang2024statistical} establish qualitative robustness for ERM based kernel learning estimators under the $\psi$-weak topology, whereas the ERM based estimator fails to be qualitatively robust under the weak topology considered here.

\section{PAC Bound for In-CVaR based Regression}
\label{sec:PAC}

Besides the distributional perturbation considered in the section above, in this section, we address the finite-sample robustness  under sample contamination by analyzing the  PAC bound.
Section~\ref{sec:adversarial_pac} introduces the adversarial contamination model and the PAC criterion, while Section~\ref{sec:pac_incvar} presents the main PAC result.

\subsection{Adversarial contamination model and PAC criterion}
\label{sec:adversarial_pac}

Throughout this section, uncontaminated observations are assumed to follow the nominal distribution $D_0$, and we suppress the subscript $D_0$ when no ambiguity arises. Let $\boldsymbol Z=(\boldsymbol X,Y)$ take values in $\mathcal Z=\mathcal X\times\mathcal Y \subseteq \mathbb{R}^{p+1}$. Given a sample size $n\in\mathbb{N}_+$ and a nominal distribution $D_0\in\mathcal D$, let
$\boldsymbol{Z}_{1:n}=(\boldsymbol{Z}_1,\ldots,\boldsymbol{Z}_n)$ be an
independent and identically distributed (i.i.d.) sample drawn from $D_0$. For a contamination level $\varepsilon\in[0,1)$,  a sample $\boldsymbol Z_{1:n}^{\varepsilon}$ is said to be an $\varepsilon$-adversarially contaminated sample associated
with $\boldsymbol{Z}_{1:n}$ if 
$
\left|\boldsymbol Z_{1:n}^{\varepsilon} \cap \boldsymbol{Z}_{1:n}\right| \geq (1-\varepsilon) n.
$
For simplicity, we relabel the contaminated observations by a suitable permutation so that the above condition is equivalently expressed as
$
\left|\{i\in[n]:\boldsymbol Z_i^\varepsilon\neq\boldsymbol Z_i\}\right|
\leq\varepsilon n.
$
The contaminated observations can be selected arbitrarily and may depend on the clean sample.  So in this sense, it is stronger than both Huber's contamination and full-neighborhood contamination \citep{blanchet2025distributionally}.

 Under the nominal distribution $D_0$, suppose that $
Y=f(\boldsymbol X, \BFtheta^*)+\xi,$
where $f:\mathcal X \times \Theta \to\mathcal Y$, $\BFtheta^* \in\Theta \subseteq\mathbb{R}^q$ denotes the true parameter, and $\xi$ is a noise term. Let $D_\varepsilon^{(n)}$ denote the empirical distribution of an $\varepsilon$-adversarially contaminated sample $\boldsymbol Z_{1:n}^{\varepsilon}$, and let
$\hat{\BFtheta}(D_\varepsilon^{(n)}, f)$ be the corresponding estimator. With $
\mathcal D_{\varepsilon}^n$ representing a set of empirical distributions associated with   $\varepsilon$-adversarially contaminated samples corresponding to $\boldsymbol{Z}_{1:n}$, we consider the PAC guarantee of the form
\begin{equation}
\label{eq:pac}
\mathbb{P}\left(\,\,\sup_{D_\varepsilon^{(n)} \in 
\mathcal D_{\varepsilon}^n} \|\hat{\BFtheta}(D_\varepsilon^{(n)}, f) - \BFtheta^*\| \leq t( n ,\delta, \varepsilon)\right) \geq 1 - \delta,
\end{equation}
where  $\delta \in (0,1)$ is a prescribed confidence level, and $t(n,\delta,\varepsilon)$ is an estimation error bound. Thus, with probability at least $1-\delta$, the estimation error is controlled uniformly by $t( n ,\delta, \varepsilon)$ over all admissible adversarial contamination; we refer to \cite{loh2025theoretical} for the detailed discussion. 

Several recent studies establish related PAC bounds for trimmed regression.
\cite{bhatia2015robust, bhatia2017consistent} study LTS based linear regression under adversarial response contamination. More recently, allowing contamination of both covariates and responses, \cite{pensia2025robust} derive a PAC bound for LTS based linear regression preceded by a covariate-filtering step.
According to their results, for $q$-dimensional linear regression with $(4,2)$-hypercontractive covariates and additive noise satisfying $(\sigma_k)^k=\mathbb{E}[|\xi|^k]<\infty$, their bound is of order 
\begin{equation}
\label{eq:bound_pre}
O\left(\sigma_k\left(\frac{q\log q }{n} + \varepsilon + \frac{\log (1/\delta)}{n}\right)^{1/2 - 1/k}\right).
\end{equation} 
With an additional postprocessing step, their LTS estimator attains the near-optimal bound
\begin{equation}
\label{eq:near-optimal}
O\left(\sigma \left(\sqrt{\frac{q \log(qn)}{n}} + \sqrt{\varepsilon}  + \sqrt{\frac{\log(1/\delta)}{n}}\right)\right),
\end{equation}
where $\sigma^2 = \mathbb{E}[\xi^2]$. Their analysis relies on a strong stability property induced by covariate filtering, which provides uniform control of the empirical mean and covariance over sufficiently large subsets of the sample. Because this argument is closely tied to the linear covariate structure, extending it directly to nonlinear regression is nontrivial. \cite{oliveira2025trimmed} obtain a related PAC bound for trimmed estimators in minimaximization problems over convex parameter spaces. In the linear setting, their statistical bound slightly improves upon \eqref{eq:near-optimal}, whereas the computation of such a trimmed minimaximization estimator remains highly challenging even for linear regression.

So in Section \ref{sec:pac_incvar},  without resorting to preliminary covariate filtering or minimaximizing optimization, we illustrate that an In-CVaR based nonlinear regression estimator can achieve a good PAC bound in addressing the sample contamination, which is further evidenced by the numerical experiments.   To the best of our knowledge, no PAC bound has been established for the empirical In-CVaR regression estimator itself.  
In the linear setting, under the same $(4,2)$-hypercontractivity condition on the covariates as in \cite{pensia2025robust}, the empirical In-CVaR based estimator can be combined with the postprocessing procedure of \cite{pensia2025robust} to attain the near-optimal error bound in \eqref{eq:near-optimal}.

\subsection{PAC bound of In-CVaR based estimators}
\label{sec:pac_incvar}

Given an $\varepsilon$-adversarially contaminated sample $\boldsymbol{Z}_{1:n}^\varepsilon = \left\{\left(\boldsymbol{X}_i^\varepsilon, Y_i^\varepsilon\right)\right\}_{i=1}^n$ associated with an i.i.d. sample drawn from the nominal
distribution $D_0$, let $D_\varepsilon^{(n)}$ denote its empirical distribution. For $0 \leq \alpha < \beta \leq 1$,  the empirical In-CVaR based estimator under the $\ell_2$ loss is
\[
\hat{\mathcal{S}}_\alpha^\beta(D_\varepsilon^{(n)},f) = \argmin_{\BFtheta \in \Theta} \, \operatorname{In-CVaR}_{\alpha}^{\beta} \left(\left(f(\boldsymbol{X}_{D_\varepsilon^{(n)}}, \BFtheta) - Y_{D_\varepsilon^{(n)}} \right)^2 \right).
\]

We make the following assumptions:

\begin{enumerate}[
    label=\textbf{Assumption~(P\arabic*):}~,
    ref=(P\arabic*),
    start=1,
    leftmargin=0pt,
    labelindent=0pt,
    itemindent=0pt,
    labelwidth=0pt,
    labelsep=0pt,
    listparindent=0pt,
    align=left
    ]
    \item 
\label{ass:p1}
    The noise $\xi$ is a  {random variable independent} of the covariates $\boldsymbol X$, satisfying $\mathbb{E}[\xi] = 0$, $\sigma^2 = \mathbb{E}[\xi^2] < \infty$ and $(\sigma_k)^k = \mathbb{E}[|\xi|^k] < \infty$ for some $k>2$;
    \item 
\label{ass:p2}
    There exists a measurable function  $\omega : \mathcal{X} \to \mathbb{R}_+$ satisfying $\|\omega\|_{L_4(D_{0,\boldsymbol{X}})}< \infty$, where $D_{0, \boldsymbol{X}}$ denotes the marginal distribution of $\boldsymbol{X}$ under $D_0$, such that  for all $\boldsymbol x\in \mathcal{X} \text{ and } \BFtheta_1, \BFtheta_2 \in \Theta$,
    \[     |f(\boldsymbol{x},\BFtheta_1) -  f(\boldsymbol{x},\BFtheta_2) | \leq \omega(\boldsymbol{x}) \|\BFtheta_1 - \BFtheta_2\|;
    \]
    \item 
    \label{ass:p3}
    There exists $\mu > 0$ such that,  for every $\BFtheta \in \Theta$, $
    \mu \|\BFtheta - \BFtheta^*\| \leq r_\BFtheta,$
    where  $
r_\BFtheta \triangleq \|\Delta_{\BFtheta}\|_{L_2(D_{0,\boldsymbol{X}})}
$, and
$\Delta_{\BFtheta}(\boldsymbol{x})
\triangleq
f(\boldsymbol{x},\BFtheta)-f(\boldsymbol{x},\BFtheta^*)$ for $\BFtheta\in\Theta$ and $\boldsymbol{x} \in\mathcal X$,  
representing the prediction error.
\end{enumerate}

Assumption~\ref{ass:p1} allows for heavy-tailed noise while requiring a finite moment of order $k>2$, as in \cite{pensia2025robust}. Assumption~\ref{ass:p2} requires a Lipschitz condition on the regression function  with respect to its parameter, with a covariate-dependent modulus parameter having a finite fourth moment.
Assumption~\ref{ass:p3} ensures that a small prediction error implies a small estimation error. 

In the following propositions, we show that  Assumptions \ref{ass:p2} and \ref{ass:p3} hold
for linear and generalized linear regression models respectively, with proofs in Appendix \ref{app:assp3-p3}. In addition, we also show that these models satisfy the metric entropy condition which is useful in our analysis.  As discussed in Section \ref{sec:sufficient_qrao}, these arguments could be extended to other regression models with suitable covariate transformations or reparameterizations, including polynomial, exponential, logarithmic, and power regression.

\begin{proposition}
    \label{prop:assp2-p3_linear}
    Consider the linear regression function $f(\boldsymbol{x}, \BFtheta) = \BFtheta^\top \boldsymbol{x}$. Suppose that $\mathbb{E} \left[\|\boldsymbol{X}\|^4\right] < \infty$ and $\Sigma_{\boldsymbol{X}} \triangleq \mathbb{E} \left[\boldsymbol{X} \boldsymbol{X}^\top\right]$ is positive definite. Then:
    \begin{enumerate}[label=(\alph*)]
        \item Assumption \ref{ass:p2} holds with $\omega(\boldsymbol{x}) = \|\boldsymbol{x}\|$;
        \item 
        Assumption \ref{ass:p3} holds with $ \mu =        \sqrt{\lambda_{\min}\left(\Sigma_{\boldsymbol{X}} \right)}$;
        \item defining
\begin{equation}
\label{eq:G}
\mathcal{G} \triangleq \left\{{\Delta_\BFtheta(\bullet)}/{r_\BFtheta}: \BFtheta \in \Theta \backslash \{\BFtheta^*\}\right\},
\end{equation} 
for every distribution $D \in \mathcal{D}$ and $\eta > 0$, we have
    \[ \log\mathcal{N}\left(\mathcal{G}, \|\bullet\|_{L_2(D_{\boldsymbol{X}})}, \eta \|\omega\|_{L_2(D_{\boldsymbol{X}})} \right) \leq q \log \left(1+ \frac{2}{\eta \sqrt{\lambda_{\min}\left(\Sigma_{\boldsymbol{X}} \right)}}\right).
    \]
    \end{enumerate}
\end{proposition}

    Proposition \ref{prop:assp2-p3_linear} shows that, for linear regression, a finite fourth moment and a positive-definite second-moment matrix suffice for Assumptions~\ref{ass:p2} and~\ref{ass:p3}.  These conditions are weaker than the assumptions of \cite{pensia2025robust}, that the covariates are isotropic and satisfy a $(4,2)$-hypercontractivity condition.

We next consider generalized linear regression of the form $f(\boldsymbol{x}, \BFtheta) = g(\BFtheta^\top \boldsymbol{x})$,
    where $g : \mathbb{R} \to \mathbb{R}$ is continuously differentiable. For each $\boldsymbol{x} \in \mathcal{X}$, define 
\[
\underline{m} (\boldsymbol{x}) \triangleq \inf_{\BFtheta \in \Theta} \left|g^\prime(\BFtheta^\top \boldsymbol{x})\right|, \, \overline{m} (\boldsymbol{x}) \triangleq \sup_{\BFtheta \in \Theta} \left|g^\prime(\BFtheta^\top \boldsymbol{x})\right|
\]
and 
\[
M(\boldsymbol{x}) \triangleq \sup_{\BFtheta_1, \BFtheta_2 \in \Theta, \BFtheta_1^\top \boldsymbol{x} \neq \BFtheta_2^\top \boldsymbol{x}}\frac{\left|g^\prime(\BFtheta_1^\top \boldsymbol{x}) - g^\prime(\BFtheta_2^\top \boldsymbol{x})\right|}{|\BFtheta_1^\top \boldsymbol{x} - \BFtheta_2^\top \boldsymbol{x}|}.
\]

\begin{proposition}
    \label{prop:assp2-p3_glinear}
    Consider $f(\boldsymbol{x}, \BFtheta) = g(\BFtheta^\top \boldsymbol{x})$,
    where $g : \mathbb{R} \to \mathbb{R}$ is continuously differentiable. Suppose that $\Theta \subseteq  \mathbb{R}^q$ is bounded and convex. Set ${\omega(\boldsymbol{x})} = \overline{m} (\boldsymbol{x})\|\boldsymbol{x}\| + \frac{\max\{1, \operatorname{diam}({\Theta})\}}{2}M(\boldsymbol{x}) \|\boldsymbol{x}\|^2.$
    If $\mathbb{E} \left[\omega(\boldsymbol{X})^4\right] < \infty$ and $\Sigma_{g} \triangleq \mathbb{E} \left[\underline{m}(\boldsymbol{X})^2 \boldsymbol{X}\boldsymbol{X}^\top\right]$ is positive definite, then:
    \begin{enumerate}[label=(\alph*)]
        \item Assumption \ref{ass:p2} holds with $\omega$ as defined above;
        \item Assumption \ref{ass:p3} holds with 
    $\mu =  \sqrt{\lambda_{\min}\left(\Sigma_{g}\right)}$; 
    \item for  every $D\in\mathcal{D}$ and $\eta > 0$,
    \[    
    \begin{aligned}
    & \log\mathcal{N}\left(\mathcal{G}, \|\bullet\|_{L_2(D_{\boldsymbol{X}})}, \eta \|\omega\|_{L_2(D_{\boldsymbol{X}})} \right) \\
    & \qquad \leq 2 q \log \left(1+ \frac{4 \max
    \left\{1, \operatorname{diam}({\Theta})
    \right\}}{\eta}\left(\frac{1}{\sqrt{\lambda_{\min}\left(\Sigma_{g}\right)}}+ \frac{\|\omega\|_{L_2(D_{0,\boldsymbol{X}})}}{\lambda_{\min}\left(\Sigma_{g}\right)}\right)\right),
    \end{aligned}
    \]
    where $\mathcal{G}$ is defined in \eqref{eq:G}.
    \end{enumerate}
\end{proposition}

    Proposition \ref{prop:assp2-p3_glinear} also applies to models of the form $f(\boldsymbol{x}, \BFtheta) = g\left(\BFtheta^\top \phi(\boldsymbol{x})\right)$, where $\phi: \mathcal{X} \to \mathbb{R}^q$ is a given covariate transformation. In this case, $\boldsymbol{X}$ is replaced by $\phi(\boldsymbol{X})$ throughout the conditions and conclusions. This formulation covers exponential and power regression when the corresponding moment and nondegeneracy conditions hold.
    For exponential regression $f(\boldsymbol{x}, 
    \BFtheta)=\exp \left(\BFtheta^{\top} \boldsymbol{x}\right)$, when $\Theta$ is bounded, it holds that $\omega(\boldsymbol{x}) \lesssim \exp (C\|\boldsymbol{x}\|)\left(\|\boldsymbol{x}\|+\|\boldsymbol{x}\|^2\right)$ for some constant $C>0$. In particular, $\mathbb{E}\left[\omega(\boldsymbol{X})^4\right]<\infty$  when $\boldsymbol{X}$ is sub-Gaussian. Next, we present our main PAC result.

\begin{theorem}
    \label{thm:pacbound}
Suppose that Assumption \ref{ass:p1} holds and that the conditions of either  Proposition \ref{prop:assp2-p3_linear} or Proposition \ref{prop:assp2-p3_glinear} are satisfied with correspondingly defined $\omega (\bullet)$. Let $0 \leq \alpha < \beta \leq 1$  such that $\alpha n $ and $\beta n$ are integers and
\begin{equation}
    \label{eq:pac_requirement_beta_alpha}
    1 - \beta - \varepsilon = \Omega\left(\varepsilon + \frac{\log(1/\delta)}{n}\right) \text{ and } 1 - \beta + \alpha = O\left(\varepsilon + \frac{\log(1/\delta)}{n}\right).
\end{equation}
Then there exist constants $\kappa>0$ and $M>0$ such that, whenever $\varepsilon \leq \kappa/\|\omega\|_{L_4(D_{0,\boldsymbol{X}})}^4$ and 
$
n \geq M\|\omega\|_{L_4(D_{0,\boldsymbol{X}})}^4(q+ \log(1/\delta)),
$
with probability at least $1-O(\delta)$,
    \begin{equation}
    \label{eq:pac_bound}
    \begin{aligned}
    & \sup_{D_\varepsilon^{(n)} \in 
\mathcal D_{\varepsilon}^n}\sup_{\BFtheta \in\hat{\mathcal{S}}_\alpha^\beta(D_\varepsilon^{(n)},f)}\left\|\BFtheta- \BFtheta^*\right\| \\
&\qquad \lesssim \sigma\|\omega\|_{L_2(D_{0,\boldsymbol{X}})}\sqrt{\frac{q}{n} + \varepsilon + \frac{\log(1/\delta)}{n}}+  \sigma_k\left(\varepsilon + \frac{\log(1/\delta)}{n}\right)^{1/2-1/k}.
    \end{aligned}
    \end{equation}
\end{theorem}

This is the first  PAC bound established for the In-CVaR based nonlinear regression without covariate filtering, to the best of our understanding.  {
Condition \eqref{eq:pac_requirement_beta_alpha} requires the upper-tail trimming fraction to exceed the contamination level $\varepsilon$, while keeping the total trimming fraction small.
} 
Moreover, in the following remark, we make a comparison to the existing results specialized for linear regression. 
\begin{remark}
    In  linear regression, our analysis requires only a finite fourth moment of the covariates, which is weaker than the   $(4,2)$-hypercontractivity condition imposed in \cite{pensia2025robust}. In addition, our analysis does not require the covariates to have mean zero. Under the additional isotropic condition $\mathbb{E}\left[\boldsymbol{X}\boldsymbol{X}^{\top}\right]=I,$ as assumed in \cite{pensia2025robust}, the PAC bound for the empirical In-CVaR based estimator in \eqref{eq:pac_bound} specializes to
    \[
    O\left(
    \sigma\sqrt{q}
    \sqrt{
    \frac{q}{n}
    +\varepsilon
    +\frac{\log(1/\delta)}{n}
    }
    +
    \sigma_k
    \left(
    \varepsilon+\frac{\log(1/\delta)}{n}
    \right)^{1/2-1/k}
    \right),
    \]
    whose dependence on $n$, $\varepsilon$ and $\delta$  under fixed $q$   matches that of \eqref{eq:bound_pre} by \cite{pensia2025robust},   with the benefit of passing over  the  preliminary covariate filtering.
    When $q$ varies, our bound is of no larger order than that of \cite{pensia2025robust} whenever the following holds: 
    \[
    \frac{q\log q}{n}
    +\varepsilon
    +\frac{\log(1/\delta)}{n}
    \lesssim
    \left(\frac{\sigma_k}{\sigma}\right)^k q^{-k/2}.
    \]
\end{remark}

\begin{remark} 
    {In the isotropic linear setting, the In-CVaR based estimator can attain the near-optimal error bound via the postprocessing procedure as in \citet[Theorem 6.2]{pensia2025robust}. Specifically, the In-CVaR based estimator without preliminary covariate filtering can attain the estimation error $O(\sigma)$  as guaranteed by our  PAC analysis  whenever $\sigma_k/\sigma=O(1)$ and $q\bigl(q/n+\varepsilon+\log(1/\delta)/n\bigr)=O(1)$. Then under the additional $(4,2)$-hypercontractivity assumption of the covariates, the postprocessing procedure in the aforementioned reference can then be applied to the In-CVaR based estimator with the estimation error $O(\sigma)$, and  this yields the estimator with near-optimal error bound, without the need for  preliminary covariate filtering.}
\end{remark}

We next outline the proof sketch of Theorem \ref{thm:pacbound}, with the detailed proof deferred to Section \ref{app:thm:pac}. 
The idea is to establish a bound $t(n, \delta, \varepsilon)$ such that, with high probability, every $\BFtheta \in \Theta$ with $ \|\BFtheta - \BFtheta^*\| > t(n, \delta, \varepsilon)$ satisfies
\begin{equation}
\label{eq:thm3_in-cvar-high-prob-bound}
\operatorname{In-CVaR}_{\alpha}^{\beta} \left( \left(\, f(\boldsymbol{X}_{D_\varepsilon^{(n)}}, \BFtheta) - Y_{D_\varepsilon^{(n)}} \right)^2 \right) >  \operatorname{In-CVaR}_{\alpha}^{\beta} \left( \left(\, f(\boldsymbol{X}_{D_\varepsilon^{(n)}}, \BFtheta^*) - Y_{D_\varepsilon^{(n)}} \right)^2 \right). 
\end{equation}
The above inequality indicates that the In-CVaR based estimators $\hat{\mathcal{S}}_\alpha^\beta(D_\varepsilon^{(n)},f)$ which minimize the left-hand-side term,  must satisfy $\sup_{\BFtheta \in\hat{\mathcal{S}}_\alpha^\beta(D_\varepsilon^{(n)},f)}  \|\BFtheta - \BFtheta^*\| \leq  t(n, \delta, \varepsilon)$ and thus a valid PAC bound would follow.

To establish \eqref{eq:thm3_in-cvar-high-prob-bound}, in the first step ({\bf Step 1} in Section \ref{app:thm:pac}), by introducing  the excess loss 
$
d_{\BFtheta}(\boldsymbol{X},Y)
\triangleq
\bigl(Y-f(\boldsymbol{X},\BFtheta)\bigr)^2
-
\bigl(Y-f(\boldsymbol{X},\BFtheta^*)\bigr)^2
$ for $\BFtheta\in\Theta$,
and the clipping function  $
\psi_{H}(u) \triangleq \max\{-H, \min\{u, H\}\}$ for $H>0$, we manage to decompose the difference term as follows:  
\begin{equation}
\label{eq:objective_decompose}
\begin{aligned}
& (\beta - \alpha)\left(\operatorname{In-CVaR}_{\alpha}^{\beta} \left( \left(\, f(\boldsymbol{X}_{D_\varepsilon^{(n)}}, \BFtheta) - Y_{D_\varepsilon^{(n)}} \right)^2 \right) - \operatorname{In-CVaR}_{\alpha}^{\beta} \left( \left(\, f(\boldsymbol{X}_{D_\varepsilon^{(n)}}, \BFtheta^*) - Y_{D_\varepsilon^{(n)}} \right)^2 \right) \right)\\
& \qquad \geq \frac{1}{n}\sum_{i =1 }^n \psi_{H_\BFtheta}\left(d_{\BFtheta}(\boldsymbol{X}_i,Y_i)\right) -\left(3(\alpha+ 1 - \beta) + 2 \varepsilon\right)  H_\BFtheta
\\
& \qquad = \mathbb{E}\left[d_\BFtheta(\boldsymbol{X}, Y)\right] +\Big(\, \mathbb{E}\left[\psi_{H_\BFtheta}\left(d_\BFtheta(\boldsymbol{X}, Y)\right)\right] - \mathbb{E}\left[d_\BFtheta(\boldsymbol{X}, Y)\right]\, \Big)   \\
& \qquad \quad  + \left( \, \frac{1}{n}\sum_{i =1 }^n \psi_{H_\BFtheta}\left(d_{\BFtheta}(\boldsymbol{X}_i,Y_i)\right) - \mathbb{E}\left[\psi_{H_\BFtheta}\left(d_{\BFtheta}(\boldsymbol{X},Y)\right)\right] \, \right)  - \left(3(\alpha+ 1 - \beta) + 2 \varepsilon\right)  H_\BFtheta,
\end{aligned}
\end{equation}
where the clipping threshold $H_\theta$ is carefully chosen such that, with high probability, at least $\beta n$ observations
in the contaminated sample have losses at $\BFtheta^*$
no greater than $H_\theta$ by Lemma~\ref{lem:counting}.

Then we subsequently bound the first three terms in the right-hand side of  \eqref{eq:objective_decompose} with proof details provided in  Section \ref{app:thm:pac}. Specifically, in {\bf Step 2}, we show that the  first term  $\mathbb{E}\left[d_{\BFtheta}(\boldsymbol{X},Y)\right] = r_\BFtheta^2$, and for the second term, since 
$
| u - \psi_{H}(u)| \leq |u| \mathbb{I}(|u| > H) \leq {u^2}/{H}
$
for all $u \in \mathbb{R}$ and $H > 0$, it suffices to bound $\left\|d_{\BFtheta}\right\|_{L_2(D_0)}$ established in Lemma~\ref{lem:d_theta}. In {\bf Step 3},  we bound the third term by $s_n(q, \varepsilon, \delta, \BFtheta)$ defined in \eqref{eq:sn}
uniformly over $\BFtheta \in \Theta \backslash\{\BFtheta^*\}$ with high probability via the Rademacher complexity established in  Lemma \ref{lem:upperbound_emp}. In the final step we derive the claimed PAC bound by combining these bounds.

\begin{remark}
    We want to make a remark here regarding the classical approach for analyzing the difference term as in the literature. A natural alternative  is to decompose the difference between the two empirical In-CVaR objectives into the population excess risk and two deviation terms:
    \[
    \begin{aligned}
    & \operatorname{In-CVaR}_{\alpha}^{\beta} \left( \left(\, f(\boldsymbol{X}_{D_\varepsilon^{(n)}}, \BFtheta) - Y_{D_\varepsilon^{(n)}} \right)^2 \right) - \operatorname{In-CVaR}_{\alpha}^{\beta} \left( \left(\, f(\boldsymbol{X}_{D_\varepsilon^{(n)}}, \BFtheta^*) - Y_{D_\varepsilon^{(n)}} \right)^2 \right) \\
    &  \geq \mathbb{E} \left[d_\BFtheta(\boldsymbol{X}, Y)\right] - \left|\operatorname{In-CVaR}_{\alpha}^{\beta} \left( \left(\, f(\boldsymbol{X}_{D_\varepsilon^{(n)}}, \BFtheta) - Y_{D_\varepsilon^{(n)}} \right)^2 \right) - \mathbb{E} \left[\left(\, f(\boldsymbol{X}, \BFtheta) - Y\right)^2\right]\right| \\
    &\qquad -  \left|\mathbb{E} \left[\left(\, f(\boldsymbol{X}, \BFtheta^*) - Y\right)^2\right] - \operatorname{In-CVaR}_{\alpha}^{\beta} \left( \left(\, f(\boldsymbol{X}_{D_\varepsilon^{(n)}}, \BFtheta^*) - Y_{D_\varepsilon^{(n)}} \right)^2 \right)\right|.
    \end{aligned}
    \]
   The deviation terms could be controlled following \citet[Theorem 3.1]{oliveira2025trimmed}. 
   However, this approach treats the squared losses at $\BFtheta$ and $\BFtheta^*$ separately and therefore requires controlling  an additional expected empirical-process deviation associated with the singleton class $\{\xi^2\}$. Our analysis instead only works with the excess loss
    $
    d_\BFtheta (\boldsymbol{X}, Y) = \Delta_\BFtheta(\boldsymbol{X})^2 -2 \xi \Delta_\BFtheta(\boldsymbol{X})
    $
    without involving the $\xi^2$ term, and  controls the empirical-process deviation of the clipped and normalized classes that are particularly designed. This allows us to obtain a tighter bound in our setting.
\end{remark}

\subsection{Proof of Theorem \ref{thm:pacbound}}
\label{app:thm:pac}

We first present the following three lemmas which are used in the main proof. The proofs of the following lemmas are deferred to Appendix~\ref{app:lem}.  

\begin{lemma}
    \label{lem:counting}
    Suppose that Assumption~\ref{ass:p1} holds. For any $\tau > 0$, there exists $C>0$ such that if 
    \[
   \tau \left(\varepsilon + \frac{\log(1/\delta)}{n}\right) + \varepsilon\leq 1-\beta, 
    \]
    then, with probability at least $1-\delta$,
    \[
    \frac{1}{n}\sum_{i=1}^n \mathbb{I}\left({\xi_i^2} > C\sigma_k^2\left(\varepsilon + \frac{\log(1/\delta)}{n}\right)^{-2/k} \right) \leq 1-\beta - \varepsilon,
    \]
     {where $\xi_i \triangleq Y_i -f(\boldsymbol X_i, \BFtheta^*)$}.
\end{lemma}

\begin{lemma}
\label{lem:d_theta}
Suppose that Assumptions~\ref{ass:p1}--\ref{ass:p3} hold. Then, for any
$\BFtheta\in\Theta$,
\[
\left\|d_{\BFtheta}\right\|_{L_2(D_0)}
\leq
\frac{r_{\BFtheta}^2}{\mu^2}
\left\|\omega\right\|_{L_4(D_{0, \boldsymbol{X}})}^2
+
2\sigma r_{\BFtheta}.
\]
\end{lemma}

For a distribution $D\in \mathcal{D}$ and a class $\mathcal{F}$ of measurable functions from $\mathcal{Z}$ to $\mathbb{R}$, define the
expected empirical-process deviation by
\[
\operatorname{Emp}_n(\mathcal{F}, D) \triangleq \mathbb{E}_{\boldsymbol Z_{1: n} \stackrel{i . i . d.} {\sim} D}\left[\sup _{g \in \mathcal{F}}\left|\frac{1}{n} \sum_{i=1}^n  g(\boldsymbol{ Z}_i)-\mathbb{E}_{D} g(\boldsymbol Z)\right|\right]
\]
and define the Rademacher complexity
\[
\operatorname{Rad}_n(\mathcal{F}, D) \triangleq\mathbb{E}_{\boldsymbol Z_{1: n} \stackrel{i . i . d.} {\sim} D,\, \epsilon_{1:n} \stackrel{i . i . d.} {\sim} \operatorname{Unif}\{-1,1\} }\left[\sup _{g\in \mathcal{F}}\left|\frac{1}{n} \sum_{i=1}^n \epsilon_i g(\boldsymbol{Z}_i)\right|\right].
\]

\begin{lemma}
    \label{lem:upperbound_emp}
    Suppose that Assumptions \ref{ass:p1}--\ref{ass:p3} hold with $\omega$, and there exist $A, C>0$ such that for every $\eta>0$ and $D \in \mathcal{D}$,
    \begin{equation}
    \label{eq:covering}
    \log\mathcal{N}\left(\mathcal{G}, \|\bullet\|_{L_2(D_{\boldsymbol{X}})}, \eta \|\omega\|_{L_2(D_{\boldsymbol{X}})} \right) \leq C q \log \left(1+ \frac{A}{\eta }\right),
    \end{equation}
    where $\mathcal{G}$ is defined in \eqref{eq:G}.
    Then \[
    \operatorname{Emp}_n(\mathcal T_n(q,\varepsilon, \delta),D_0) \leq 2 \operatorname{Rad}_n(\mathcal T_n(q,\varepsilon, \delta),D_0) \lesssim
1,
    \]
    where $\mathcal T_n(q,\varepsilon,\delta)
\triangleq
\left\{
{d_\BFtheta(\bullet)}/{s_n(q,\varepsilon,\delta,\BFtheta)}
:\BFtheta\in\Theta\setminus\{\BFtheta^*\}
\right\}.$
\end{lemma}

We next prove Theorem~\ref{thm:pacbound} in the following four steps.

\textbf{Step 1:} Derive the decomposition inequality  \eqref{eq:objective_decompose} by the construction of the clipping threshold $H_\BFtheta$.

For each $\boldsymbol{\theta}\in\Theta$, define
\[
\begin{array}{l}
I_{\boldsymbol{\theta}}^\varepsilon \triangleq
\left\{
    \begin{array}{l} \{\pi(j)\}_{ 
    j=\alpha n+1,\ldots,\beta n}: \pi(\cdot) \mbox{ is a permutation of } [n] \mbox{ satisfying} \\ 
    \left(f(\boldsymbol{X}_{\pi(1)}^\varepsilon,\boldsymbol{\theta}) - Y_{\pi(1)}^\varepsilon\right)^2
    \leq \cdots \leq \left(f(\boldsymbol{X}_{\pi(n)}^\varepsilon,\boldsymbol{\theta}) - Y_{\pi(n)}^\varepsilon\right)^2
    \end{array}
\right\},  
\text{ and }\\[0.2in]
K_{\BFtheta}^\varepsilon \triangleq I_{\BFtheta}^\varepsilon \cap I_{\BFtheta^*}^\varepsilon \cap \{i: \boldsymbol{Z}_i^\varepsilon = \boldsymbol{Z}_i\}.
\end{array}\] 
Then, for $\BFtheta\in\Theta$,
\begin{equation}
\label{eq:pac_1}
\begin{aligned}
&  (\beta - \alpha)\left(\operatorname{In-CVaR}_{\alpha}^{\beta} \left( \left(\, f(\boldsymbol{X}_{D_\varepsilon^{(n)}}, \BFtheta) - Y_{D_\varepsilon^{(n)}} \right)^2 \right) - \operatorname{In-CVaR}_{\alpha}^{\beta} \left( \left(\, f(\boldsymbol{X}_{D_\varepsilon^{(n)}}, \BFtheta^*) - Y_{D_\varepsilon^{(n)}} \right)^2 \right) \right) \\
& \qquad =  \frac{1}{n}\sum_{i \in K_{\BFtheta}^\varepsilon } d_{\BFtheta}(\boldsymbol{X}_i, Y_i)  +  \frac{1}{n}\sum_{i \in I_{\BFtheta}^\varepsilon \backslash K_{\BFtheta}^\varepsilon }\bigl(Y_i^\varepsilon-f(\boldsymbol{X_i}^\varepsilon,\BFtheta)\bigr)^2 - \frac{1}{n}\sum_{i \in I_{\BFtheta^*}^\varepsilon \backslash K_{\BFtheta}^\varepsilon }\bigl(Y_i^\varepsilon-f(\boldsymbol{X_i}^\varepsilon,\BFtheta^*)\bigr)^2  \\
&  \qquad \geq \frac{1}{n}\sum_{i \in K_{\BFtheta}^\varepsilon }d_{\BFtheta}(\boldsymbol{X}_i, Y_i)  - \frac{1}{n}\left|I_{\BFtheta^*}^\varepsilon \backslash K_{\BFtheta}^\varepsilon \right| \max_{ i \in I_{\BFtheta^*}^\varepsilon }\bigl(Y_i^\varepsilon-f(\boldsymbol{X_i}^\varepsilon,\BFtheta^*)\bigr)^2.
\end{aligned}
\end{equation}
By Lemma \ref{lem:counting}, there exists $C$ such that with probability at least $1-\delta$,
    \[
    \sum_{i=1}^n \mathbb{I}\left(\xi_i^2 > C\sigma_k^2\left(\varepsilon + \frac{\log(1/\delta)}{n}\right)^{-2/k} \right) + \varepsilon n \leq (1-\beta)n.
    \]
For each $\BFtheta \in \Theta$, set
\[
 H_\BFtheta \triangleq \frac{C}{\varepsilon+ {\log(1/\delta)}/{n}} s_n(q, \varepsilon, \delta, \BFtheta),
\]
where 
\begin{equation}
\label{eq:sn}
\begin{aligned}
& s_n(q,\varepsilon,\delta, \BFtheta)\\
&  \triangleq \left(\sqrt{\frac{q}{n}} + \sqrt{\varepsilon + \frac{\log(1/\delta)}{n}
}\right)\left(\|\omega\|_{L_4(D_{0,\boldsymbol{X}})}^2r_\BFtheta^2+{\sigma \|\omega\|_{L_2(D_{0,\boldsymbol{X}})}}r_\BFtheta\right)+ \sigma_k^2\left(\varepsilon + \frac{\log(1/\delta)}{n}
\right)^{1-2/k}.
\end{aligned}
\end{equation}
Then
$
H_\BFtheta \geq C\sigma_k^2\left(\varepsilon + \frac{\log(1/\delta)}{n}\right)^{-2/k}. 
$
Hence, with probability at least $1-\delta$, uniformly over $\BFtheta \in \Theta$,
$
\max_{ i \in I_{\BFtheta^*}^\varepsilon }\bigl(Y_i^\varepsilon-f(\boldsymbol{X_i}^\varepsilon,\BFtheta^*)\bigr)^2 \leq H_\BFtheta.
$
Therefore
$
d_{\BFtheta}(\boldsymbol{X}_i, Y_i) \geq \psi_{H_\BFtheta}\left(d_{\BFtheta}(\boldsymbol{X}_i, Y_i)\right)
$
for $i \in K_{\BFtheta}^\varepsilon$.
Moreover,
\[
\begin{array}{l}
\left|[n]\backslash K_{\BFtheta}^\varepsilon\right| \leq \left|\left([n] \backslash I_{\BFtheta}^\varepsilon\right) \cup \left( [n]\backslash I_{\BFtheta^*}^\varepsilon\right) \cup \{i: \boldsymbol{Z}_i^\varepsilon \neq \boldsymbol{Z}_i\}\right| \leq \left(2(\alpha+ 1 - \beta) +  \varepsilon\right) n,\\[0.1in]{\left|I_{\BFtheta^*}^\varepsilon \backslash K_{\BFtheta}^\varepsilon \right| \leq \left|\left(I_{\BFtheta^*}^\varepsilon \backslash I_{\BFtheta}^\varepsilon\right) \cup \left(I_{\BFtheta^*}^\varepsilon \backslash \{i: \boldsymbol{Z}_i^\varepsilon = \boldsymbol{Z}_i\}\right) \right|
\leq n(\alpha+1-\beta+\varepsilon).} 
\end{array}
\]
It follows that
\[
\begin{aligned}
& \sum_{i \in K_{\BFtheta}^\varepsilon } d_{\BFtheta}(\boldsymbol{X}_i, Y_i) \geq \sum_{i \in K_{\BFtheta}^\varepsilon } \psi_{H_\BFtheta}\left(d_{\BFtheta}(\boldsymbol{X}_i, Y_i)\right) \geq \sum_{i =1 }^n \psi_{H_\BFtheta}\left(d_{\BFtheta}(\boldsymbol{X}_i, Y_i)\right) - \left(2(\alpha+ 1 - \beta) +  \varepsilon\right) n H_\BFtheta.
\end{aligned}
\]
Combining this inequality with \eqref{eq:pac_1} yields the following and thus \eqref{eq:objective_decompose} follows, 
\begin{equation}
\label{eq:pac_2}
\begin{aligned}
&  (\beta - \alpha)\left(\operatorname{In-CVaR}_{\alpha}^{\beta} \left( \left(\, f(\boldsymbol{X}_{D_\varepsilon^{(n)}}, \BFtheta) - Y_{D_\varepsilon^{(n)}} \right)^2 \right) - \operatorname{In-CVaR}_{\alpha}^{\beta} \left( \left(\, f(\boldsymbol{X}_{D_\varepsilon^{(n)}}, \BFtheta^*) - Y_{D_\varepsilon^{(n)}} \right)^2 \right) \right) \\
& \qquad \geq \frac{1}{n}\sum_{i =1 }^n \psi_{H_\BFtheta}\left(d_{\BFtheta}(\boldsymbol{X}_i, Y_i)\right) - \left(3(\alpha+ 1 - \beta) + 2 \varepsilon\right)  H_\BFtheta.
\end{aligned}
\end{equation}

    \textbf{Step 2:} Quantify the first two terms of \eqref{eq:objective_decompose} in terms of $r_{\theta}$.

For every $\BFtheta \in \Theta$, 
\begin{equation}
\label{eq:bound_vp}
\mathbb{E}\left[d_\BFtheta(\boldsymbol{X}, Y)\right] = \mathbb{E}\left[\left( f(\boldsymbol{X}, \BFtheta) - f(\boldsymbol{X}, \BFtheta^*)\right)^2 - 2 \xi \left(f(\boldsymbol{X}, \BFtheta) - f(\boldsymbol{X}, \BFtheta^*)\right)\right] = r_\BFtheta^2.
\end{equation}

For every $u \in \mathbb{R}$ and $H > 0$, 
$
| u - \psi_{H}(u)| \leq |u| \mathbb{I}(|u| > H) \leq {u^2}/{H}.
$
Therefore, by Lemma \ref{lem:d_theta},
\begin{equation}
\label{eq:pac_3}
\begin{aligned}
& \mathbb{E} \left[\left|d_{\BFtheta}(\boldsymbol{X}, Y) - \psi_{H_\BFtheta}(d_\BFtheta(\boldsymbol{X}, Y))\right|\right] \lesssim \frac{\left(\|\omega\|_{L_4(D_{0,\boldsymbol{X}})}^2 r_\BFtheta^2 + 2 \sigma r_\BFtheta\right)^2}{H_\BFtheta} \\
& \qquad \lesssim \left(\|\omega\|_{L_4(D_{0,\boldsymbol{X}})}^2 r_\BFtheta^2 +  \sigma \|\omega\|_{L_2(D_{0,\boldsymbol{X}})}r_\BFtheta\right)\sqrt{\varepsilon+ \frac{\log(1/\delta)}{n}} \leq s_n(q,\varepsilon,\delta,\BFtheta),
\end{aligned}
\end{equation}
where the second inequality uses $ \|\omega\|_{L_2(D_{0,\boldsymbol{X}})} = \Omega(1)$ from Assumptions \ref{ass:p2} and \ref{ass:p3}.

    \textbf{Step 3:}  Bound the third term   of \eqref{eq:objective_decompose},  $\left|\frac{1}{n}\sum_{i =1 }^n \psi_{H_\BFtheta}\left(d_{\BFtheta}(\boldsymbol{X}_i, Y_i)\right) - \mathbb{E}\left[\psi_{H_\BFtheta}\left(d_\BFtheta(\boldsymbol{X}, Y)\right)\right]\right|$.

 Since the clipping map $\psi_{H_\BFtheta}$ is $1$-Lipschitz, by  \citet[Lemma 11.4 and Theorem 11.6]{ste2013} and  Lemma \ref{lem:upperbound_emp}, we have 
\begin{equation}
\label{eq:bound_emp_1}
\begin{aligned}
& \mathbb{E} \left[\sup_{\theta \in \Theta} \left|\frac{1}{n}\sum_{i=1}^n \frac{\psi_{H_\BFtheta}(d_{\BFtheta}(\boldsymbol{X}_i, Y_i)) }{s_n(q,\varepsilon,\delta, \BFtheta)} - \mathbb{E} \left[\frac{\psi_{H_\BFtheta}(d_{\BFtheta}(\boldsymbol{X}, Y)) }{s_n(q,\varepsilon,\delta, \BFtheta)}\right]\right|\right]\\
& \qquad \leq 2\mathbb{E}_{\boldsymbol Z_{1: n} \stackrel{i . i . d.} {\sim} D_0, \epsilon_{1:n} \stackrel{i . i . d.} {\sim} \operatorname{Unif}\{-1,1\} }\left[\sup _{\BFtheta \in\Theta}\left|\frac{1}{n} \sum_{i=1}^n \epsilon_i \frac{\psi_{H_\BFtheta}(d_{\BFtheta}(\boldsymbol{X}_i, Y_i)) }{s_n(q,\varepsilon,\delta, \BFtheta)}\right|\right]\\
& \qquad\leq 2\operatorname{Rad}_n(\mathcal T_n(q,\varepsilon, \delta),D_0)\lesssim 1.
\end{aligned}
\end{equation}
By Bousquet’s version of Talagrand’s concentration inequality \citep{bousquet2002bennett}, with probability at least $1-\delta$,
\begin{equation}
\label{eq:psih_emp}
\begin{aligned}
& \sup_{\theta \in \Theta} \left|\frac{1}{n}\sum_{i=1}^n \frac{\psi_{H_\BFtheta}(d_{\BFtheta}(\boldsymbol{X}_i, Y_i)) }{s_n(q,\varepsilon,\delta, \BFtheta)} - \mathbb{E} \left[\frac{\psi_{H_\BFtheta}(d_{\BFtheta}(\boldsymbol{X}, Y)) }{s_n(q,\varepsilon,\delta, \BFtheta)}\right]\right| \\
& \quad \leq \mathbb{E} \left[\sup_{\theta \in \Theta} \left|\frac{1}{n}\sum_{i=1}^n \frac{\psi_{H_\BFtheta}(d_{\BFtheta}(\boldsymbol{X}_i, Y_i)) }{s_n(q,\varepsilon,\delta, \BFtheta)} - \mathbb{E} \left[\frac{\psi_{H_\BFtheta}(d_{\BFtheta}(\boldsymbol{X}, Y)) }{s_n(q,\varepsilon,\delta, \BFtheta)}\right]\right|\right] \\
& \qquad \quad + \sqrt{\frac{2\log(1/\delta)}{n} v} + \frac{2H_\BFtheta \log(1/\delta)}{3ns_n(q,\varepsilon,\delta, \BFtheta)},
\end{aligned}
\end{equation}
where
\[
\begin{aligned}
&v = 4 \frac{H_\BFtheta}{s_n(q,\varepsilon,\delta, \BFtheta)} \mathbb{E} \left[\sup_{\theta \in \Theta} \left|\frac{1}{n}\sum_{i=1}^n \frac{\psi_{H_\BFtheta}(d_{\BFtheta}(\boldsymbol{X}_i, Y_i)) }{s_n(q,\varepsilon,\delta, \BFtheta)} - \mathbb{E} \left[\frac{\psi_{H_\BFtheta}(d_{\BFtheta}(\boldsymbol{X}, Y)) }{s_n(q,\varepsilon,\delta, \BFtheta)}\right]\right|\right]\\
& \qquad + \sup_{\BFtheta\in\Theta} \mathbb{E} \left[\left(\frac{\psi_{H_\BFtheta}(d_{\BFtheta}(\boldsymbol{X}, Y)) }{s_n(q,\varepsilon,\delta, \BFtheta)} - \mathbb{E}\left[\frac{\psi_{H_\BFtheta}(d_{\BFtheta}(\boldsymbol{X}, Y)) }{s_n(q,\varepsilon,\delta, \BFtheta)}\right]\right)^2\right].
\end{aligned}
\]
Moreover, by the inequalities $\sqrt{a+ b} \leq \sqrt{a} + \sqrt{b}$ and $\sqrt{ab} \leq (a+b)/2$ for $a, b>0$,
\[
\begin{aligned}
& \sqrt{\frac{2\log(1/\delta)}{n} v}   \leq  \frac{4 H_\BFtheta\log(1/\delta)}{ns_n(q,\varepsilon,\delta, \BFtheta)}  + \frac{1}{2}\mathbb{E} \left[\sup_{\theta \in \Theta} \left|\frac{1}{n}\sum_{i=1}^n \frac{\psi_{H_\BFtheta}(d_{\BFtheta}(\boldsymbol{X}_i, Y_i)) }{s_n(q,\varepsilon,\delta, \BFtheta)} - \mathbb{E} \left[\frac{\psi_{H_\BFtheta}(d_{\BFtheta}(\boldsymbol{X}, Y)) }{s_n(q,\varepsilon,\delta, \BFtheta)}\right]\right|\right] \\
& \qquad \quad + \sqrt{\frac{2\log(1/\delta)}{n}\sup_{\BFtheta\in\Theta} \mathbb{E} \left[\left(\frac{d_{\BFtheta}(\boldsymbol{X}, Y) }{s_n(q,\varepsilon,\delta, \BFtheta)} \right)^2\right]} \lesssim 1,
\end{aligned}
\]
where the last inequality follows by the choice of $H_\BFtheta$, inequality \eqref{eq:bound_emp_1}, and Lemma \ref{lem:d_theta}.
Combining the above inequality with \eqref{eq:psih_emp} yields
\[
\sup_{\theta \in \Theta} \left|\frac{1}{n}\sum_{i=1}^n \frac{\psi_{H_\BFtheta}(d_{\BFtheta}(\boldsymbol{X}_i, Y_i)) }{s_n(q,\varepsilon,\delta, \BFtheta)} - \mathbb{E} \left[\frac{\psi_{H_\BFtheta}(d_{\BFtheta}(\boldsymbol{X}, Y)) }{s_n(q,\varepsilon,\delta, \BFtheta)}\right]\right| \lesssim 1.
\]
Therefore, with probability at least $1-\delta$, 
\begin{equation} 
\label{eq:pac_4}
\left|\frac{1}{n}\sum_{i=1}^n \psi_{H_\BFtheta}(d_{\BFtheta}(\boldsymbol{X}_i, Y_i))  - \mathbb{E} \left[\psi_{H_\BFtheta}(d_{\BFtheta}(\boldsymbol{X}, Y)) \right]\right| \lesssim s_n(q, \varepsilon, \delta, \BFtheta)
\end{equation}
uniformly over $\BFtheta \in \Theta \backslash\{\BFtheta^*\}.$

    \textbf{Step 4:} Derive the PAC bound of the In-CVaR based estimator by combining the bounds.

Combining \eqref{eq:pac_2}--\eqref{eq:pac_3} and \eqref{eq:pac_4}, there exists $\bar C >0$ such that, with probability at least $1 - O(\delta)$, uniformly over $\BFtheta\in\Theta \backslash \{\BFtheta^*\}$,
\[
\begin{aligned}
& (\beta - \alpha)\left(\operatorname{In-CVaR}_{\alpha}^{\beta} \left( \left(\, f(\boldsymbol{X}_{D_\varepsilon^{(n)}}, \BFtheta) - Y_{D_\varepsilon^{(n)}} \right)^2 \right) - \operatorname{In-CVaR}_{\alpha}^{\beta} \left( \left(\, f(\boldsymbol{X}_{D_\varepsilon^{(n)}}, \BFtheta^*) - Y_{D_\varepsilon^{(n)}} \right)^2 \right) \right)\\
&  \qquad \geq  r_\BFtheta^2 - \left|\frac{1}{n}\sum_{i =1 }^n \psi_{H_\BFtheta}\left(d_{\BFtheta}(\boldsymbol{X}_i, Y_i)\right) - \mathbb{E}\left[\psi_{H_\BFtheta}\left(d_\BFtheta(\boldsymbol{X}, Y)\right)\right]  \right|  \\
&  \qquad \quad -\left| \mathbb{E}\left[\psi_{H_\BFtheta}\left(d_\BFtheta(\boldsymbol{X}, Y)\right)\right] - \mathbb{E}\left[d_\BFtheta(\boldsymbol{X}, Y)\right]\right| - \left(3(\alpha+ 1 - \beta) + 2 \varepsilon\right) H_\BFtheta\\
&  \qquad \geq r_\BFtheta^2 - \bar C  \left(\sqrt{\frac{q}{n}} + \sqrt{\varepsilon + \frac{\log(1/\delta)}{n}
}\right)\left(\|\omega\|_{L_4(D_{0,\boldsymbol{X}})}^2r_\BFtheta^2 + {\sigma \|\omega\|_{L_2(D_{0,\boldsymbol{X}})}}r_\BFtheta\right)  \\
&  \qquad \quad - \bar C\sigma_k^2\left(\varepsilon + \frac{\log(1/\delta)}{n}
\right)^{1-2/k}.
\end{aligned}
\]
Choose $\kappa > 0$ sufficiently small  and $M>0$  sufficiently large so that
$
\bar C\left(\sqrt{{1}/{M}} + \sqrt{\kappa+ {1}/{M}}\right) < \frac{1}{2}.
$
Then, whenever $\varepsilon \leq \kappa/\|\omega\|_{L_4(D_{0,\boldsymbol{X}})}^4$ and $n \geq M\|\omega\|_{L_4(D_{0,\boldsymbol{X}})}^4(q+ \log(1/\delta))$, we have
\begin{equation}
\label{eq:pac_lower}
\begin{aligned}
& (\beta - \alpha)\left(\operatorname{In-CVaR}_{\alpha}^{\beta} \left( \left(\, f(\boldsymbol{X}_{D_\varepsilon^{(n)}}, \BFtheta) - Y_{D_\varepsilon^{(n)}} \right)^2 \right) - \operatorname{In-CVaR}_{\alpha}^{\beta} \left( \left(\, f(\boldsymbol{X}_{D_\varepsilon^{(n)}}, \BFtheta^*) - Y_{D_\varepsilon^{(n)}} \right)^2 \right) \right)\\
&\qquad \geq \frac{1}{2}r_\BFtheta^2  - \bar C \sigma \|\omega\|_{L_2(D_{0,\boldsymbol{X}})}r_\BFtheta \left(\sqrt{\frac{q}{n}} + {\sqrt{\varepsilon + \frac{\log(1/\delta)}{n}}} \right) - \bar C \sigma_k^2 \left(\varepsilon + \frac{\log(1/\delta)}{n}\right)^{1-2/k}.
\end{aligned}
\end{equation}
The right-hand side of \eqref{eq:pac_lower} is strictly positive whenever
\[
r_\BFtheta > 2\bar C  \sigma \|\omega\|_{L_2(D_{0,\boldsymbol{X}})}\left(\sqrt{\frac{q}{n}} + {\sqrt{\varepsilon + \frac{\log(1/\delta)}{n}}}\right) + \sigma_k\sqrt{2 \bar C\left(\varepsilon + \frac{\log(1/\delta)}{n}\right)^{1-2/k}}.
\]
Therefore, with probability at least $1-O(\delta)$,    
\[
\sup_{\BFtheta \in\hat{\mathcal{S}}_\alpha^\beta(D_\varepsilon^{(n)},f)}  r_{\BFtheta} \lesssim \sigma\|\omega\|_{L_2(D_{0,\boldsymbol{X}})}\sqrt{\frac{q}{n} + \varepsilon + \frac{\log(1/\delta)}{n}}+  \sigma_k \left(\varepsilon + \frac{\log(1/\delta)}{n}\right)^{1/2-1/k},
\]
which completes the proof by Assumption \ref{ass:p3}.

\section{Numerical Experiments}
\label{sec:exp}

We conduct experiments on synthetic data and real data to examine the robustness properties established in the preceding sections. Sections \ref{subsection:exp_QRAO} and \ref{sec:exp:pac} report results under distributional perturbations and sample contamination, respectively. All experiments are implemented in Python on macOS Sonoma 14.2.1, equipped with an Apple M1 Pro chip and 16 GB of RAM. The convex subproblems arising in the difference-of-convex algorithm (DCA) are solved using the L-BFGS-B solver from SciPy.

\subsection{Distributional perturbation}
\label{subsection:exp_QRAO}

We consider the In-CVaR based piecewise affine regression
\[
    f(x, {\BFtheta}) \triangleq
    \displaystyle{\max_{i=1, 2}} \{a_i x + c_i\} -\displaystyle{\max_{j=0,1, 2}}\{b_j x +d_j\} ,
\]
where  $\BFtheta \triangleq \left(\{(a_i, c_i)\}_{i=1}^2, \{(b_j, d_j)\}_{j=1}^2\right) \in \mathbb{R}^8$ and  $(b_0, d_0) = (0,0)$.
Let $D$ represent the empirical distribution of feature-response pairs whose generation will be explained later. According to the analysis in \cite{liu2022risk}, $\min_{\BFtheta \in \mathbb{R}^{8}} \mbox{In-CVaR}_{\alpha}^\beta \left(\,\abs{f(X_D, {\BFtheta}) - Y_D}\right)$ admits a difference-of-convex (dc) representation, which enables DCA to obtain the In-CVaR based estimator efficiently. For comparison, we compute the ERM based estimators and CVaR based estimators via DCA. We also consider
M-estimators   with the Huber loss. Because the  Huber loss is a  convex function, the corresponding empirical parameter estimation problems can be represented as a class of convex-composite-dc problems, which was studied in \cite{Cui_composite} by the majorization-minimization algorithm.

Consider the feature and response variables $(X,Y)$ with  $X, \xi \sim N(0,1)$, and 
\[
Y= \max\{-X + 1, 2\} - \max \{X - 1, -2X, 0\}+ 0.05\xi.
\]
We generate 200 independent samples $\{(x_{D_0}^{(j)}, y_{D_0}^{(j)})\}_{j=1}^{200}$ from the above model and take the corresponding empirical distribution as the nominal distribution $D_0$.  For each  $k \in \mathbb{N}_+$, we construct a perturbed distribution $D_k$ as follows. We generate 200 design points $\{x_{k}^{(j)}\}_{j=1}^{200}$ from 0 to 1990 with step size 10, and set
\[
y_{k}^{(j)} = k\big(\max\big\{x_{k}^{(j)} + 2, 3x_{k}^{(j)}-4\big\} - \max\big\{2 x_{k}^{(j)} - 1, 4x_{k}^{(j)}+1,0\big\}\big)+0.05 \,  {\xi_k^{(j)}},
\]
where $\xi_k^{(j)} \sim N(0,1)$. Let $G_k$ be the empirical distribution of $\{\boldsymbol{z}_{k}^{(j)}\}_{j=1}^{200} = \{(x_{k}^{(j)}, y_{k}^{(j)}) \}_{j=1}^{200}$ and define the random vector $\boldsymbol{Z}_{D_k} \in \mathbb R^2$ corresponding to the perturbed distribution $D_k$ following
\[
\boldsymbol{Z}_{D_k} = \left\{
\begin{array}{ll}
    \boldsymbol{Z}_{D_0} + \frac{1}{k} \BFdelta, & \text{with probability } 1 - \frac{1}{k}; \\
    \boldsymbol{Z}_{G_k} + \frac{1}{k} \BFdelta, & \text{with probability } \frac{1}{k}, 
\end{array}
\right.
\]
where $\BFdelta$ is uniformly distributed within the unit ball in $\mathbb{R}^2$. By Strassen’s theorem \cite[Theorem 2.13]{huber2009robust}, we have $d_{\text{Prok}}(D_k, D_0) \leq 1/k$. For each $k$, we draw $n=1000$ independent  samples $\{\boldsymbol{z}_{D_k}^{(i)}\}_{i=1}^{1000}$ from $D_k$. 
Figure \ref{fig:piecewise_perturbation_log10} illustrates the performance of the In-CVaR based estimator, the ERM based estimator, the CVaR based estimator, and the Huber estimator via the log estimation errors (e.g., $\log(\|\hat\BFtheta_\alpha^\beta(D_k^{(n)},f)-\BFtheta^*\|)$,
where $\BFtheta^*$ denotes the true parameter).  As $k$ increases, In-CVaR based estimator remains close to the true parameter, while the others exhibit large deviations.
\begin{figure}[h]
    \FIGURE
    {\includegraphics[width=0.6\textwidth]{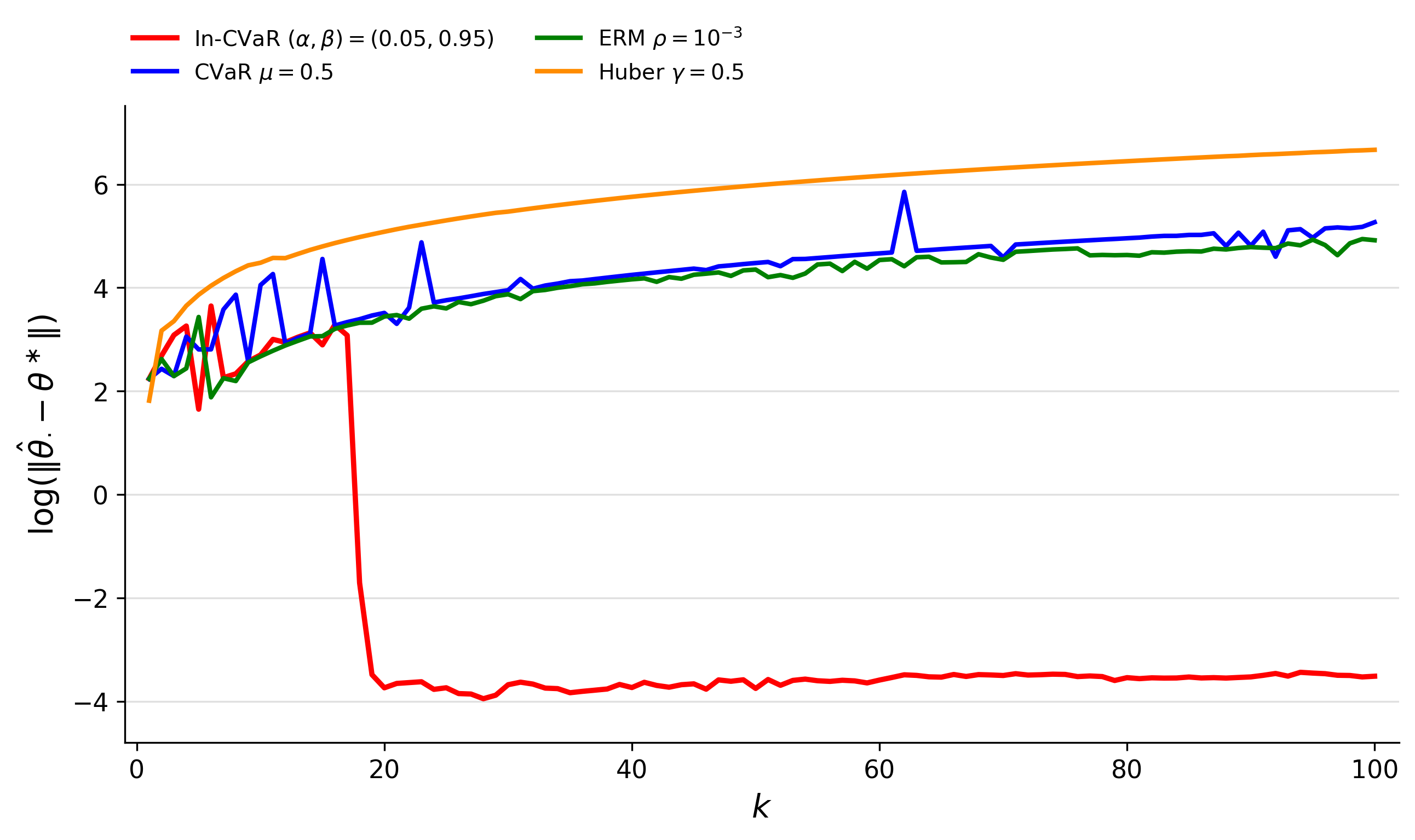}} 
    {
    Estimation error under distributional perturbation. 
    \label{fig:piecewise_perturbation_log10}}
    {$\rho$ denotes the regularization parameter for the squared $\ell_2$-norm penalty in ERM, $\mu$ denotes the CVaR threshold, and $\gamma$ denotes the Huber loss parameter.}
\end{figure}

\subsection{Sample contamination}
\label{sec:exp:pac}
In this subsection, we show the performance of the In-CVaR based estimator and some related estimators under sample contamination for linear regression and exponential regression models.

\subsubsection{Linear regression} 

We assess the empirical performance of In-CVaR based linear regression under adversarial sample contamination using a setting adapted from \citet{pensia2025robust}. We run 50,000 independent replications of a 40-dimensional linear regression experiment with sample size $n = 200$. In each replication, we generate $Y_i=\boldsymbol{X}_i^\top\BFtheta^*+\xi_i$ where $\BFtheta^*$ is sampled from the unit sphere, the covariate coordinates and additive noise follow symmetrized Pareto distributions with density $f_Z(z)\propto(1+|z|)^{-(1+\nu)}$ and tail parameters $\nu=4$ and $2$, respectively. We adversarially contaminate 20 observations, corresponding to $10\%$ of the sample:
\begin{enumerate}
    \item We replace the covariates of 10 randomly selected observations with $10\boldsymbol{1}$, where $\boldsymbol{1}$ is the all-ones vector.
    \item We select 10 additional observations and set their responses, along with the responses of the 10 observations selected in the first step, as 200.
\end{enumerate}

We compare Huber regression with In-CVaR regression under $\ell_1$ and $\ell_2$ losses. Each method is evaluated with and without covariate filtering. Following \citet{pensia2025robust}, the filtering procedure removes 30 observations. For Huber regression, as in \cite{pensia2025robust},
we set $\gamma=0.5$ and run gradient descent with a line-search procedure.
For In-CVaR regression, we set $(\alpha, \beta) = (0.01, 0.85)$ and we optimize via DCA. We initialize all of our estimators at the same point, sampled independently from a sphere of unit norm.

\begin{figure}[h]
\FIGURE
    {
    \subcaptionbox{Estimation errors over 50,000 replications of 40-dimensional linear regression with 10\% adversarial contamination ($n=200$).  \label{fig:pac}}
    {\includegraphics[width=0.48\textwidth]{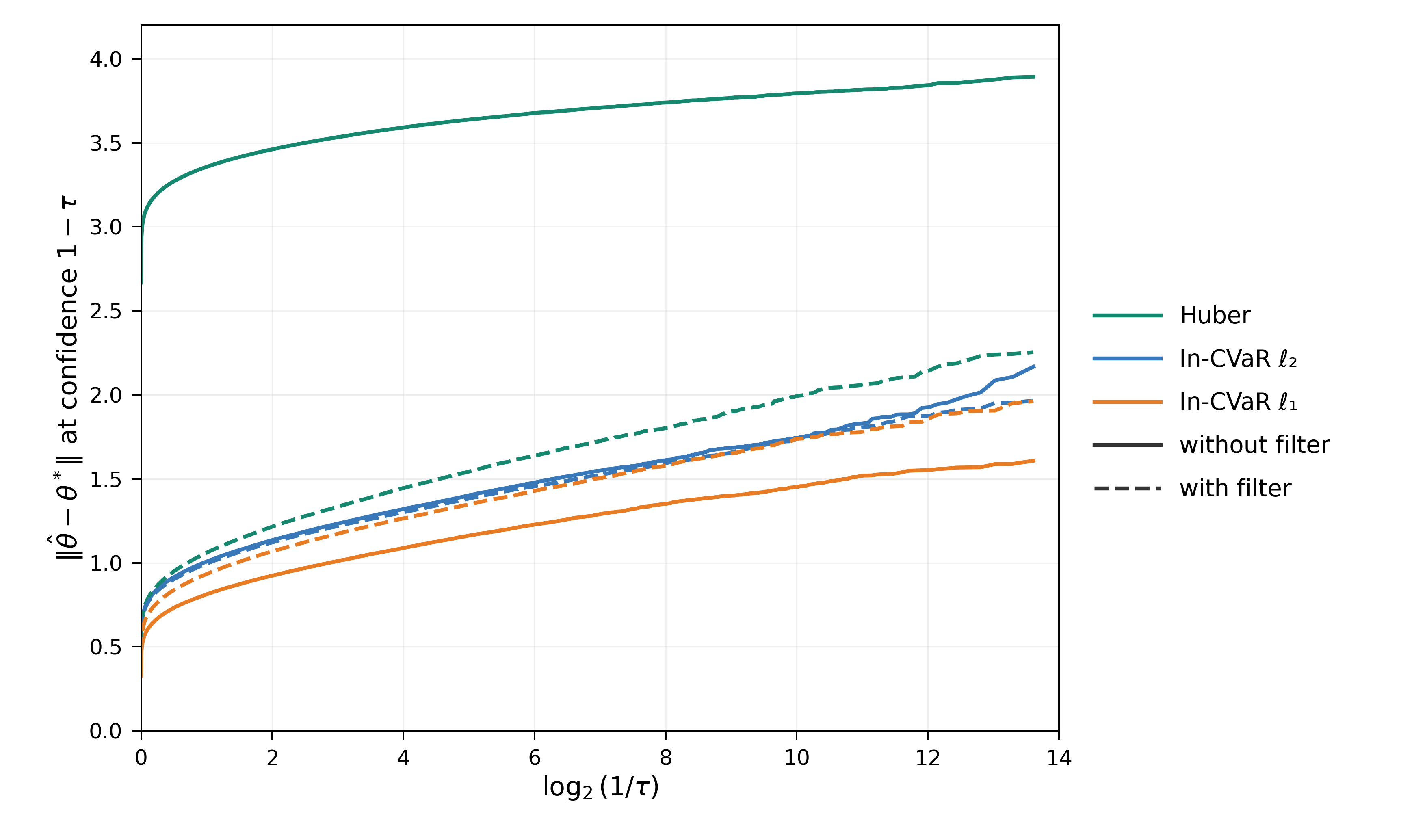}}
    \hfill
    \subcaptionbox{Histogram of response data in Yacht Hydrodynamics dataset. \label{fig:histogram}}
    {\includegraphics[width=0.48\textwidth]{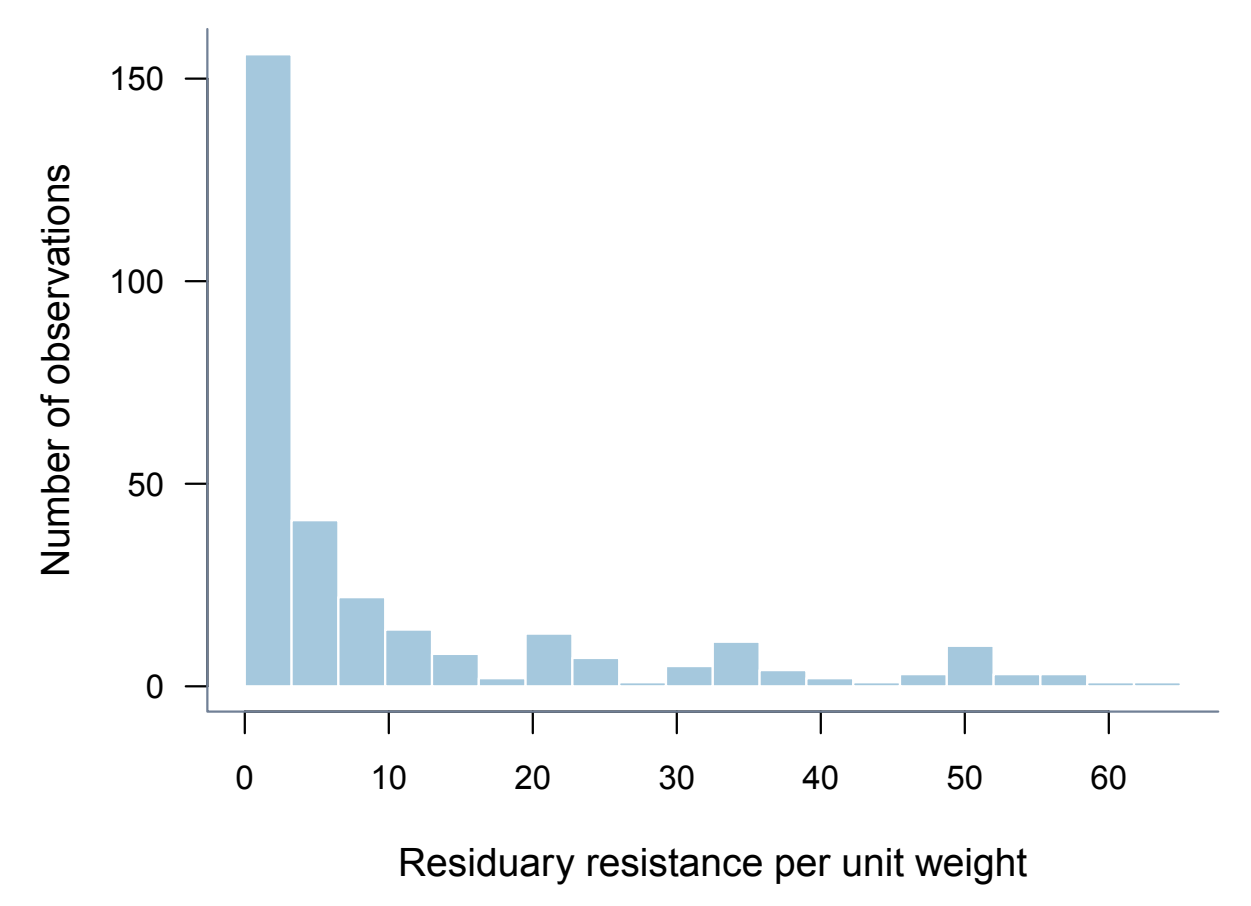}}
    }
    {}
    {}
\end{figure}

Figure~\ref{fig:pac} shows lower estimation errors for the In-CVaR based estimators than for Huber regression in this experiment. Among these methods,  In-CVaR based regression under $\ell_1$ loss without filtering performs best. Filtering substantially reduces the error of Huber regression. However, filtering does little benefit for In-CVaR based regression; in particular, it only provides a slight improvement under $\ell_2$ loss, and even increases the error under  $\ell_1$ loss. 
This result exhibits a different phenomenon from the numerical results in \citet{pensia2025robust}, where Huber regression outperforms LTS. The discrepancy may reflect differences in optimization procedures  {as  the LTS estimator is computed by alternating methods in \citet{pensia2025robust}, whereas the present experiments compute  In-CVaR based regression using DCA with computational benefits.}

\subsubsection{Exponential regression on real data}

We use the Yacht Hydrodynamics dataset \footnote{\texttt{https://archive.ics.uci.edu/dataset/243/yacht+hydrodynamics}}
 to examine robustness under adversarial contamination in nonlinear regression. The dataset contains 308 observations, each with six numerical covariates (Longitudinal position of the center of buoyancy, Prismatic coefficient, Length-displacement ratio, Beam-draught ratio, Length-beam ratio, Froude number) and a response (Residuary resistance per unit weight of displacement). As shown in Figure \ref{fig:histogram},
the response is strictly positive and strongly right-skewed
with its marginal Box--Cox estimate of $0.072$, which motivates the regression model
$f(\boldsymbol{x},\boldsymbol{\theta})
=\exp(\theta_0+\boldsymbol{\theta}_1^\top\boldsymbol{x})$.

We randomly assign 184 observations to the training set and 124 to the test set. We contaminate 0\%, 5\%, and 10\% of the training observations, with the test set remaining unchanged. For each replaced observation, we draw the $j$-th covariate independently from
\[
\widetilde X_j\sim N(\bar X_j+5\hat\sigma_j,\hat \sigma_j^2).
\]
Here $\bar X_j$ and
$\hat{\sigma}_j^2$
are the empirical mean and variance of the $j$-th covariate in the clean training set. We then set
$\widetilde Y=10 \sum_{j=1}^6 \widetilde X_j+\xi$,
where $\xi\sim N(0,1)$.

At each contamination level, we fit In-CVaR ($\alpha=0.05$, $\beta=0.9$), CVaR ($\mu=0.5$), and ERM using either $\ell_1$ or $\ell_2$ loss. We evaluate all fitted models on the same uncontaminated test set using
\[
\operatorname{MAE}
=\frac{1}{N_{\mathrm{test}}}
\sum_{i=1}^{N_{\mathrm{test}}}
\left|f(\boldsymbol x_{i},\hat{\boldsymbol\theta})-y_i\right|.
\]
Table~\ref{tab:yacht_contamination_mae_combined} reports the test MAEs. With uncontaminated training data, the $\ell_2$ based ERM performs best. At both contamination levels, the $\ell_1$ loss gives a lower MAE than the $\ell_2$ loss for each method and the In-CVaR test MAEs are lower compared with the CVaR and ERM test MAEs.
\begin{table}[htbp]
\centering
\small
\caption{Performance for In-CVaR, CVaR and ERM based exponential regression on the Yacht Hydrodynamics data.}
\label{tab:yacht_contamination_mae_combined}
\begin{tabular}{ccccc}
\toprule
\multirow[c]{3}{*}{Method}
& \multirow[c]{3}{*}{Training loss}
& \multicolumn{3}{c}{Test MAE by contamination ratio} \\
\cmidrule(lr){3-5}
& &  0\% & 5\% & 10\% \\
\midrule
\multirow[c]{2}{*}{In-CVaR ($\alpha=0.05$, $\beta=0.9$)}
    & $\ell_1$ & 1.258  & \textbf{2.227} & \textbf{2.347} \\
    & $\ell_2$ & 1.282 & 2.383 & 2.688 \\
\addlinespace
\multirow[c]{2}{*}{CVaR ($\mu=0.5$)}
    & $\ell_1$  &0.691 & 5.343 & 11.305 \\
    & $\ell_2$  &0.687 & 13.011 & 14.573 \\
\addlinespace
\multirow[c]{3}{*}{ERM}
    & $\ell_1$  $(\rho=10^{-3})$ &0.718 & 5.967 & 9.497 \\
    & $\ell_2 (\rho=10^{-3})$  &\textbf{0.685} & 12.055 & 13.155 \\
    & Huber ($\gamma = 0.5)$ & 0.689 & 6.464 & 9.436 \\
\bottomrule
\end{tabular}
\end{table}

\section{Conclusion}

\label{sec:conclusion}

In this paper, we investigate the robustness of In-CVaR based regression under distributional perturbation and sample contamination. For distributional perturbation, we show that the In-CVaR based estimator is qualitatively robust if and only if it trims the largest portion of  losses. For sample contamination, we derive a PAC bound for the In-CVaR based estimator.  The theoretical results are presented in Sections~\ref{sec:qrao} and~\ref{sec:PAC} and are supported by numerical experiments in Section \ref{sec:exp}. This leads to several  promising  directions for future work. One is  to further investigate quantitative robustness under perturbation, for instance through the influence function as in \cite{cont2010robustness} and \cite{zhang2024statistical}. It is also worth investigating the PAC bound for $\ell_1$ based In-CVaR regression given its favorable performance in our numerical experiment. 

\bibliographystyle{informs2014} 
\bibliography{main}

@book{huber2009robust,
  title={Robust Statistics},
  author={Huber, Peter J and Ronchetti, Elvezio M},
  year={2009},
  publisher={John Wiley \& Sons}
}

@article{hossjer1994rank,
  title={Rank-based estimates in the linear model with high breakdown point},
  author={H{\"o}ssjer, Ola},
  journal={Journal of the American Statistical Association},
  volume={89},
  number={425},
  pages={149--158},
  year={1994},
  publisher={Taylor \& Francis}
}

@article{rousseeuw1984least,
  title={Least median of squares regression},
  author={Rousseeuw, Peter J},
  journal={Journal of the American Statistical Association},
  volume={79},
  number={388},
  pages={871--880},
  year={1984},
  publisher={Taylor \& Francis}
}

@article{cont2010robustness,
  title={Robustness and sensitivity analysis of risk measurement procedures},
  author={Cont, Rama and Deguest, Romain and Scandolo, Giacomo},
  journal={Quantitative Finance},
  volume={10},
  number={6},
  pages={593--606},
  year={2010},
  publisher={Taylor \& Francis}
}

@article{kratschmer2012qualitative,
  title={Qualitative and infinitesimal robustness of tail-dependent statistical functionals},
  author={Kr{\"a}tschmer, Volker and Schied, Alexander and Z{\"a}hle, Henryk},
  journal={Journal of Multivariate Analysis},
  volume={103},
  number={1},
  pages={35--47},
  year={2012},
  publisher={Elsevier}
}

@article{zhang2024statistical,
  title={Statistical robustness of kernel learning estimator with respect to data perturbation},
  author={Zhang, Sainan and Xu, Huifu and Sun, Hailin},
  journal={arXiv preprint arXiv:2406.10555},
  year={2024}
}

@article{guo2023statistical,
  title={Statistical robustness of empirical risks in machine learning},
  author={Guo, Shaoyan and Xu, Huifu and Zhang, Liwei},
  journal={Journal of Machine Learning Research},
  volume={24},
  number={125},
  pages={1--38},
  year={2023}
}

@book{shapiro2021lectures,
  title={Lectures on Stochastic Programming: Modeling and Theory},
  author={Shapiro, Alexander and Dentcheva, Darinka and Ruszczynski, Andrzej},
  year={2021},
  publisher={SIAM}
}

@article{laguel2021superquantiles,
  title={Superquantiles at work: Machine learning applications and efficient subgradient computation},
  author={Laguel, Yassine and Pillutla, Krishna and Malick, J{\'e}r{\^o}me and Harchaoui, Zaid},
  journal={Set-Valued and Variational Analysis},
  volume={29},
  number={4},
  pages={967--996},
  year={2021},
  publisher={Springer}
}

@article{liu2022risk,
  title={Risk-based robust statistical learning by stochastic difference-of-convex value-function optimization},
  author={Liu, Junyi and Pang, Jong-Shi},
  journal={Operations Research},
  volume={71},
  number={2},
  pages={397--414},
  year={2022},
  publisher={INFORMS}
}

@book{hampel1968contributions,
  title={Contributions to the Theory of Robust Estimation},
  author={Hampel, Frank Rudolf},
  year={1968},
  publisher={University of California, Berkeley}
}

@article{tsyurmasto2013support,
  title={Support vector classification with positive homogeneous risk functionals},
  author={Tsyurmasto, Peter and Gotoh, J and Uryasev, S},
  journal={University of Florida, Gainesville, FL},
  volume={4},
  pages={2013},
  year={2013}
}

@article{fujiwara2017dc,
  title={{DC} algorithm for extended robust support vector machine},
  author={Fujiwara, Shuhei and Takeda, Akiko and Kanamori, Takafumi},
  journal={Neural Computation},
  volume={29},
  number={5},
  pages={1406--1438},
  year={2017},
  publisher={MIT Press}
}

@article{kanamori2017breakdown,
  title={Breakdown point of robust support vector machines},
  author={Kanamori, Takafumi and Fujiwara, Shuhei and Takeda, Akiko},
  journal={Entropy},
  volume={19},
  number={2},
  pages={83},
  year={2017},
  publisher={MDPI}
}

@article{hu2020learning,
  title={Learning by minimizing the sum of ranked range},
  author={Hu, Shu and Ying, Yiming and Wang, Xin and Lyu, Siwei},
  journal={Advances in Neural Information Processing Systems},
  volume={33},
  pages={21013--21023},
  year={2020}
}

@article{yao2022large,
  title={Large-scale optimization of partial {AUC} in a range of false positive rates},
  author={Yao, Yao and Lin, Qihang and Yang, Tianbao},
  journal={Advances in Neural Information Processing Systems},
  volume={35},
  pages={31239--31253},
  year={2022}
}

@article{embrechts2018quantile,
  title={Quantile-based risk sharing},
  author={Embrechts, Paul and Liu, Haiyan and Wang, Ruodu},
  journal={Operations Research},
  volume={66},
  number={4},
  pages={936--949},
  year={2018},
  publisher={INFORMS}
}

@article{hu2023rank,
  title={Rank-based decomposable losses in machine learning: A survey},
  author={Hu, Shu and Wang, Xin and Lyu, Siwei},
  journal={IEEE Transactions on Pattern Analysis and Machine Intelligence},
  volume={45},
  number={11},
  pages={13599--13620},
  year={2023},
  publisher={IEEE}
}

@article{blanchet2025distributionally,
  title={Distributionally robust optimization and robust statistics},
  author={Blanchet, Jose and Li, Jiajin and Lin, Sirui and Zhang, Xuhui},
  journal={Statistical Science},
  volume={40},
  number={3},
  pages={351--377},
  year={2025},
  publisher={Institute of Mathematical Statistics}
}

@article{hampel1971general,
  title={A general qualitative definition of robustness},
  author={Hampel, Frank R},
  journal={The Annals of Mathematical Statistics},
  volume={42},
  number={6},
  pages={1887--1896},
  year={1971},
  publisher={Institute of Mathematical Statistics}
}

@article{kratschmer2017domains,
  title={Domains of weak continuity of statistical functionals with a view toward robust statistics},
  author={Kr{\"a}tschmer, Volker and Schied, Alexander and Z{\"a}hle, Henryk},
  journal={Journal of Multivariate Analysis},
  volume={158},
  pages={1--19},
  year={2017},
  publisher={Elsevier}
}

@article{kratschmer2014comparative,
  title={Comparative and qualitative robustness for law-invariant risk measures},
  author={Kr{\"a}tschmer, Volker and Schied, Alexander and Z{\"a}hle, Henryk},
  journal={Finance and Stochastics},
  volume={18},
  pages={271--295},
  year={2014},
  publisher={Springer}
}

@article{aravkin2020trimmed,
  title={Trimmed statistical estimation via variance reduction},
  author={Aravkin, Aleksandr and Davis, Damek},
  journal={Mathematics of Operations Research},
  volume={45},
  number={1},
  pages={292--322},
  year={2020},
  publisher={INFORMS}
}

@article{jiang2024distributionally,
  title={Distributionally favorable optimization: A framework for data-driven decision-making with endogenous outliers},
  author={Jiang, Nan and Xie, Weijun},
  journal={SIAM Journal on Optimization},
  volume={34},
  number={1},
  pages={419--458},
  year={2024},
  publisher={SIAM}
}

@article{hu2022sum,
  title={Sum of ranked range loss for supervised learning},
  author={Hu, Shu and Ying, Yiming and Wang, Xin and Lyu, Siwei},
  journal={Journal of Machine Learning Research},
  volume={23},
  number={112},
  pages={1--44},
  year={2022}
}

@article{oliveira2025trimmed,
  title={Trimmed sample means for robust uniform mean estimation and regression},
  author={Oliveira, Roberto I and Resende, Lucas},
  journal={The Annals of Statistics},
  volume={53},
  number={5},
  pages={2153--2178},
  year={2025},
  publisher={Institute of Mathematical Statistics}
}

@article{loh2025theoretical,
  title={A theoretical review of modern robust statistics},
  author={Loh, Po-Ling},
  journal={Annual Review of Statistics and Its Application},
  volume={12},
  number={1},
  pages={477--496},
  year={2025},
  publisher={Annual Reviews}
}

@article{pensia2025robust,
  title={Robust regression with covariate filtering: Heavy tails and adversarial contamination},
  author={Pensia, Ankit and Jog, Varun and Loh, Po-Ling},
  journal={Journal of the American Statistical Association},
  volume={120},
  number={550},
  pages={1002--1013},
  year={2025},
  publisher={Taylor \& Francis}
}

@book{vershynin2018high,
  title={High-Dimensional Probability: An Introduction with Applications in Data Science},
  author={Vershynin, Roman},
  volume={47},
  year={2018},
  publisher={Cambridge University Press}
}

@book{wainwright2019high,
  title={High-Dimensional Statistics: A Non-Asymptotic Viewpoint},
  author={Wainwright, Martin J},
  volume={48},
  year={2019},
  publisher={Cambridge University Press}
}

@article{bhatia2017consistent,
  title={Consistent robust regression},
  author={Bhatia, Kush and Jain, Prateek and Kamalaruban, Parameswaran and Kar, Purushottam},
  journal={Advances in Neural Information Processing Systems},
  volume={30},
  year={2017}
}

@article{bhatia2015robust,
  title={Robust regression via hard thresholding},
  author={Bhatia, Kush and Jain, Prateek and Kar, Purushottam},
  journal={Advances in Neural Information Processing Systems},
  volume={28},
  year={2015}
}

@article{bousquet2002bennett,
  title={A {Bennett} concentration inequality and its application to suprema of empirical processes},
  author={Bousquet, Olivier},
  journal={Comptes Rendus Mathematique},
  volume={334},
  number={6},
  pages={495--500},
  year={2002},
  publisher={Elsevier}
}

@book{ste2013,
    author = {Boucheron, Stéphane and Lugosi, Gábor and Massart, Pascal},
    title = {Concentration Inequalities: A Nonasymptotic Theory of Independence},
    publisher = {Oxford University Press},
    year = {2013},
    month = {02},
}

@article{rockafellar2000optimization,
  title={Optimization of conditional value-at-risk},
  author={Rockafellar, R Tyrrell and Uryasev, Stanislav},
  journal={Journal of Risk},
  volume={2},
  pages={21--42},
  year={2000}
}

@article{panaretos2019statistical,
  title={Statistical aspects of {Wasserstein} distances},
  author={Panaretos, Victor M and Zemel, Yoav},
  journal={Annual Review of Statistics and Its Application},
  volume={6},
  number={1},
  pages={405--431},
  year={2019},
  publisher={Annual Reviews}
}

@article{Cui_composite,
	author = {Cui, Ying and Pang, Jong-Shi and Sen, Bodhisattva},
	eprint = {https://doi.org/10.1137/18M117337X},
	journal = {SIAM Journal on Optimization},
	number = {4},
	pages = {3344-3374},
	title = {Composite Difference-Max Programs for Modern Statistical Estimation Problems},
	volume = {28},
	year = {2018}}

@article{chen2018robust,
  title={A robust learning approach for regression models based on distributionally robust optimization},
  author={Chen, Ruidi and Paschalidis, Ioannis Ch},
  journal={Journal of Machine Learning Research},
  volume={19},
  number={13},
  pages={1--48},
  year={2018}
}

@article{bertsimas2018characterization,
  title={Characterization of the equivalence of robustification and regularization in linear and matrix regression},
  author={Bertsimas, Dimitris and Copenhaver, Martin S},
  journal={European Journal of Operational Research},
  volume={270},
  number={3},
  pages={931--942},
  year={2018},
  publisher={Elsevier}
}

@article{bottmer2022sparse,
  title={Sparse regression for large data sets with outliers},
  author={Bottmer, Lea and Croux, Christophe and Wilms, Ines},
  journal={European Journal of Operational Research},
  volume={297},
  number={2},
  pages={782--794},
  year={2022},
  publisher={Elsevier}
}

@book{dudley2018real,
  title={Real Analysis and Probability},
  author={Dudley, Richard M},
  year={2018},
  publisher={CRC Press}
}

@article{takeda2009robust,
  title={A robust approach based on conditional value-at-risk measure to statistical learning problems},
  author={Takeda, Akiko and Kanamori, Takafumi},
  journal={European Journal of Operational Research},
  volume={198},
  number={1},
  pages={287--296},
  year={2009},
  publisher={Elsevier}
}

@article{curi2020adaptive,
  title={Adaptive sampling for stochastic risk-averse learning},
  author={Curi, Sebastian and Levy, Kfir Y and Jegelka, Stefanie and Krause, Andreas},
  journal={Advances in Neural Information Processing Systems},
  volume={33},
  pages={1036--1047},
  year={2020}
}

\begin{APPENDICES}

\section{Local Compactness of Population Solution Sets}
\phantomsection 
\label{app:compact}

\begin{proof}{Proof of Proposition \ref{prop:compactness_linear}.}
    By condition \eqref{eq:linear_hyperplane}, there exist $\bar{\delta} \in (0, 1 - \beta)$  and $\varepsilon \in \left(0, 2(\beta - \alpha)\right)$ such that 
    \[
    \sup_{\BFtheta \in \Theta}\mathbb{P}\left(\, \abs{\BFtheta^\top \boldsymbol{X}_{D_0}} < 2 \bar \delta \|\BFtheta \|\,  \right) < \beta- \bar \delta  - \varepsilon.
    \]
    Choose $T>0$ such that
    $\mathbb{P}\left(\,\abs{Y_{D_0}}\leq T -\bar  \delta\right) \geq 1 -{\varepsilon}/{2}.$
    For any $D\in \mathcal{B}_{\text{Prok}}(D_0,\bar \delta)$, by Strassen’s theorem \cite[Theorem 2.13]{huber2009robust}, there exists a coupling of $\boldsymbol{Z}_D\sim D$ and $\boldsymbol{Z}_{D_0}\sim D_0$ such that $\mathbb{P}(\Vert \boldsymbol{Z}_{D} - \boldsymbol{Z}_{D_0}\Vert > \bar \delta) \leq \bar \delta$. Therefore, for every $\BFtheta \in \Theta$,
    \[
    \begin{aligned}
    & \mathbb{P} \left(\left|\BFtheta^\top \boldsymbol{X}_{D}\right| \geq \bar\delta \|\BFtheta\|, |Y_D| \leq T\right)\geq  \mathbb{P}\left(\left|\BFtheta^\top \boldsymbol{X}_{D}\right| \geq \bar\delta \|\BFtheta\|, \,\abs{Y_D} \leq T, \Vert \boldsymbol{Z}_{D} - \boldsymbol{Z}_{D_0}\Vert \leq \bar \delta \right) \\
    &\qquad\geq \mathbb{P}\left(\left|\BFtheta^\top \boldsymbol{X}_{D_0}\right| \geq 2\bar\delta \|\BFtheta\|, \,\abs{Y_{D_0}} \leq T - \bar \delta, \Vert \boldsymbol{Z}_{D} - \boldsymbol{Z}_{D_0}\Vert \leq \bar \delta \right) > 1 - \beta + \frac{\varepsilon}{2}.
    \end{aligned}
    \]
    Then, for sufficiently small $\varepsilon$, 
    \[
    \begin{aligned} 
        & \displaystyle \mathbb{P} \left(\,\mathcal{L}\left(\,\abs{\BFtheta^\top\boldsymbol{X}_D-Y_D}\right)\leq \operatorname{VaR}_{\beta - \frac{\varepsilon}{2} }\left( \, \mathcal{L}\left(\,\abs{\BFtheta^\top\boldsymbol{X}_D-Y_D}\right)\right), \left|\BFtheta^\top \boldsymbol{X}_{D}\right| \geq \bar\delta \|\BFtheta\|, \, \abs{Y_D} \leq T \right) \\
        & \qquad \geq \beta -\frac{\varepsilon}{2} + \mathbb{P}\left(\left|\BFtheta^\top \boldsymbol{X}_{D}\right| \geq \bar\delta \|\BFtheta\| , \, \abs{Y_D} \leq T \right) - 1 > 0.
    \end{aligned}
    \]
    Therefore, 
    \begin{equation}
    \label{eq:stability(a)1}
    \begin{aligned}
        & \lim _{M \rightarrow \infty} \inf_{D \in \mathcal{B}_{\text{Prok}}(D_0, \bar \delta)}\inf _{\{\BFtheta\in\Theta :\Vert \BFtheta \Vert >M\}}\left\{\operatorname{VaR}_{\beta- \frac{\varepsilon}{2}}\left( \,\mathcal{L}\left(\,\abs{\BFtheta^\top\boldsymbol{X}_D-Y_D}\right) \right)\right\}  \geq \lim _{M \rightarrow \infty} \inf _{\{\BFtheta\in\Theta :\Vert \BFtheta \Vert  >M\}}\left\{\mathcal{L}\left(\, \bar\delta\|\BFtheta \| - T\right) \right\}=\infty.
    \end{aligned}
    \end{equation}
    On the other hand, fix $\bar \BFtheta \in \Theta$. For any $t > 0$, 
    \[
    \begin{aligned}
    & \mathbb{P}\left(\,\mathcal{L}\left(\,\abs{\bar \BFtheta^\top\boldsymbol{X}_D-Y_D}\right) \leq t\right) \geq  \mathbb{P}\left(\,\mathcal{L}\left(\,\abs{\bar \BFtheta^\top\boldsymbol{X}_D-Y_D}\right) \leq t \,, \, \Vert \boldsymbol{Z}_D - \boldsymbol{Z}_{D_0}\Vert \leq \bar \delta \right) \\
    & \qquad \geq  \mathbb{P}\left(\,\sup_{\boldsymbol{z} \in \mathcal{B}(\boldsymbol{Z}_{D_0}, \bar \delta)}\mathcal{L}\left(\,\abs{\bar \BFtheta^\top\boldsymbol{x}-y}\right) \leq t  \right) - \bar \delta,
    \end{aligned}
    \]
    which implies 
    \[
    \operatorname{VaR}_{\beta}\left( \,\mathcal{L}\left(\,\abs{\bar \BFtheta^\top\boldsymbol{X}_D-Y_D}\right) \right) \leq \operatorname{VaR}_{\beta + \bar\delta}\left( \,\sup_{\boldsymbol{z} \in \mathcal{B}(\boldsymbol{Z}_{D_0}, \bar \delta)}\mathcal{L}\left(\,\abs{\bar \BFtheta^\top\boldsymbol{x}-y}\right) \right).
    \]
    Since $\mathbb{P}(\Vert \boldsymbol{Z}_{D_0} \Vert \leq T^\prime - \bar \delta) \geq  \beta + \bar\delta$ for some $T^\prime > 0$, we have 
    \[
    \begin{aligned}
    & \sup_{D\in \mathcal{B}_{\text{Prok}}(D_0, \bar \delta)} \inf_{\BFtheta \in \Theta}\operatorname{In-CVaR}_\alpha^{\beta}\left( \,\mathcal{L}\left(\,\abs{\BFtheta^\top\boldsymbol{X}_D-Y_D}\right) \right) \\
     & \qquad \leq \sup_{D\in \mathcal{B}_{\text{Prok}}(D_0, \bar \delta)} \operatorname{VaR}_{\beta}\left( \,\mathcal{L}\left(\,\abs{\bar \BFtheta^\top\boldsymbol{X}_D-Y_D}\right) \right) \leq \sup_{\{\boldsymbol{z}\in \mathcal{Z}: \Vert \boldsymbol{z}\Vert \leq T^\prime \}} \mathcal{L}\left(\,\abs{\bar \BFtheta^\top\boldsymbol{x}-y}\right) < \infty.
    \end{aligned}
    \]
    Together with \eqref{eq:stability(a)1} and Assumption \ref{ass:a2}, there exists a compact set $C \subseteq \Theta$ such that if $d_{\text{Prok}} (D, D_0) < \bar \delta$, then ${\mathcal{S}}_\alpha^\beta(D,f) \subseteq C$. Nonemptiness then follows from the extreme value theorem.
\Halmos\end{proof}

\begin{proposition}
    \label{lem:probabilistic_requirement}
Suppose for $\tau \in (0, \beta)$,
\[
D_0 = (\beta - \tau )D_{0,1} + \left(1-\beta + \tau\right) D_{0,2},
\]
where $D_{0,1}$ is an arbitrary distribution and $\boldsymbol{X}_{D_{0,2}}$ has a bounded density.
Then $D_0$ satisfies the probabilistic requirement \eqref{eq:linear_hyperplane}.
\end{proposition}

\begin{proof}{Proof of Proposition \ref{lem:probabilistic_requirement}.}
    It suffices to show that $
    \lim_{\delta\rightarrow0^+} \sup_{H\in\mathcal{H}_{p+1}}\mathbb{P}(\operatorname{dist}\left((\boldsymbol{X}_{D_{0,2}},1), H\right) < \delta) = 0.$
    Let the density of $\boldsymbol{X}_{D_{0,2}}$ be bounded by $M>0$. For any $\varepsilon > 0,$ choose $T>0$ such that $\mathbb{P}(\| \boldsymbol{X}_{D_{0,2}}\| > T) \leq \varepsilon/2.$ Then, for any $\delta \in (0,1),$
    \[
    \begin{aligned}
        & \displaystyle\sup_{H\in\mathcal{H}_{p+1}}\mathbb{P}\left(\operatorname{dist}\left((\boldsymbol{X}_{D_{0,2}},1), H\right) < \delta\right) 
        \leq  \frac{\varepsilon}{2} + \sup_{H\in\mathcal{H}_{p+1}}\mathbb{P}\left(\operatorname{dist}\left((\boldsymbol{X}_{D_{0,2}},1), H\right) < \delta, \| \boldsymbol{X}_{D_{0,2}}\| \leq T\right).
    \end{aligned}
    \]
    For $(\BFtheta_1, \theta_0) \in \mathbb{R}^{p+1}$ with $\|(\BFtheta_1, \theta_0)\| = 1$ and $\boldsymbol{x} \in \mathcal{X}$ with $\| \boldsymbol{x}\| \leq T$, by the triangle inequality and the Cauchy–Schwarz inequality, we have 
    \[
    |\BFtheta_1^T \boldsymbol{x} + \theta_0| \geq |\theta_0| - |\BFtheta_1^T \boldsymbol{x}| \geq  |\theta_0| - T\sqrt{1-\theta_0^2},
    \]
     where the right-hand side is increasing with respect to $|\theta_0|$. 
     Choose $\theta_0^* \in (0,1)$ such that  
     $\bar \delta= \theta_0^* - T\sqrt{1-\theta_0^{*2}} > 0$. Let $R(\BFtheta) \in \mathbb{R}^{p \times p}$ be orthogonal such that $R(\BFtheta) \BFtheta_1 = \sqrt{1 - \theta_0^2} \BFe_1$. Define $\Tilde{\boldsymbol{x}}(\BFtheta) = R(\BFtheta) \boldsymbol{x}$. Then, for any $\delta < \bar \delta$, we have 
    \[
    \begin{aligned}
    & \operatorname{Volume} \left\{\boldsymbol{x} : |\BFtheta_1^T \boldsymbol{x} + \theta_0| < \delta, \| \boldsymbol{x}\| \leq T\right\}  =  \operatorname{Volume} \left\{\Tilde{\boldsymbol{x}}(\BFtheta) : \left|\sqrt{1-\theta_0^2} \tilde{x}_1(\BFtheta) + \theta_0\right| < \delta, \| \Tilde{\boldsymbol{x}}(\BFtheta)\| \leq T\right\}\\
    & \qquad \leq \frac{2\delta}{\sqrt{1-\theta_0^2}} V_{p-1} T^{p-1} \leq \frac{2\delta}{\sqrt{1-\theta_0^{*2}}} V_{p-1} T^{p-1},
    \end{aligned}
    \]
    where $V_{p-1}$ is the volume of the unit ball in $\mathbb{R}^{p-1}$, yielding 
    \[
    \displaystyle\sup_{H\in\mathcal{H}_{p+1}}\mathbb{P}\left(\operatorname{dist}\left((\boldsymbol{X}_{D_{0,2}},1), H\right) < \delta\right) \leq  \displaystyle \frac{\varepsilon}{2} + \frac{2\delta}{\sqrt{1-\theta_0^{*2}}}M V_{p-1} T^{p-1}.
    \]
    By choosing $ \delta < \min\{\bar\delta, \varepsilon \sqrt{1-\theta_0^{*2}}/ 4 M V_{p-1} T^{p-1}\}$, we have
    ${\sup}_{H\in\mathcal{H}_{p+1}}\mathbb{P}(\operatorname{dist}\left((\boldsymbol{X}_{D_{0,2}},1), H\right) < \delta) < \varepsilon$
    and the proof is completed by letting $\varepsilon \to 0$.
\Halmos\end{proof}

\begin{proof}{Proof of Proposition \ref{prop:compact_glm}.}
   Suppose, on the contrary, that the conclusion fails. Then there exist distributions $\{D_n\}_{n=1}^\infty \subseteq \mathcal{D}$ and $\BFtheta_n \in \mathcal{S}_\alpha^\beta (D_n, f)$ such that $\delta_n \triangleq d_{\text{Prok}}(D_n, D_0) \to 0$ and $\|\BFtheta_n\| \to \infty$. 
    By Strassen's theorem \cite[Theorem 2.13]{huber2009robust}, for each $n \in \mathbb{N}_+$, there exists a coupling of $\boldsymbol{Z}_{D_n} \sim D_n$ and $\boldsymbol{Z}_{D_0} \sim D_0$ such that
    \[
    \mathbb{P}\left(\|\boldsymbol{Z}_{D_n} - \boldsymbol{Z}_{D_0}\|\leq \delta_n\right) \geq 1- \delta_n.
    \]
    Fix any $r \in (0, \beta - \eta)$. By the definition of $\eta$, there exists $\bar \tau>0$ such that 
    \[
    \sup _{\BFtheta \in \Theta} \mathbb{P}\left(|\BFtheta^\top \boldsymbol{X}_{D_0}|<2\bar\tau \|\BFtheta\|\right) \leq \eta + r.
    \]
    For $n \in \mathbb{N}_+$, define the event
    \[
    E_n \triangleq \{\|\boldsymbol{Z}_{D_n} - \boldsymbol{Z}_{D_0}\| \leq \delta_n\} \cap \left\{\left|\BFtheta_n^\top \boldsymbol{X}_{D_0}\right| \geq 2\bar\tau\|\BFtheta_n\|\right\}. 
    \]
    Then
    $  \mathbb{P} \left(E_n^c\right) \leq \delta_n + \eta + r.$
    For sufficiently large $n$, on $E_n$, $
    |\BFtheta_n^\top \boldsymbol{X}_{D_n}| \geq \bar \tau \|\BFtheta_n\|,
    $
    and therefore
    \[
    \begin{aligned}
    & \mathcal{L} \left(\left|g\left(\BFtheta_n^\top \boldsymbol{X}_{D_n} \right) - Y_{D_n}\right|\right) \geq \mathcal{L} \left(\max\left\{\left|g\left(\BFtheta_n^\top \boldsymbol{X}_{D_n} \right) - Y_{D_0}\right| - \left|Y_{D_n} - Y_{D_0}\right|,0\right\}\right)\\
    & \qquad \geq \mathcal{L} \left(\max\left\{\inf_{|v| \geq \bar\tau \|\BFtheta_n\|}|g\left(v\right) - Y_{D_0}| - \delta_n, 0\right\}\right).
    \end{aligned}
    \]
    For every finite $t \geq 0$,
    \[
    \begin{aligned}
    & \mathbb{P} \left(\mathcal{L} \left(\left|g\left(\BFtheta_n^\top \boldsymbol{X}_{D_n} \right) - Y_{D_n}\right|\right)\leq t\right) \leq  \eta + r + \delta_n  + \mathbb{P} \left(\mathcal{L} \left(\max\left\{\inf_{|v| \geq \bar\tau \|\BFtheta_n\|}|g\left(v\right) - Y_{D_0}| - \delta_n, 0\right\}\right) \leq t\right).
    \end{aligned}
    \]
    Since for any $y \in \mathbb{R}$, 
    \[
    \begin{aligned}
    & \limsup_{n \to \infty} \mathbb{I} \left(\mathcal{L} \left(\max\left\{\inf_{|v| \geq \bar \tau \|\BFtheta_n\|}|g\left(v\right) - y| - \delta_n, 0\right\}\right)\leq t\right)  \leq \mathbb{I}\left(\min\left\{\mathcal{L}\left(\,\abs{m -y}\right), \mathcal{L}\left(\,\abs{M -y}\right)\right\} \leq t\right),
    \end{aligned}
    \]
    by Fatou's Lemma, 
    \[
    \limsup_{n \to \infty}\mathbb{P} \left(\mathcal{L} \left(\left|g\left(\BFtheta_n^\top \boldsymbol{X}_{D_n} \right) - Y_{D_n}\right|\right)\leq t\right) \leq \eta + r + \mathbb{P}\left(\min\left\{\mathcal{L}\left(\,\abs{m -Y_{D_0}}\right), \mathcal{L}\left(\,\abs{M -Y_{D_0}}\right)\right\} \leq t\right).
    \]
    By letting $r \downarrow 0$, we obtain that
    \[
    \begin{aligned}
    & \limsup_{n \to \infty} \mathbb{P} \left(\mathcal{L}\left(|g\left(\BFtheta_n^\top \boldsymbol{X}_{D_n} \right) - Y_{D_n}|\right) \leq t\right)  \leq \eta + \mathbb{P}\left(\min\left\{\mathcal{L}\left(\,\abs{m -Y_{D_0}}\right), \mathcal{L}\left(\,\abs{M -Y_{D_0}}\right)\right\} \leq t\right).
    \end{aligned}
    \]
    Therefore, for every $\gamma \in (\eta, 1), $
    \[
    \liminf_{n \to \infty} \operatorname{VaR}_\gamma \left(\mathcal{L}\left(|g\left(\BFtheta_n^\top \boldsymbol{X}_{D_n} \right) - Y_{D_n}|\right)\right) \geq \operatorname{VaR}_{\gamma - \eta} \left(\min\{\mathcal{L}\left(|m-Y_{D_0}|\right), \mathcal{L}\left(|M- Y_{D_0}|\right)\}\right).
    \]
    By Fatou's lemma, we have 
    \begin{equation}
        \label{eq:compo6}
    \begin{aligned}
        & \liminf_{n\to \infty}\operatorname{In-CVaR}_{\alpha}^{\beta} \left(\mathcal{L}\left(\,\abs{g\left(\BFtheta_n^\top \boldsymbol{X}_{D_n}\right) -Y_{D_n}}\right)\right) \\
        & \qquad \geq \frac{1}{\beta -\alpha} \int_{\max\{\alpha, \eta\}}^{\beta} \operatorname{VaR}_{\gamma - \eta} \left(\min\{\mathcal{L}\left(|m-Y_{D_0}|\right), \mathcal{L}\left(|M- Y_{D_0}|\right)\}\right)\mathrm{d} \gamma \\
        & \qquad = \displaystyle\min\left\{1, \frac{\beta - \eta}{\beta - \alpha}\right\}\operatorname{In-CVaR}_{\max\{\alpha - \eta, 0\}}^{\beta  - \eta }\left(\min\left\{\mathcal{L}\left(\,\abs{m -Y_{D_0}}\right), \mathcal{L}\left(\,\abs{M -Y_{D_0}}\right)\right\}\right),
    \end{aligned}
    \end{equation}
    On the other hand, by condition \eqref{eq:compo}, there exists $\bar \BFtheta \in \Theta$ such that
    \begin{equation}
    \label{eq:compo1}
    \begin{aligned}
    \begin{array}{ll}
     & \operatorname{In-CVaR}_{\alpha}^{\beta} \left(\mathcal{L}\left(\,\abs{g\left(\bar\BFtheta^\top \boldsymbol{X}_{D_0}\right) -Y_{D_0}}\right)\right)  
     \\
     & \qquad < \displaystyle\min\left\{1, \frac{\beta - \eta}{\beta - \alpha}\right\}\operatorname{In-CVaR}_{\max\{\alpha - \eta, 0\}}^{\beta  - \eta }\left(\min\left\{\mathcal{L}\left(\,\abs{m -Y_{D_0}}\right), \mathcal{L}\left(\,\abs{M -Y_{D_0}}\right)\right\}\right).
     \end{array} 
    \end{aligned}
    \end{equation}
    Since $d_{\text{Prok}}(D_n, D_0) \to 0$, $D_n$ converges weakly to $D_0$. By Assumption \ref{ass:a1} and the continuous mapping theorem, $\mathcal{L}\left(\left|g\left(\bar \BFtheta^\top \boldsymbol{X}_{D_n}\right) -Y_{D_n}\right|\right) $ converges weakly to $\mathcal{L}\left(\left|g\left(\bar \BFtheta^\top \boldsymbol{X}_{D_0}\right) -Y_{D_0}\right|\right)$. By Lemma \ref{lem:continuity_incvar}, we have 
    \begin{equation}
        \label{eq:compo7}
    \operatorname{In-CVaR}_{\alpha}^{\beta} \left(\mathcal{L}\left(\,\abs{g\left(\bar \BFtheta^\top \boldsymbol{X}_{D_n}\right) -Y_{D_n}}\right)\right) \to \operatorname{In-CVaR}_{\alpha}^{\beta} \left(\mathcal{L}\left(\,\abs{g\left(\bar \BFtheta^\top \boldsymbol{X}_{D_0}\right) -Y_{D_0}}\right)\right).
    \end{equation}
    By optimality of $\BFtheta_n$, 
    \begin{equation}
    \label{eq:compo5}
    \operatorname{In-CVaR}_{\alpha}^{\beta} \left(\mathcal{L}\left(\,\abs{g\left(\BFtheta_n^\top \boldsymbol{X}_{D_n}\right) -Y_{D_n}}\right)\right)  \leq \operatorname{In-CVaR}_{\alpha}^{\beta} \left(\mathcal{L}\left(\,\abs{g\left(\bar\BFtheta^\top \boldsymbol{X}_{D_n}\right)-Y_{D_n}}\right)\right).
    \end{equation}
    Combining \eqref{eq:compo1}-\eqref{eq:compo5} yields 
    \[
    \begin{aligned}
    &\liminf_{n\to\infty} \operatorname{In-CVaR}_{\alpha}^{\beta} \left(\mathcal{L}\left(\,\abs{g\left(\BFtheta_n^\top \boldsymbol{X}_{D_n}\right) -Y_{D_n}}\right)\right)\\
    & \qquad < \displaystyle\min\left\{1, \frac{\beta - \eta}{\beta - \alpha}\right\}\operatorname{In-CVaR}_{\max\{\alpha - \eta, 0\}}^{\beta  - \eta }\left(\min\left\{\mathcal{L}\left(\,\abs{m -Y_{D_0}}\right), \mathcal{L}\left(\,\abs{M -Y_{D_0}}\right)\right\}\right),
    \end{aligned}
    \]
    which contradicts \eqref{eq:compo6}. Nonemptiness then follows from the extreme value theorem.
\Halmos\end{proof}

\begin{proof}{Proof of Proposition \ref{prop:compact_pa}.}
Fix $
\tau \in \left(\max\left\{\alpha, 2- \nu - \min_{k \in [I+J]}\mathbb{P}\left(X_{D_0} \in (t_{k-1}, t_k)\right)\right\}, \beta \right).$ Then
$\eta \triangleq \tau + \min_{k \in [I+J]}\mathbb{P}\left(X_{D_0} \in (t_{k-1}, t_k)\right) + \nu - 2 > 0$. For every $k \in [I+J]$, choose a compact interval $[A_{k}, B_k] \subseteq (t_{k-1}, t_k)$ such that $\mathbb{P}\left(X_{D_0} \in [A_k, B_k ]\right) > \min_{k \in [I+J]}\mathbb{P}\left(X_{D_0} \in (t_{k-1}, t_k)\right) - \eta /4$. Then there exists $M > 0$ such that $\cup_{k=1}^{I+J} [A_k, B_k ] \subseteq [-M+1, M-1]$. 
By condition (a), we can choose $\bar h$ sufficiently small such that 
\begin{equation}
\label{eq:barh}
\sup_{x \in \mathbb{R}} \mathbb{P} \left(X_{D_0} \in [x -\bar h, x+ \bar h]\right) < \frac{\eta}{4(I+J)}.
\end{equation}
 For any $r_1, r_0, r_1^\prime, r_0^\prime \in \mathbb{R}$ and any $x_1, x_2 \in [-M, M]$ satisfying $|x_1 - x_2| \geq \bar h$, we have 
\[
|r_1 -r_1^\prime|\bar h \leq  |(r_1-r_1^\prime) (x_1 - x_2)| \leq |(r_1 -r_1^\prime) x_1 + r_0 - r_0^\prime| + |(r_1 -r_1^\prime) x_2 + r_0 - r_0^\prime|,
\]
which implies that 
\[
|r_1 -r_1^\prime| \leq \frac{|(r_1 -r_1^\prime) x_1 + r_0 - r_0^\prime| + |(r_1 -r_1^\prime) x_2 + r_0 - r_0^\prime|}{\bar h}
\]
and
\[
|r_0 - r_0^\prime| \leq \left(\frac{M}{ \bar h}  + 1\right)\left(|(r_1 -r_1^\prime) x_1 + r_0 - r_0^\prime| + |(r_1 -r_1^\prime) x_2 + r_0 - r_0^\prime|\right).
\]
Therefore, 
\begin{equation}
\label{eq:deviation_pa}
\|(r_1, r_0) - (r_1^\prime, r_0^\prime)\| \leq \left(\frac{M+1}{ \bar h}  + 1\right)\left(|(r_1 -r_1^\prime) x_1 + r_0 - r_0^\prime| + |(r_1 -r_1^\prime) x_2 + r_0 - r_0^\prime|\right).
\end{equation}
By condition (d), we can choose $\varepsilon> 0$ sufficiently small such that
\begin{equation}
\label{eq:kneqs}
\begin{aligned}
& 4 \left(\frac{M+1}{ \bar h}  + 1\right) \left(\mathcal{L}^{-1}\left(\frac{3(\beta - \alpha)}{\beta - \tau } \varepsilon\right) + \mathcal{L}^{-1} (\varepsilon)\right) \\
& \qquad < \min_{k \in [I+J-1]} \left\|(a_{i(k)}^* - b_{j(k)}^*, c_{i(k)}^*-d_{j(k)}^*) - (a_{i(k+1)}^* - b_{j(k+1)}^*, c_{i(k+1)}^*-d_{j(k+1)}^*)\right\|.
\end{aligned}
\end{equation}

We now prove the desired result by contradiction. Suppose that the statement is false. Then there exists a sequence of distributions $\{D_n\}_{n=1}^\infty$ satisfying $\delta_n \triangleq d_{\text{Prok}} (D_n, D_0) \to 0$ and corresponding estimators $\BFtheta_n \in \mathcal{S}_\alpha^\beta (D_n, f)$ such that $\|\BFtheta_n\| \to \infty$. By Strassen's theorem \cite[Theorem 2.13]{huber2009robust}, there exist couplings $(X_{D_n}, Y_{D_n}) \sim D_n$ and $(X_{D_0}, Y_{D_0}) \sim D_0$ such that $\mathbb{P}\left(\|(X_{D_n}, Y_{D_n}) - (X_{D_0},Y_{D_0})\| > \delta_n\right) \leq  \delta_n$. Assume that $\|(X_{D_n}, Y_{D_n}) - (X_{D_0},Y_{D_0})\| \leq  \delta_n$ and $|\xi_{D_0}| \leq \mathcal{L}^{-1}(\varepsilon)$, then
\[
\begin{aligned}
& |f(X_{D_n}, \BFtheta^*) - Y_{D_n}| \leq |f(X_{D_n}, \BFtheta^*) - f(X_{D_0}, \BFtheta^*)| + |Y_{D_n} - Y_{D_0}| + |\xi_{D_0}|\\
& \qquad \leq (\max_{i \in [I]}|a_i^*| + \max_{j \in [J]}|b_j^*| + 1) \delta_n + \mathcal{L}^{-1} (\varepsilon).
\end{aligned}
\]
For sufficiently large $n$, we have $\nu - \delta_n > \beta$ and 
\[
\begin{aligned}
& \mathbb{P}\left(|f(X_{D_n}, \BFtheta^*) - Y_{D_n}| \leq (\max_{i \in [I]}|a_i^*| + \max_{j \in [J]}|b_j^*| + 1) \delta_n + \mathcal{L}^{-1} (\varepsilon)\right) \\
& \qquad \geq \mathbb{P}\left(\|(X_{D_n}, Y_{D_n}) - (X_{D_0},Y_{D_0})\| \leq  \delta_n, |\xi_{D_0}| \leq \mathcal{L}^{-1}(\varepsilon)\right) \geq \beta.
\end{aligned}
\]
Therefore,
\[
\operatorname{VaR}_\beta \left(\mathcal{L}\left(\left|f(X_{D_n},\BFtheta^*) - Y_{D_n}\right|\right)\right) \leq \mathcal{L} \left((\max_{i \in [I]}|a_i^*| + \max_{j \in [J]}|b_j^*| + 1) \delta_n + \mathcal{L}^{-1} (\varepsilon)\right).
\]
By continuity of $\mathcal{L}$ and optimality of $\BFtheta_n$, it follows that 
\[
\begin{aligned}
&\limsup_{n \to \infty} \operatorname{In-CVaR}_\alpha^\beta \left(\mathcal{L}\left(\,\abs{f(\boldsymbol{X}_{D_n}, \BFtheta_n) - Y_{D_n}}\right)\right) \leq \limsup_{n \to \infty} \operatorname{In-CVaR}_\alpha^\beta \left(\mathcal{L}\left(\,\abs{f(\boldsymbol{X}_{D_n}, \BFtheta^*) - Y_{D_n}}\right)\right) < 2\varepsilon.
\end{aligned}
\]
On the other hand,
\[
\operatorname{In-CVaR}_\alpha^\beta \left(\mathcal{L}\left(\,\abs{f(\boldsymbol{X}_{D_n}, \BFtheta_n) - Y_{D_n}}\right)\right) \geq \frac{\beta - \tau}{\beta - \alpha} \operatorname{VaR}_\tau\left(\mathcal{L}\left(\,\abs{f(\boldsymbol{X}_{D_n}, \BFtheta_n) - Y_{D_n}}\right)\right),
\]
which implies that
\[
\limsup_{n \to \infty}\operatorname{VaR}_\tau\left(\mathcal{L}\left(\,\abs{f(\boldsymbol{X}_{D_n}, \BFtheta_n) - Y_{D_n}}\right)\right) < \frac{2(\beta - \alpha)}{\beta - \tau} \varepsilon.
\]
Therefore, for sufficiently large $n$, there exists an event $\mathcal{A}_n$ with $\mathbb{P}(\mathcal{A}_n) \geq \tau - \delta_n$ such that on $\mathcal{A}_n$, $|X_{D_n} -X_{D_0}| \leq \delta_n$ and
\[
\begin{aligned}
& |f(X_{D_n}, \BFtheta_n) - f(X_{D_0}, \BFtheta^*)| \leq |f(X_{D_n}, \BFtheta_n) - Y_{D_n}| + |Y_{D_n} - Y_{D_0}| + |\xi_{D_0}| \leq \mathcal{L}^{-1} \left(\frac{3(\beta - \alpha)}{\beta - \tau} \varepsilon \right) + \delta_n + |\xi_{D_0}|.
\end{aligned}
\]
For each $k \in [I+J]$, we have 
\[
\mathbb{P}\left(\mathcal{A}_n, X_{D_0} \in [A_k, B_k], \mathcal{L}\left(|\xi_{D_0}| \right)\leq \varepsilon\right) \geq \mathbb{P}(\mathcal{A}_n) + \mathbb{P}\left(X_{D_0} \in [A_k, B_k]\right) + \nu - 2 > \frac{3\eta}{4}.
\]
For every $n$, the function $f(\bullet, \BFtheta_n)$ contains at most $I+J$ nondegenerate affine segments with corresponding knots $-\infty = t_{n,0} \leq t_{n,1} \leq \cdots \leq  t_{n, I+J-1} \leq  t_{n, I+J} = \infty$. By condition (a), 
\[
\begin{aligned}
& \mathbb{P}\left(X_{D_n} \in \{t_{n,i}\}_{i=1}^{I+J-1}\right) \leq \mathbb{P}\left(X_{D_n} \in \{t_{n,i}\}_{i=1}^{I+J-1}, |X_{D_n} -X_{D_0}| \leq \delta_n\right) + \mathbb{P}\left( |X_{D_n} -X_{D_0}| >\delta_n\right) \\
& \qquad \leq (I+J-1) \sup_{x \in \mathbb{R}} \mathbb{P}\left(X_{D_0 } \in [x - \delta_n, x+ \delta_n]\right)
 + \delta_n = o(1).\end{aligned}
\]
Therefore, for each sufficiently large $n$ and $k \in [I+J]$, there exists an index $m(k,n) \in [I+J]$ such that
\[
\mathbb{P}\left(\mathcal{A}_n, X_{D_0} \in [A_k, B_k], X_{D_n} \in (t_{n, m(k,n)-1}, t_{n, m(k,n)}),\mathcal{L}\left(|\xi_{D_0}| \right)\leq \varepsilon\right) > \frac{\eta}{4(I+J)}.
\]
For all sufficiently large $n$, we have $B_k + \delta_n < A_{k+1} - \delta_n$, and therefore $m(1,n) \leq m (2,n) \leq \cdots \leq m(I+J,n)$.
For each $k \in [I+J]$, by \eqref{eq:barh} and the definition of $\mathcal{A}_n$, there exist $x_{1,n}, x_{2, n}, x_{1,n}^\prime, x_{2,n}^\prime \in [-M,M]$ satisfying
\begin{enumerate}[label= (T\arabic*)]
    \item \label{t1} $|x_{1,n}^\prime - x_{1, n}| \leq \delta_n, |x_{2,n}^\prime - x_{2, n}| \leq \delta_n$;
    \item \label{t2}for $l = 1, 2$,
    \[
 \begin{aligned}
 & \left|(a_{i(k)}^* -  b_{j(k)}^*) x_{l,n} + (c_{i(k)}^*-d_{j(k)}^*) - (a_{n, i(m(k,n))} - b_{n, j(m(k,n))}) x_{l,n}^\prime \right. \\
 & \qquad \left.- (c_{n, i(m(k,n))}-d_{n, j(m(k,n))})\right| \leq \mathcal{L}^{-1} \left(\frac{3(\beta - \alpha)}{\beta - \tau} \varepsilon \right) + \mathcal{L}^{-1} \left(\varepsilon\right) + \delta_n;
 \end{aligned}
 \]
 \item \label{t3} $|x_{1,n} - x_{2, n}| \geq 2\bar h$ and $|x_{1,n}^\prime - x_{2, n}^\prime| 
 \geq \bar h$.
\end{enumerate}
Taking the difference between the inequalities in \ref{t2}, we obtain
\[
\begin{aligned}
& |(a_{i(k)}^* -  b_{j(k)}^*) (x_{1,n} -x_{2,n})  - (a_{n, i(m(k,n))} - b_{n, j(m(k,n))}) (x_{1,n}^\prime-x_{2,n}^\prime)|\\
& \qquad \leq  2\left(\mathcal{L}^{-1} \left(\frac{3(\beta - \alpha)}{\beta - \tau} \varepsilon \right) + \mathcal{L}^{-1} \left(\varepsilon\right) + \delta_n\right).
\end{aligned}
\]
By \ref{t3} and the fact that $x_{1, n}, x_{2,n} \in [-M, M]$, we have
\begin{equation}
\label{eq:bound_ab}
\begin{aligned}
& |(a_{n, i(m(k,n))} - b_{n, j(m(k,n))})|  \leq \frac{2\left(\mathcal{L}^{-1} \left(\frac{3(\beta - \alpha)}{\beta - \tau} \varepsilon \right) + \mathcal{L}^{-1} \left(\varepsilon \right) + \delta_n\right) + |(a_{i(k)}^* -  b_{j(k)}^*) (x_{1,n} -x_{2,n})|}{|x_{1,n}^\prime-x_{2,n}^\prime|}  = O(1).
\end{aligned}
\end{equation}
Thus, by \ref{t2}, we further obtain that
\begin{equation}
\label{bound:cd}\begin{aligned}
& |c_{n, i(m(k,n))}-d_{n, j(m(k,n))}|   \leq \mathcal{L}^{-1} \left(\frac{3(\beta - \alpha)}{\beta - \tau} \varepsilon\right) + \mathcal{L}^{-1} \left(\varepsilon\right) + \delta_n + |a_{i(k)}^* -  b_{j(k)}^*|| x_{1,n}| \\
&\qquad \quad + |c_{i(k)}^*-d_{j(k)}^*| + |a_{n, i(m(k,n))} - b_{n, j(m(k,n))}|| x_{1,n}^\prime | = O(1).
\end{aligned}
\end{equation}
Moreover, for $l=1,2$, by \ref{t2} and \ref{t3},
 \[
 \begin{aligned}
 & \left|(a_{i(k)}^* -  b_{j(k)}^*) x_{l,n} + (c_{i(k)}^*-d_{j(k)}^*) - (a_{n, i(m(k,n))} - b_{n, j(m(k,n))}) x_{l,n} - (c_{n, i(m(k,n))}-d_{n, j(m(k,n))})\right| \\
 & \qquad \leq \mathcal{L}^{-1} \left(\frac{3(\beta - \alpha)}{\beta - \tau} \varepsilon \right) + \mathcal{L}^{-1} \left(\varepsilon\right) + o(1).
 \end{aligned}
 \]
 By \eqref{eq:deviation_pa}, we have 
 \[
 \begin{aligned}
 & \limsup_{n \to \infty}\left\|\left(a_{i(k)}^* -  b_{j(k)}^*, c_{i(k)}^*-d_{j(k)}^*\right) - \left(a_{n, i(m(k,n))} - b_{n, j(m(k,n))}, c_{n, i(m(k,n))}-d_{n, j(m(k,n))}\right)\right\| \\
 & \qquad \leq 2\left(\frac{M+1}{\bar h} + 1\right) \left(\mathcal{L}^{-1} \left(\frac{3(\beta - \alpha)}{\beta - \tau} \varepsilon\right) + \mathcal{L}^{-1} \left(\varepsilon\right)\right).
 \end{aligned}
 \]
 Suppose that for some sufficiently large $n$ and some $k \in[I+J-1]$, $m(k,n) = m(k+1,n)$. Then, by the triangle inequality,
 \[
 \begin{aligned}
 & \left\|\left(a_{i(k)}^* -  b_{j(k)}^*, c_{i(k)}^*-d_{j(k)}^*\right) - \left(a_{i(k+1)}^* -  b_{j(k+1)}^*, c_{i(k+1)}^*-d_{j(k+1)}^*\right)\right\| \\
 & \qquad \leq \left\|\left(a_{i(k)}^* -  b_{j(k)}^*, c_{i(k)}^*-d_{j(k)}^*\right) - \left(a_{n, i(m(k,n))} - b_{n, j(m(k,n))}, c_{n, i(m(k,n))}-d_{n, j(m(k,n))}\right)\right\|\\
 & \qquad  \quad +\left\|\left(a_{i(k+1)}^* -  b_{j(k+1)}^*, c_{i(k+1)}^*-d_{j(k+1)}^*\right) - \left(a_{n, i(m(k+1,n))} - b_{n, j(m(k+1,n))}, c_{n, i(m(k+1,n))}-d_{n, j(m(k+1,n))}\right)\right\|.
 \end{aligned}
 \]
 Taking the limit of the above inequality, we have 
 \[
 \begin{aligned}
 & \left\|\left(a_{i(k)}^* -  b_{j(k)}^*, c_{i(k)}^*-d_{j(k)}^*\right) - \left(a_{i(k+1)}^* -  b_{j(k+1)}^*, c_{i(k+1)}^*-d_{j(k+1)}^*\right)\right\|\\
 & \qquad \leq  4\left(\frac{M+1}{\bar h} + 1\right) \left(\mathcal{L}^{-1} \left(\frac{3(\beta - \alpha)}{\beta - \tau} \varepsilon\right) + \mathcal{L}^{-1} \left(\varepsilon\right)\right),
 \end{aligned}
 \]
   which contradicts the choice of $\varepsilon$ in \eqref{eq:kneqs}. Therefore, for sufficiently large $n$ and $k \in [I+J-1]$, we have $m(k,n) \neq m(k+1,n) $. Hence $f(\bullet, \BFtheta_n)$ must contain $I+J$  nondegenerate affine segments and we can complete the proof by \eqref{eq:bound_ab}, \eqref{bound:cd} and the fact that $b_0 = d_0 = 0$.
\Halmos\end{proof}

\begin{example}
\label{ex:pa}
The following examples illustrate how Assumption \ref{ass:a2} may fail in the absence of conditions (a), (b) or (c). Throughout, let $\alpha=0, \beta = 3/4$, $\mathcal L(t)=t$, and
$\xi_{D_0} \equiv 0$. 

\emph{(a)}
Let $I=J=1$ and $\Theta=\mathbb R^4$. Suppose
\[
X_{D_0} \sim \frac{1}{2}\iota_{-1} + \frac{1}{2}\operatorname{Unif}(0,1),
\]
where $\iota_{-1}$  denotes the point mass at $-1$ and $\operatorname{Unif}(0,1)$ denotes the uniform distribution on $(0,1)$.
Set $\BFtheta^*=(a,c,b,d)=(1,0,1,0)$. Then
$Y_{D_0}=f(X_{D_0},\BFtheta^*)=\min\{X_{D_0},0\}$. The two affine pieces of $f(\bullet,\BFtheta^*)$
each have probability $1/2$, so conditions (b) and (c) hold. However, condition (a) fails
because $X_{D_0}$ has an atom at $-1$.
For $n \in \mathbb{N}_+$, define
$
\BFtheta_n=(n,n-1,n,n-1).
$
Then
$
f(x,\BFtheta_n)=\min\{nx+n-1,0\}.
$
In particular, $f(-1,\BFtheta_n)=-1$, while $f(x,\BFtheta_n)=0$ for every
$x\in(0,1)$. Thus $f(X_{D_0},\BFtheta_n)=Y_{D_0}$ w.p.1 for every $n$, although
$\|\BFtheta_n\|\to\infty$.

\emph{(b)}
Let
$I = 3$, $J = 1$ and $\Theta = \mathbb{R}^8$. Suppose $X_{D_0} \sim \operatorname{Unif}(0,1)$ and set $\BFtheta^*=\boldsymbol{0}$, so that $Y_{D_0} \equiv 0$. Conditions (a) and (c) hold, whereas condition (b) fails.
For each $n \in \mathbb{N}_+$, define $\BFtheta_n = (8n,-3n,16n,-8n,0,0,16n, -8n)$. Then
\[
f(x, \BFtheta_n)
=
n\max\left\{
8x - 3,
16x-8, 0
\right\}
-
n\max\left\{16x- 8,0
\right\}.
\]
In particular, $f(x, \BFtheta_n) = 0$ for every $x \in (0,3/8] \cup [5/8, 1)$. Thus $\mathbb{P}\left(Y_{D_0} = f(X_{D_0}, \BFtheta_n)\right) \geq 3/4$. Therefore, $\BFtheta_n \in \mathcal{S}_\alpha^\beta (D_0, f)$, while $\|\BFtheta_n\| \to \infty$.

\emph{(c)}
Let $I=J=1$ and $\Theta=\mathbb R^4$. Suppose that 
\[
X_{D_0} \sim \frac14\operatorname{Unif}(-1,0)
+\frac34\operatorname{Unif}(0,1).
\]
Set $\BFtheta^*=(1,0,1,0)$. Then $Y_{D_0}=\min\{X_{D_0},0\}$. Conditions (a)
and (b) hold, while condition (c) fails.
For $n \in \mathbb{N}_+$, let $\BFtheta_n=(n,0,n,0)$. Then
$f(x,\BFtheta_n)=\min\{nx,0\}$. Thus $\mathbb{P}\left(Y_{D_0} = f(X_{D_0}, \BFtheta_n)\right) \geq 3/4$. Therefore, $\BFtheta_n \in \mathcal{S}_\alpha^\beta (D_0, f)$, while $\|\BFtheta_n\| \to \infty$.
\end{example}

\section{Verification of Assumption \ref{ass:b3} for Typical Regression Functions}
\phantomsection 
\label{app:assumption_b3}

\,\newline \noindent$\bullet$ \textbf{Linear regression:} $f\left(\boldsymbol{x}, \BFtheta\right)=\BFtheta^{\top} \boldsymbol{x}$ with $\mathcal{X} = \Theta = \mathbb{R}^{p}$. 

For any $M > 0$, set $\tilde \BFtheta = (M+1)\BFe_1$. Then, 
    \[
    \begin{aligned}
    & \sup_{\boldsymbol{x}\in \mathcal{X}} \inf_{\{\BFtheta\in \Theta: \|\BFtheta\| \leq M\}} \big|{f(\boldsymbol{x}, \tilde \BFtheta) - f(\boldsymbol{x}, \BFtheta)}\big| \geq \lim_{m\to \infty}\inf_{\{\BFtheta\in \Theta: \|\BFtheta\| \leq M\}} \big|f(m\BFe_1, \tilde\BFtheta) - f(m\BFe_1, \BFtheta)\big|\\
    & \qquad =  \lim_{m \to \infty}\inf_{\{\BFtheta\in \Theta: \|\BFtheta\| \leq M\}} \abs{(M+1)m - m\BFtheta_1 }
    \geq \lim_{m \to \infty} m = \infty,
    \end{aligned}
    \]
    which verifies Assumption \ref{ass:b3}.
    
\noindent$\bullet$ \textbf{Polynomial regression:} $f(\boldsymbol{x}, \BFtheta) = \sum_{0 < n_1 + \cdots + n_p \leq n} \theta_{n_1, \cdots, n_p} \prod_{i=1}^p x_i^{n_i}$ with $\mathcal{X} = \mathbb{R}^p$ and $\Theta = \mathbb{R}^{(n+p)!/n!p!-1}$. 

For any fixed $M > 0$, set $\tilde\theta_{n, 0, \cdots, 0} = M+1$ and all other entries zero. Then 
\[
\begin{aligned}
& \sup_{\boldsymbol{x} \in \mathcal{X}} \inf_{\{\BFtheta \in \Theta: \| \BFtheta \| \leq M\}} \big|{f(\boldsymbol{x}, \tilde\BFtheta) - f(\boldsymbol{x}, \BFtheta)}\big| \geq \lim_{m \to \infty}\inf_{\{\BFtheta \in \Theta: \| \BFtheta \| \leq M\}} \big|{f(m\BFe_1, \tilde\BFtheta) - f(m\BFe_1, \BFtheta)}\big|\\
& \qquad = \lim_{m \to \infty}\inf_{\{\BFtheta \in \Theta: \| \BFtheta \| \leq M\}} \Big|{(M+1) m^n - \sum_{i=1}^n \theta_{i, 0,\cdots, 0} m^i}\Big| \geq  \lim_{m \to \infty} \big(m^n - M \sum_{i=1}^{n-1} m^i\big) = \infty,
\end{aligned}
\]
which verifies Assumption \ref{ass:b3}.

\noindent $\bullet$ \textbf{Exponential regression:} $f(\boldsymbol{x}, \BFtheta) = \exp(\BFtheta^\top \boldsymbol{x})$ with $\mathcal{X} =\Theta = \mathbb{R}^p$. 

For any $M > 0,$ set $\tilde\BFtheta = (M+1) \BFe_1$, then 
\[
\begin{aligned}
& \sup_{\boldsymbol{x} \in \mathcal{X}} \inf_{\{\BFtheta \in \Theta: \| \BFtheta \| \leq M\}} \big|{f(\boldsymbol{x}, \tilde\BFtheta) - f(\boldsymbol{x}, \BFtheta)}\big| \geq \lim_{m \to \infty}\inf_{\{\BFtheta \in \Theta: \| \BFtheta \| \leq M\}} \big|{f(m\BFe_1, \tilde\BFtheta) - f(m\BFe_1, \BFtheta)}\big|\\
& \qquad = \lim_{m \to \infty}\inf_{\{\BFtheta \in \Theta: \| \BFtheta \| \leq M\}} \abs{\exp((M+1)m) - \exp(m\theta_{1})}  \geq  \lim_{m \to \infty} \big(\exp((M+1)m)  - \exp(Mm)\big) = \infty,
\end{aligned}
\]
which verifies Assumption \ref{ass:b3}.

\noindent $\bullet$ \textbf{Logarithmic regression:} $f(\boldsymbol{x}, \BFtheta) = \sum_{i=1}^p \theta_i \log(x_i)$ with $\mathcal{X} = (\mathbb{R}_+\backslash\{0\})^{\otimes p}$ and $\Theta = \mathbb{R}^{p}$, where $(\mathbb{R}_+\backslash\{0\})^{\otimes p}$ denotes the Cartesian product of 
$p$ copies of $\mathbb{R}_+\backslash\{0\}$. Verification of Assumption \ref{ass:b3} is straightforward by following the same argument used in the linear regression case.

\noindent $\bullet$ \textbf{Power regression:} $f(\boldsymbol{x}, \BFtheta) =  \prod_{i=1}^p x_i^{\theta_i}$ with $\mathcal{X} = (\mathbb{R}_+\backslash\{0\})^{\otimes p}$ and $\Theta =  \mathbb{R}^p$.

For any fixed $M > 0,$ set $\tilde\theta_1 = M+1$ and $\tilde\theta_i = 0$ for $i \in \{2, \cdots, p\}$, then 
\[
\begin{aligned}
& \sup_{\boldsymbol{x} \in \mathcal{X}} \inf_{\{\BFtheta \in \Theta: \| \BFtheta \| \leq M\}} \big|{f(\boldsymbol{x}, \tilde\BFtheta) - f(\boldsymbol{x}, \BFtheta)}\big|  \geq \lim_{m \to \infty}\inf_{\{\BFtheta \in \Theta: \| \BFtheta \| \leq M\}} \abs{f\left(m\BFe_1 + \sum_{i=2}^p \BFe_i, \tilde\BFtheta\right) - f\left(m\BFe_1 + \sum_{i=2}^p \BFe_i, \BFtheta\right)}\\
& \qquad = \lim_{m \to \infty}\inf_{\{\BFtheta \in \Theta: \| \BFtheta \| \leq M\}} \abs{m^{M+1} -  m^{\theta_1}} \geq  \lim_{m \to \infty} (m^{M+1}  -  m^M) = \infty,
\end{aligned}
\]
which verifies Assumption \ref{ass:b3}.

\noindent $\bullet$ \textbf{Piecewise affine regression:}
$f(\boldsymbol{x}, \BFtheta) = \max_{i \in [I]} \{\boldsymbol{a}_i^\top \boldsymbol{x} + c_i\} - \max_{j \in [J]\cup\{0\}} \{\boldsymbol{b}_j^\top \boldsymbol{x} + d_j\},$ where $(\boldsymbol{b}_0, d_0) = \boldsymbol{0}$, $\BFtheta = \big(\{(\boldsymbol{a}_i,c_i)\}_{i=1}^I, \{(\boldsymbol{b}_j,d_j)\}_{j=1}^J\big) \in \Theta $ with $\mathcal{X} = \mathbb{R}^p$ and  $\Theta = \mathbb{R}^{(I+J)(p+1)}$.

For any $M > 0$, set $\tilde{\boldsymbol{a}}_1 = (2M+3)\BFe_1$, $\tilde{\boldsymbol{a}}_i = \boldsymbol{0}$ for $i \in [I] \backslash \{1\}$, $\tilde{c}_i = 0$  for $i \in [I]$, $\tilde{\boldsymbol{b}}_j = \boldsymbol{0}$ for $j \in [J]$ and $\tilde{d}_j = 0$ for $j \in [J]$. Then, 
    \[
    \begin{aligned}
    & \sup_{\boldsymbol{x}\in \mathcal{X}} \inf_{\{\BFtheta\in \Theta: \|\BFtheta\| \leq M\}} \big|{f(\boldsymbol{x}, \tilde \BFtheta) - f(\boldsymbol{x}, \BFtheta)}\big| \geq \lim_{m\to \infty}\inf_{\{\BFtheta\in \Theta: \|\BFtheta\| \leq M\}} \big|f(m\BFe_1, \tilde\BFtheta) - f(m\BFe_1, \BFtheta)\big|\\
    & \qquad =  \lim_{m \to \infty}\inf_{\{\BFtheta\in \Theta: \|\BFtheta\| \leq M\}} \abs{(2M+3)m - \left(\max_{i \in [I]}\{\boldsymbol{a}_{i,1}m+ c_i\} - \max_{j \in [J]\cup\{0\}}\{\boldsymbol{b}_{j,1}m +d_j\}\right) } 
    \geq \lim_{m \to \infty} m = \infty,
    \end{aligned}
    \]
    which verifies Assumption \ref{ass:b3}.

\section{Proofs of Proposition \ref{prop:assp2-p3_linear} and Proposition \ref{prop:assp2-p3_glinear}}
\phantomsection 
\label{app:assp3-p3}

\begin{proof}{Proof of Proposition \ref{prop:assp2-p3_linear}.}
    For any $\boldsymbol x\in \mathcal{X} \text{ and } \BFtheta_1, \BFtheta_2 \in \Theta$, by the Cauchy--Schwarz inequality, 
    \[
    |f(\boldsymbol{x}, \BFtheta_1) - f(\boldsymbol{x}, \BFtheta_2)| \leq \|\boldsymbol{x}\| \|\BFtheta_1 - \BFtheta_2\|.
    \]
    Thus, Assumption \ref{ass:p2} holds with $\omega(\boldsymbol{x}) = \|\boldsymbol{x}\|$.
    
    For $\BFtheta \in \Theta \backslash\{\BFtheta^*\}$,
    $
    \Delta_\BFtheta(\boldsymbol{x}) = \left(\BFtheta - \BFtheta^*\right)^\top \boldsymbol{x},$ and $ r_\BFtheta^2 = \left(\BFtheta - \BFtheta^*\right)^\top \Sigma_{\boldsymbol{X}} \left(\BFtheta - \BFtheta^*\right).
    $
    Therefore, 
    \begin{equation}
    \label{eq:p2-p3-1}
    r_\BFtheta \geq \sqrt{\lambda_{\min}\left(\Sigma_{\boldsymbol{X}}\right)} \|\BFtheta - \BFtheta^*\|,
    \end{equation}
    which verifies
    Assumption \ref{ass:p3} with $\mu = \sqrt{\lambda_{\min}\left(\Sigma_{\boldsymbol{X}}\right)}$. 
    
    For $D\in \mathcal{D}$ and $\BFtheta_1, \BFtheta_2 \in \Theta \backslash\{\BFtheta^*\},$
    \[
    \begin{aligned}
    & \left\|\frac{\Delta_{\BFtheta_1}}{r_{\BFtheta_1}} - \frac{\Delta_{\BFtheta_2}}{r_{\BFtheta_2}}\right\|_{L_2(D_{\boldsymbol{X}})}
    \leq \left\|\frac{\BFtheta_1 - \BFtheta^*}{r_{\BFtheta_1}} - \frac{\BFtheta_2 - \BFtheta^*}{r_{\BFtheta_2}}\right\| \|\omega\|_{L_2(D_{ \boldsymbol{X}})}.
    \end{aligned}
    \]    
    For any $\BFtheta \in \Theta \backslash\{\BFtheta^*\}$, by \eqref{eq:p2-p3-1}, we have 
    \[
    \left\|\frac{\BFtheta - \BFtheta^*}{r_{\BFtheta}}\right\| \leq \frac{1}{\sqrt{\lambda_{\min}\left(\Sigma_{\boldsymbol{X}}\right)}}.
    \]
    Therefore, for every $\eta > 0$,
    \[
    \begin{aligned}
    &\log\mathcal{N}\left(\mathcal{G}, \|\bullet\|_{L_2(D_{\boldsymbol{X}})}, \eta \|\omega\|_{L_2(D_{\boldsymbol{X}})} \right) \leq \log\mathcal{N}\left(\frac{1}{\sqrt{\lambda_{\min}\left(\Sigma_{\boldsymbol{X}}\right)}}\mathcal{B}^q, \|\bullet\|, \eta \right)  \leq q \log \left(1+ \frac{2}{\eta \sqrt{\lambda_{\min}\left(\Sigma_{\boldsymbol{X}}\right)}}\right),
    \end{aligned}
    \]
    where the last inequality follows from Corollary 4.2.11 of \cite{vershynin2018high}.
    This completes the proof.
\Halmos\end{proof}

\begin{proof}{Proof of Proposition \ref{prop:assp2-p3_glinear}.}
    For any $\boldsymbol x\in \mathcal{X} \text{ and } \BFtheta_1, \BFtheta_2 \in \Theta$, by the mean value theorem and the Cauchy–Schwarz inequality,
    \[
    |f(\boldsymbol{x}, \BFtheta_1) - f(\boldsymbol{x}, \BFtheta_2)| \leq \overline{m}(\boldsymbol{x}) \|\boldsymbol{x}\| \|\BFtheta_1 - \BFtheta_2\|.
    \]
    Thus, Assumption \ref{ass:p2} holds with the stated choice of $\omega$.
    
    For $\BFtheta \in \Theta \backslash\{\BFtheta^*\}$, by the mean value theorem, 
    \[
    \left|\Delta_\BFtheta(\boldsymbol{x})\right| \geq 
    \underline{m}(\boldsymbol{x}) \left|\left(\BFtheta - \BFtheta^*\right)^\top \boldsymbol{x}\right|.
    \]
    It follows that
    \begin{equation}
    \label{eq:p2-p3-2}
    r_\BFtheta 
    \geq \left(\mathbb{E}\left[\left(\underline{m}(\boldsymbol{X})\left(\BFtheta - \BFtheta^*\right)^\top \boldsymbol{X}\right)^2\right]\right)^{1/2} \geq \sqrt{ \lambda_{\min} \left(\Sigma_{g}\right)}\|\BFtheta - \BFtheta^*\|,
    \end{equation}
    which verifies
    Assumption \ref{ass:p3} with $\mu = \sqrt{\lambda_{\min}\left(\Sigma_{g}\right)}$.
    
    For any $\BFtheta \in \Theta \backslash \{\BFtheta^*\}$,
    write 
    $    \BFtheta = \BFtheta^* + t \boldsymbol{u},
    $
    where $t = \|\BFtheta - \BFtheta^*\| \in (0, \operatorname{diam}(\Theta)]$ and $\boldsymbol{u} = (\BFtheta - \BFtheta^*)/\|\BFtheta - \BFtheta^*\|$.
    Define
    \[
    \phi_{t, \boldsymbol{u}} (\boldsymbol{x})\triangleq \frac{g((\BFtheta^* + t\boldsymbol{u})^\top \boldsymbol{x}) - g((\BFtheta^* )^\top \boldsymbol{x})}{t}.
    \]
    Then $\Delta_\BFtheta(\boldsymbol{x}) = t \phi_{t, \boldsymbol{u}}(\boldsymbol{x})$, $r_\BFtheta = t \|\phi_{t, \boldsymbol{u}}\|_{L_2(D_{0, \boldsymbol{X}})}$, and therefore
    \[
    \frac{\Delta_\BFtheta(\boldsymbol{x})}{r_\BFtheta} = \frac{\phi_{t, \boldsymbol{u}}(\boldsymbol{x})}{\|\phi_{t, \boldsymbol{u}}\|_{L_2(D_{0, \boldsymbol{X}})}}.
    \]
    Since
    \[
    \phi_{t, \boldsymbol{u}} (\boldsymbol{x})= \int_0^1 g^\prime\left((\BFtheta^* + st\boldsymbol{u})^\top \boldsymbol{x}\right) \boldsymbol{u}^\top \boldsymbol{x} \mathrm{d} s,
    \]
    for any $t, r \in (0, \operatorname{diam}({\Theta})]$ and unit vectors $\boldsymbol{u}, \boldsymbol{v} \in \mathbb{R}^q$, we have 
    \[
    \begin{aligned}
    &|\phi_{t, \boldsymbol{u}} (\boldsymbol{x}) - \phi_{r, \boldsymbol{v}} (\boldsymbol{x})| \leq \int_0^1 \left|g^\prime\left((\BFtheta^* + st\boldsymbol{u})^\top \boldsymbol{x}\right) \boldsymbol{u}^\top \boldsymbol{x} - g^\prime\left((\BFtheta^* + sr\boldsymbol{v})^\top \boldsymbol{x}\right) \boldsymbol{v}^\top \boldsymbol{x} \right|\mathrm{d} s \\
     &\qquad \leq \int_0^1 \left|g^\prime\left((\BFtheta^* + st\boldsymbol{u})^\top \boldsymbol{x}\right) \boldsymbol{u}^\top \boldsymbol{x} - g^\prime\left((\BFtheta^* + sr\boldsymbol{v})^\top \boldsymbol{x}\right) \boldsymbol{u}^\top \boldsymbol{x} \right|\mathrm{d} s \\
     & \qquad \quad + \int_0^1 \left|g^\prime\left((\BFtheta^* + sr\boldsymbol{v})^\top \boldsymbol{x}\right) \boldsymbol{u}^\top \boldsymbol{x}  - g^\prime\left((\BFtheta^* + sr\boldsymbol{v})^\top \boldsymbol{x}\right) \boldsymbol{v}^\top \boldsymbol{x} \right|\mathrm{d} s \\
     & \qquad \leq \omega(\boldsymbol{x}) \left(|t-r| + \|\boldsymbol{u} - \boldsymbol{v}\|\right).
    \end{aligned}
    \]
    Moreover, by \eqref{eq:p2-p3-2},
    \[
    \|\phi_{t, \boldsymbol{u}}\|_{L_2(D_{0,\boldsymbol{X}})} \geq \sqrt{ \lambda_{\min} \left(\Sigma_{g}\right)}.
    \]
    Therefore, for any distribution $D \in \mathcal{D}$, 
    \[
    \begin{aligned}
    & \left\|\frac{\phi_{t, \boldsymbol{u}}(\boldsymbol{x})}{\|\phi_{t, \boldsymbol{u}}\|_{L_2(D_{0, \boldsymbol{X}})}} - \frac{\phi_{r, \boldsymbol{v}}(\boldsymbol{x})}{\|\phi_{r, \boldsymbol{v}}\|_{L_2(D_{0, \boldsymbol{X}})}}\right\|_{L_2(D_{\boldsymbol{X}})} \\
    & \qquad \leq  \left\|\frac{\phi_{t, \boldsymbol{u}}(\boldsymbol{x})}{\|\phi_{t, \boldsymbol{u}}\|_{L_2(D_{0, \boldsymbol{X}})}} - \frac{\phi_{t, \boldsymbol{u}}(\boldsymbol{x})}{\|\phi_{r, \boldsymbol{v}}\|_{L_2(D_{0, \boldsymbol{X}})}}\right\|_{L_2(D_{\boldsymbol{X}})}+ \left\|\frac{\phi_{t, \boldsymbol{u}}(\boldsymbol{x})}{\|\phi_{r, \boldsymbol{v}}\|_{L_2(D_{0, \boldsymbol{X}})}} -\frac{\phi_{r, \boldsymbol{v}}(\boldsymbol{x})}{\|\phi_{r, \boldsymbol{v}}\|_{L_2(D_{0, \boldsymbol{X}})}}\right\|_{L_2(D_{\boldsymbol{X}})}\\
    & \qquad \leq \underbrace{\left(\frac{1 }{\sqrt{ \lambda_{\min} \left(\Sigma_{g}\right)}} + \frac{\|\omega\|_{L_2(D_{0, \boldsymbol{X}})} }{ \lambda_{\min} \left(\Sigma_{g}\right)} \right)}_{\triangleq C}\|\omega\|_{L_2(D_{\boldsymbol{X}})}\left(|t-r| + \|\boldsymbol{u} - \boldsymbol{v}\|\right).
    \end{aligned}
    \]
    Then, for every $\eta > 0$,
    \[
    \begin{aligned}
    &\log\mathcal{N}\left(\mathcal{G}, \|\bullet\|_{L_2(D_{\boldsymbol{X}})}, \eta \|\omega\|_{L_2(D_{\boldsymbol{X}})} \right)   \leq \log\mathcal{N}\left(\mathcal{B}^q, \|\bullet\|, \frac{\eta}{2C}\right) +\log\mathcal{N}\left([0,\operatorname{diam}(\Theta)], |\bullet|, \frac{\eta}{2C}\right) \\
    & \qquad \leq 2 q \log \left(1+ \frac{4 \max
    \left\{1, \operatorname{diam}({\Theta})
    \right\}}{\eta}\left(\frac{1}{\sqrt{\lambda_{\min}\left(\Sigma_{g}\right)}}+ \frac{\|\omega\|_{L_2(D_{0, \boldsymbol{X}})}}{ \lambda_{\min}\left(\Sigma_{g}\right)}\right)\right),
    \end{aligned}
    \]
    where the last inequality follows from Corollary 4.2.11 of \cite{vershynin2018high}.
    This completes the proof.
\Halmos\end{proof}

\section{Lemmas}
\phantomsection \label{app:lem}

\begin{proof}{Proof of Lemma \ref{lem:counting}.}
    By Markov's inequality, for any $C>0$,
\[
\begin{aligned}
&
\mathbb{E}\left[ \mathbb{I}\left(\xi^2> C\sigma_k^2\left(\varepsilon + \frac{\log(1/\delta)}{n}\right)^{-2/k}\right)\right]  = \mathbb{P}\left(\xi^2> C\sigma_k^2 \left(\varepsilon + \frac{\log(1/\delta)}{n}\right)^{-2/k}\right) \leq \frac{1}{\sqrt{C^{k}}}\left(\varepsilon + \frac{\log(1/\delta)}{n}\right).
\end{aligned}
\]
For any $t \in (0, \tau)$, the Chernoff bound yields
\[
\begin{aligned}
& \mathbb{P}\left(\sum_{i=1}^n \mathbb{I}\left(\xi_i^2 > C\sigma_k^2 \left(\varepsilon + \frac{\log(1/\delta)}{n}\right)^{-2/k} \right) \geq \left(t + \frac{1}{\sqrt{C^k}}\right)\left(\varepsilon + \frac{\log(1/\delta)}{n}\right)n\right)\\
& \qquad\leq \exp\left(-\left(\varepsilon + \frac{\log(1/\delta)}{n}\right)n \left(  \left(t + \frac{1}{\sqrt{C^k}} \right) \ln \left(t\sqrt{C^k} + 1\right) -t\right)\right).
\end{aligned}
\]
Since 
\[
\lim_{C \to \infty} \left(t + \frac{1}{\sqrt{C^k}} \right) \ln \left(t\sqrt{C^k} + 1\right) -t  = \infty,
\]
we can choose $C>0$ sufficiently large such that
\[
\left(t + \frac{1}{\sqrt{C^k}} \right) \ln \left(t\sqrt{C^k} + 1\right) -t \geq 1 \text{ and } t + \frac{1}{\sqrt{C^k}} \leq \tau.
\]
Hence, 
\[
\begin{aligned}
& \mathbb{P}\left(\sum_{i=1}^n \mathbb{I}\left(\xi_i^2 > C\sigma_k^2 \left(\varepsilon + \frac{\log(1/\delta)}{n}\right)^{-2/k} \right) \geq \left(t + \frac{1}{\sqrt{C^k}}\right)\left(\varepsilon + \frac{\log(1/\delta)}{n}\right)n\right)\\
& \qquad \leq \exp\left(-\left(\varepsilon + \frac{\log(1/\delta)}{n}\right)n \right) \leq \delta.
\end{aligned}
\]
Therefore, with probability at least $1 - \delta$,
\[
    \frac{1}{n}\sum_{i=1}^n \mathbb{I}\left(\xi_i^2 > C\sigma_k^2\left(\varepsilon + \frac{\log(1/\delta)}{n}\right)^{-2/k} \right) \leq \tau\left(\varepsilon + \frac{\log(1/\delta)}{n}\right) \leq  1-\beta - \varepsilon,
    \]
    which completes the proof.
\Halmos\end{proof}

\begin{proof}{Proof of Lemma \ref{lem:d_theta}.}

Since $
Y=f(\boldsymbol X,\BFtheta^*)+\xi,$
we have
\[
d_{\BFtheta}(\boldsymbol{X}, Y)
=
\left(\xi-\Delta_{\BFtheta}(\boldsymbol X)\right)^2-\xi^2
=
\Delta_{\BFtheta}(\boldsymbol X)^2
-
2\xi\Delta_{\BFtheta}(\boldsymbol X).
\]
Hence, by Minkowski's inequality,
\[
\begin{aligned}
\left\|d_{\BFtheta}\right\|_{L_2(D_0)}
&\leq
\left\|
\Delta_{\BFtheta}\right\|_{L_4(D_{0,\boldsymbol{X}})}^2
+
2
\left\|
\xi
\Delta_{\BFtheta}
\right\|_{L_2(D_{0})}.
\end{aligned}
\]
By Assumptions~\ref{ass:p2} and \ref{ass:p3},
\[
\left|
\Delta_{\BFtheta}(\boldsymbol{x})
\right|
\leq
\omega(\boldsymbol{x})
\|\BFtheta-\BFtheta^*\|
\leq
\frac{r_{\BFtheta}}{\mu}
\omega(\boldsymbol{x}),
\]
and therefore
\[
\left\|
\Delta_{\BFtheta}\right\|_{L_4(D_{0,\boldsymbol{X}})}^2
\leq
\frac{r_{\BFtheta}^2}{\mu^2}
\left\|
\omega\right\|_{L_4(D_{0,\boldsymbol{X}})}^2.
\]
Moreover, by Assumption~\ref{ass:p1}, $\xi$ is independent of
$\boldsymbol X$, thus
\[
\begin{aligned}
\left\|
\xi
\Delta_{\BFtheta}
\right\|_{L_2(D_{0})}
=\left\|
\xi\right\|_{L_2(D_{0,\xi})}\left\|
\Delta_{\BFtheta}
\right\|_{L_2(D_{0,\boldsymbol{X}})} = \sigma r_{\BFtheta}.
\end{aligned}
\]
Combining the above inequalities gives
\[
\left\|d_{\BFtheta}\right\|_{L_2(D_0)}
\leq
\frac{r_{\BFtheta}^2}{\mu^2}
\left\|\omega\right\|_{L_4(D_{0,\boldsymbol{X}})}^2
+
2\sigma r_{\BFtheta},
\]
which completes the proof.
\Halmos\end{proof}

\begin{proof}{Proof of Lemma \ref{lem:upperbound_emp}.}
Assume that $\sigma >0$ and the case $\sigma=0$ can be treated analogously. For notational convenience, let
\[
\gamma_n \triangleq \sqrt{\frac{q}{n}} + \sqrt{{\varepsilon}+ \frac{\log(1/\delta)}{n}}.
\]
For any measurable function $g: \mathcal{Z} \to \mathbb{R}$, define
\[
\|g\|_n \triangleq \left(\frac{1}{n} \sum_{i=1}^ng(\boldsymbol{Z}_i)^2\right)^{1/2}.
\]
For any $\BFtheta \in \Theta \backslash \{\BFtheta^*\}$, by Assumptions \ref{ass:p2} and \ref{ass:p3}, 
\begin{equation}
\label{eq:emp1}
\left|\frac{\Delta_\BFtheta(\boldsymbol{x})}{r_\BFtheta}\right| \leq \frac{\omega(\boldsymbol{x})\|\BFtheta - \BFtheta^*\|}{r_\BFtheta}  \leq \frac{\omega(\boldsymbol{x})}{\mu}.
\end{equation}
Therefore,
\begin{equation}
\label{eq:emp2}
\left\|\frac{\Delta_{\BFtheta}^2}{\|\omega\|_{L_4(D_{0,\boldsymbol{X}})}^2r_{\BFtheta}^2} \right\|_n \leq \frac{\left(\frac{1}{n} \sum_{i=1}^n  \omega(\boldsymbol{X}_i)^4\right)^{1/2}}{\mu^2\|\omega\|_{L_4(D_{0,\boldsymbol{X}})}^2}  
\text{ and }
\left\|\frac{\xi\Delta_{\BFtheta}}{\sigma \|\omega\|_{L_2(D_{0,\boldsymbol{X}})}r_{\BFtheta}} \right\|_n\leq \frac{\left(\frac{1}{n} \sum_{i=1}^n  \xi_i^2\omega(\boldsymbol{X}_i)^2\right)^{1/2}}{\mu\sigma \|\omega\|_{L_2(D_{0,\boldsymbol{X}})}}
\end{equation}
For any $\BFtheta_1, \BFtheta_2 \in \Theta \backslash \{\BFtheta^*\}$, define
\[
\rho_n(\BFtheta_1, \BFtheta_2) \triangleq \left(\frac{1}{n} \sum_{i=1}^n \left(\frac{\omega(\boldsymbol{X}_i)^2}{\|\omega\|_{L_4(D_{0,\boldsymbol{X}})}^4} + \frac{\xi_i^2}{\sigma^2 \|\omega\|_{L_2(D_{0,\boldsymbol{X}})}^2}\right)\left(\frac{\Delta_{\BFtheta_1}(\boldsymbol{X}_i)}{r_{\BFtheta_1}}-\frac{\Delta_{\BFtheta_2}(\boldsymbol{X}_i)}{r_{\BFtheta_2}}\right)^2\right)^{1/2}.
\]
By \eqref{eq:emp1}, we have 
\begin{equation}
\label{eq:emp3}
\begin{aligned}
& \left\|\frac{\Delta_{\BFtheta_1}^2}{\|\omega\|_{L_4(D_{0,\boldsymbol{X}})}^2 r_{\BFtheta_1}^2} -\frac{ \Delta_{\BFtheta_2}^2}{\|\omega\|_{L_4(D_{0,\boldsymbol{X}})}^2r_{\BFtheta_2}^2}\right\|_n  \\
& \qquad = \left(\frac{1}{n\|\omega\|_{L_4(D_{0,\boldsymbol{X}})}^4}\sum_{i=1}^n\left(\frac{\Delta_{\BFtheta_1}(\boldsymbol{X}_i)}{r_{\BFtheta_1}}+\frac{\Delta_{\BFtheta_2}(\boldsymbol{X}_i)}{r_{\BFtheta_2}}\right)^2\left(\frac{\Delta_{\BFtheta_1}(\boldsymbol{X}_i)}{r_{\BFtheta_1}}-\frac{\Delta_{\BFtheta_2}(\boldsymbol{X}_i)}{r_{\BFtheta_2}}\right)^2\right)^{1/2}
\\
& \qquad \leq \frac{2}{\mu}\left(\frac{1}{n} \sum_{i=1}^n \frac{\omega(\boldsymbol{X}_i)^2}{\|\omega\|_{L_4(D_{0,\boldsymbol{X}})}^4}\left(\frac{\Delta_{\BFtheta_1}(\boldsymbol{X}_i)}{r_{\BFtheta_1}}-\frac{\Delta_{\BFtheta_2}(\boldsymbol{X}_i)}{r_{\BFtheta_2}}\right)^2\right)^{1/2} \leq \frac{2}{\mu} \rho_n(\BFtheta_1, \BFtheta_2)
\end{aligned}
\end{equation}
and
\begin{equation}
\label{eq:emp4}
\begin{aligned}
& \left\|\frac{\xi\Delta_{\BFtheta_1}}{\sigma\|\omega\|_{L_2(D_{0,\boldsymbol{X}})}r_{\BFtheta_1}} -\frac{\xi\Delta_{\BFtheta_2}}{\sigma\|\omega\|_{L_2(D_{0,\boldsymbol{X}})}r_{\BFtheta_2}}\right\|_n \\
& \qquad = \left(\frac{1}{n \sigma^2\|\omega\|_{L_2(D_{0,\boldsymbol{X}})}^2}\sum_{i=1}^n\xi_i^2\left(\frac{\Delta_{\BFtheta_1}(\boldsymbol{X}_i)}{r_{\BFtheta_1}}-\frac{\Delta_{\BFtheta_2}(\boldsymbol{X}_i)}{r_{\BFtheta_2}}\right)^2\right)^{1/2} \leq \rho_n(\BFtheta_1, \BFtheta_2).
\end{aligned}
\end{equation}
Define
\[
\Lambda_n \triangleq \left(\frac{1}{n} \sum_{i=1}^n  \frac{\omega(\boldsymbol{X}_i)^4}{\|\omega\|_{L_4(D_{0,\boldsymbol{X}})}^4}\right)^{1/2} + \left(\frac{1}{n} \sum_{i=1}^n  \frac{\xi_i^2\omega(\boldsymbol{X}_i)^2}{\sigma^2 \|\omega\|_{L_2(D_{0,\boldsymbol{X}})}^2}\right)^{1/2}.
\]
Then, for any $\BFtheta_1, \BFtheta_2 \in \Theta \backslash \{\BFtheta^*\}$, by \eqref{eq:emp1},
\begin{equation}
\label{eq:emp5}
\rho_n(\BFtheta_1, \BFtheta_2) \leq \frac{2}{\mu}\left(\frac{1}{n} \sum_{i=1}^n \left(\frac{\omega(\boldsymbol{X}_i)^2}{\|\omega\|_{L_4(D_{0,\boldsymbol{X}})}^4} + \frac{\xi_i^2}{\sigma^2 \|\omega\|_{L_2(D_{0,\boldsymbol{X}})}^2}\right)\omega(\boldsymbol{X}_i)^2\right)^{1/2} \leq \frac{2}{\mu} \Lambda_n.
\end{equation}
For any $\BFtheta \in \Theta \backslash \{\BFtheta^*\}$,
\[
\begin{aligned}
& \frac{d_\BFtheta(\boldsymbol{X},Y)}{ s_n(q,\varepsilon,\delta, \BFtheta)} = \frac{\Delta_\BFtheta(\boldsymbol{X})^2 - 2\xi\Delta_\BFtheta(\boldsymbol{X})}{ s_n(q,\varepsilon,\delta,\BFtheta)}  = a_n(q, \varepsilon, \delta, \BFtheta) \frac{\Delta_\BFtheta(\boldsymbol{X})^2}{\|\omega\|_{L_4(D_{0,\boldsymbol{X}})}^2 r_\BFtheta^2} - 2b_n (q, \varepsilon, \delta, \BFtheta) \frac{\xi\Delta_\BFtheta(\boldsymbol{X})}{\sigma \|\omega\|_{L_2(D_{0,\boldsymbol{X}})}r_\BFtheta},
\end{aligned}
\]
where
\[
a_n(q, \varepsilon, \delta, \BFtheta) = \frac{ \|\omega\|_{L_4(D_{0,\boldsymbol{X}})}^2r_\BFtheta^2}{s_n(q,\varepsilon,\delta, \BFtheta)} \leq \gamma_n^{-1} 
\text{ and }
b_n(q, \varepsilon, \delta, \BFtheta) = \frac{\sigma \|\omega\|_{L_2(D_{0,\boldsymbol{X}})}r_\BFtheta}{ s_n(q,\varepsilon,\delta,\BFtheta)} \leq \gamma_n^{-1} .
\]
Define  
\[
\tilde{\mathcal{T}}_n  \triangleq \left\{(\boldsymbol{X}, Y) \to  \frac{c_1\Delta_\BFtheta^2(\boldsymbol{X})}{\|\omega\|_{L_4(D_{0,\boldsymbol{X}})}^2r_\BFtheta^2} - \frac{2c_2\xi\Delta_\BFtheta(\boldsymbol{X})}{\sigma\|\omega\|_{L_2(D_{0,\boldsymbol{X}})}r_\BFtheta}: \BFtheta \in \Theta \backslash \{\theta^*\}, 0 \leq c_1, c_2  \leq \gamma_n^{-1} \right\}.
\]
Then $
\mathcal T_n(q,\varepsilon, \delta) \subseteq\tilde{\mathcal{T}}_n
$ and 
$\mathcal{N}\left(\mathcal T_n(q,\varepsilon, \delta),\|\bullet\|_n, t\right) \leq \mathcal{N}\left(\tilde{\mathcal{T}}_n,\|\bullet\|_n, t\right)$.
For any $\BFtheta, \BFtheta^\prime \in \Theta\backslash\{\BFtheta^*\}$, $c_1, c_1^\prime, c_2, c_2^\prime \in \left[0,  \gamma_n^{-1} \right]$, by \eqref{eq:emp2}-\eqref{eq:emp4}, we have 
\[
\begin{aligned}
& \left\|\frac{c_1 \Delta_\BFtheta^2}{\|\omega\|_{L_4(D_{0,\boldsymbol{X}})}^2r_\BFtheta^2} - \frac{2c_2\xi\Delta_\BFtheta}{\sigma\|\omega\|_{L_2(D_{0,\boldsymbol{X}})}r_\BFtheta} - \left(\frac{c_1^\prime \Delta_{\BFtheta^\prime}^2}{\|\omega\|_{L_4(D_{0,\boldsymbol{X}})}^2r_{\BFtheta^\prime}^2} - \frac{2c_2^\prime\xi\Delta_{\BFtheta^\prime}}{\sigma\|\omega\|_{L_2(D_{0,\boldsymbol{X}})}r_{\BFtheta^\prime}}\right)\right\|_n \\
& \qquad \leq c_1 \left\| \frac{\Delta_\BFtheta^2}{\|\omega\|_{L_4(D_{0,\boldsymbol{X}})}^2r_\BFtheta^2} - \frac{\Delta_{\BFtheta^\prime}^2}{\|\omega\|_{L_4(D_{0,\boldsymbol{X}})}^2r_{\BFtheta^\prime}^2}\right\|_n + |c_1 - c_1^\prime| \left\|\frac{\Delta_{\BFtheta^\prime}^2}{\|\omega\|_{L_4(D_{0,\boldsymbol{X}})}^2r_{\BFtheta^\prime}^2}\right\|_n \\
& \qquad \quad + 2 c_2 \left\|\frac{\xi\Delta_\BFtheta}{\sigma\|\omega\|_{L_2(D_{0,\boldsymbol{X}})}r_\BFtheta} - \frac{\xi\Delta_{\BFtheta^\prime}}{\sigma\|\omega\|_{L_2(D_{0,\boldsymbol{X}})}r_{\BFtheta^\prime}}\right\|_n + 2|c_2 - c_2^\prime| \left\|\frac{\xi\Delta_{\BFtheta^\prime}} {\sigma\|\omega\|_{L_2(D_{0,\boldsymbol{X}})}r_{\BFtheta^\prime}}\right\|_n
\\
& \qquad \leq \frac{\frac{2}{\mu} +2 }{\gamma_n} \rho_n(\BFtheta, \BFtheta^\prime) + \frac{\left(\frac{1}{n} \sum_{i=1}^n  \omega(\boldsymbol{X}_i)^4\right)^{1/2}}{\mu^2\|\omega\|_{L_4(D_{0,\boldsymbol{X}})}^2}   |c_1 - c_1^\prime|+ \frac{2\left(\frac{1}{n} \sum_{i=1}^n  \xi_i^2\omega(\boldsymbol{X}_i)^2\right)^{1/2}}{\mu\sigma \|\omega\|_{L_2(D_{0,\boldsymbol{X}})}}|c_2 - c_2^\prime|.
\end{aligned}
\]
Therefore, there exists $C_1>0$ such that
\begin{equation}
\label{eq:emp6}
\begin{aligned}
& \left\|\frac{c_1 \Delta_\BFtheta^2}{\|\omega\|_{L_4(D_{0,\boldsymbol{X}})}^2r_\BFtheta^2} - \frac{2c_2\xi\Delta_\BFtheta}{\sigma\|\omega\|_{L_2(D_{0,\boldsymbol{X}})}r_\BFtheta} - \left( \frac{c_1^\prime\Delta_{\BFtheta^\prime}^2}{\|\omega\|_{L_4(D_{0,\boldsymbol{X}})}^2r_{\BFtheta^\prime}^2} - \frac{2c_2^\prime\xi\Delta_{\BFtheta^\prime}}{\sigma\|\omega\|_{L_2(D_{0,\boldsymbol{X}})}r_{\BFtheta^\prime}}\right)\right\|_n  \leq  \frac{C_1}{\gamma_n}\rho_n(\BFtheta, \BFtheta^\prime)  + C_1\Lambda_n \|c- c^\prime\|.
\end{aligned}
\end{equation}
Thus,
\begin{equation}
\label{eq:emp7}
\begin{aligned}
&
\log\mathcal{N}\left(\mathcal T_n(q,\varepsilon, \delta),\|\bullet\|_n, t\right) \leq \log \mathcal{N}\left(\Theta\backslash\{\BFtheta^*\}, \rho_n, \frac{t\gamma_n}{2C_1}\right) + \log \mathcal{N}\left(\left[0,\gamma^{-1}_n\right]^{\otimes 2}, \|\bullet\|, \frac{t}{2C_1\Lambda_n}\right)\\
& \qquad \leq \log \mathcal{N}\left(\Theta\backslash\{\BFtheta^*\}, \rho_n, \frac{t\gamma_n}{2C_1}\right) + 2 \log \left(1+ \frac{4C_1\Lambda_n}{t\gamma_n
}\right),
\end{aligned}
\end{equation}
where the last inequality follows from Corollary 4.2.11 of
\cite{vershynin2018high}.
We next control the covering numbers of $\Theta\backslash\{\BFtheta^*\}$ associated with $\rho_n$. Define  sample-dependent probability measures $D_n$ with
\[
D_{n,\boldsymbol{X}} = \frac{\sum_{i=1}^n \left(\omega(\boldsymbol{X}_i)^2/\|\omega\|_{L_4(D_{0,\boldsymbol{X}})}^4 + \xi_i^2/\left(\sigma^2 \|\omega\|_{L_2(D_{0,\boldsymbol{X}})}^2\right)\right) \delta_{\boldsymbol{X}_i}}{\sum_{i=1}^n \left(\omega(\boldsymbol{X}_i)^2/\|\omega\|_{L_4(D_{0,\boldsymbol{X}})}^4 + \xi_i^2/\left(\sigma^2 \|\omega\|_{L_2(D_{0,\boldsymbol{X}})}^2\right)\right)}.
\]
Then
\[
\rho_n(\BFtheta_1, \BFtheta_2) = \left(\frac{1}{n}\sum_{i=1}^n \left(\frac{\omega(\boldsymbol{X}_i)^2}{\|\omega\|_{L_4(D_{0,\boldsymbol{X}})}^4} + \frac{\xi_i^2}{\sigma^2 \|\omega\|_{L_2(D_{0,\boldsymbol{X}})}^2}\right)\right)^{1/2} \left\|\frac{\Delta_{\BFtheta_1}}{r_{\BFtheta_1}} -\frac{\Delta_{\BFtheta_2}}{r_{\BFtheta_2}}\right\|_{L_2(D_{n,\boldsymbol{X}})}.
\]
Furthermore,
\[
\begin{aligned}
& \left(\frac{1}{n}\sum_{i=1}^n \left(\frac{\omega(\boldsymbol{X}_i)^2}{\|\omega\|_{L_4(D_{0,\boldsymbol{X}})}^4} + \frac{\xi_i^2}{\sigma^2 \|\omega\|_{L_2(D_{0,\boldsymbol{X}})}^2}\right)\right)^{1/2} \left\|\omega\right\|_{L_2(D_{n,\boldsymbol{X}})} \\
&\qquad
= \left(\frac{1}{n}\sum_{i=1}^n \left(\frac{\omega(\boldsymbol{X}_i)^4}{\|\omega\|_{L_4(D_{0,\boldsymbol{X}})}^4} + \frac{\xi_i^2 \omega(\boldsymbol{X}_i)^2}{\sigma^2\|\omega\|_{L_2(D_{0,\boldsymbol{X}})}^2}\right)\right)^{1/2} \leq \Lambda_n.
\end{aligned}
\]
Hence, for any $u>0$,
\[
\log \mathcal{N}\left(\Theta \backslash\{\BFtheta^*\}, \rho_n, u\right) \leq  \log \mathcal{N}\left(\mathcal{G}, \|\bullet\|_{L_2(D_{n,\boldsymbol{X}})}, \frac{u\|\omega\|_{L_2(D_{n, \boldsymbol{X}})}}{\Lambda_n}\right).
\]
Therefore, by condition \eqref{eq:covering}, we have
\begin{equation}
\label{eq:emp8}
\begin{aligned}
\log \mathcal{N}\left(\Theta \backslash\{\BFtheta^*\}, \rho_n, u\right) \leq  
C q \log \left(1+ \frac{A\Lambda_n}{u} \right).
\end{aligned}
\end{equation}
Combining \eqref{eq:emp7} and \eqref{eq:emp8}, we have
\begin{equation}
\label{eq:emp9}
\begin{aligned}
& \log\mathcal{N}\left(\mathcal T_n(q,\varepsilon, \delta),\|\bullet\|_n, t\right) \leq Cq \log \left(1+ \frac{2AC_1\Lambda_n}{t\gamma_n}\right) + 2 \log \left(1+ \frac{4C_1\Lambda_n}{t\gamma_n
}\right).
\end{aligned}
\end{equation}
Let $\epsilon_{1:n}$ be i.i.d. Rademacher random variables. Conditional
on $\boldsymbol Z_{1:n}$, for every $\lambda\in\mathbb R$,
\[
\begin{aligned}
&
\mathbb E_{\epsilon_{1:n}}\left[
\exp\Bigg(
\frac{\lambda}{n}
\sum_{i=1}^n
\epsilon_i
\left(\frac{d_{\BFtheta_1}(\boldsymbol{X}_i, Y_i)}{s_n(q,\varepsilon,\delta,\BFtheta_1)} -\frac{d_{\BFtheta_2}(\boldsymbol{X}_i, Y_i)}{ s_n(q,\varepsilon,\delta,\BFtheta_2)}\right)
\Bigg) \right]                                                    \\
&\qquad=
\prod_{i=1}^n
\mathbb E_{\epsilon_i}\left[
\exp\Bigg(
\frac{\lambda\epsilon_i}{n}
\left(\frac{d_{\BFtheta_1}(\boldsymbol{X}_i, Y_i)}{s_n(q,\varepsilon,\delta,\BFtheta_1)} -\frac{d_{\BFtheta_2}(\boldsymbol{X}_i, Y_i)}{ s_n(q,\varepsilon,\delta,\BFtheta_2)}\right)
\Bigg)\right] \\
& \qquad \leq
\exp\left(
\frac{\lambda^2}{2n^2}
\sum_{i=1}^n
\left(\frac{d_{\BFtheta_1}(\boldsymbol{X}_i, Y_i)}{s_n(q,\varepsilon,\delta,\BFtheta_1)} -\frac{d_{\BFtheta_2}(\boldsymbol{X}_i, Y_i)}{ s_n(q,\varepsilon,\delta,\BFtheta_2)}\right)^2
\right),
\end{aligned}
\]
where the last inequality holds since
$\mathbb E_\epsilon[\exp(a\epsilon)]
=\cosh(a)\leq\exp(a^2/2)$.
Thus,
\[
\left\{
\frac{1}{n}\sum_{i=1}^n
\epsilon_i
\frac{d_{\BFtheta}(\boldsymbol{X}_i, Y_i)}{ s_n(q,\varepsilon,\delta,\BFtheta)}
:
\BFtheta \in \Theta
\right\}
\]
is a zero-mean sub-Gaussian process with respect to the pseudometric
$\|\bullet\|_n/\sqrt n$. Combining \eqref{eq:emp5}, \eqref{eq:emp6} and the fact that $c_1, c_2 \leq \gamma_n^{-1}$, there exists $\tilde C$ such that
\[
\left\| \frac{c_1 \Delta_\BFtheta^2}{\|\omega\|_{L_4(D_{0,\boldsymbol{X}})}^2r_\BFtheta^2} - \frac{2c_2\xi\Delta_\BFtheta}{\sigma\|\omega\|_{L_2(D_{0,\boldsymbol{X}})}r_\BFtheta} - \left( \frac{c_1^\prime\Delta_{\BFtheta^\prime}^2}{\|\omega\|_{L_4(D_{0,\boldsymbol{X}})}^2r_{\BFtheta^\prime}^2} - \frac{2c_2^\prime\xi\Delta_{\BFtheta^\prime}}{\sigma\|\omega\|_{L_2(D_{0,\boldsymbol{X}})}r_{\BFtheta^\prime}}\right)\right\|_n   \leq    \frac{\tilde{C}\Lambda_n }{\gamma_n}.
\]
By Theorem 5.22 of \cite{wainwright2019high} and \eqref{eq:emp9},
\[
\begin{aligned}
&
\mathbb E_{\epsilon_{1:n}}
\left[
\sup_{\BFtheta \in\Theta \backslash\{\BFtheta^*\}}
\left|
\frac{1}{n}\sum_{i=1}^n
\epsilon_i
\frac{d_{\BFtheta}(\boldsymbol{X}_i, Y_i)}{ s_n(q,\varepsilon,\delta,\BFtheta)}
\right|
\right] \lesssim
\frac{1}{\sqrt n}
\int_0^{\frac{\tilde{C}\Lambda_n }{\gamma_n}}
\sqrt{
\log\mathcal N(\mathcal T_n, \|\bullet\|_n, t)
}\,
\mathrm d t                                            \\
&\qquad\leq
\frac{1}{\sqrt n}
\int_0^{\frac{\tilde{C}\Lambda_n }{\gamma_n}}
\sqrt{
Cq\log\left(1+ \frac{2AC_1\Lambda_n}{t\gamma_n}\right)
}\,
\mathrm d t  +  \frac{1}{\sqrt n}
\int_0^{\frac{\tilde{C}\Lambda_n }{\gamma_n}}
\sqrt{ 2\log\left(1+ \frac{4C_1 \Lambda_n}{t\gamma_n}\right)}
\mathrm d t   \lesssim \sqrt{\frac{q}{n}}
\frac{\Lambda_n}{\gamma_n} \lesssim \Lambda_n,
\end{aligned}
\]
where the second inequality follows from \eqref{eq:emp9}.
By Lemma 11.4 of \cite{ste2013},
\[
\begin{aligned}
\operatorname{Emp}_n(\mathcal T_n(q,\varepsilon, \delta),D_0)
&\leq
2\operatorname{Rad}_n(\mathcal T_n(q,\varepsilon, \delta),D_0)                 \\
& \lesssim
\mathbb E_{D_0^{\otimes n}}
\left[
\left(\frac{1}{n} \sum_{i=1}^n  \frac{\omega(\boldsymbol{X}_i)^4}{\|\omega\|_{L_4(D_{0,\boldsymbol{X}})}^4} \right)^{1/2} + \left(\frac{1}{n} \sum_{i=1}^n  \frac{\xi_i^2\omega(\boldsymbol{X}_i)^2}{\sigma^2 \|\omega\|_{L_2(D_{0,\boldsymbol{X}})}^2}\right)^{1/2}\right]   \leq
2,
\end{aligned}
\]
which completes the proof.
\Halmos
\end{proof}

\begin{lemma}
    \label{lem:continuity_incvar}
    
    For any nonnegative univariate random variables $X_n$ converging weakly to $X_0$, and for $0 \le \alpha < \beta < 1$, we have {\rm $\operatorname{In-CVaR}_\alpha^\beta(X_n) \rightarrow \operatorname{In-CVaR}_\alpha^\beta(X_0)$} as $n\to \infty$. 
    
\end{lemma}

\begin{proof}{Proof of Lemma \ref{lem:continuity_incvar}.}
    
    When $\alpha \neq 0$, the claim follows from Theorem 3.4 of \cite{cont2010robustness}. It thus suffices to consider $\alpha = 0$. Let $F_{X_n}$ and $F_{X_0}$ denote the cumulative distribution functions of $X_n$ and $X_0$, respectively. 
    Define the Lévy distance 
    \[
    d_{\text{Lévy}}(F_{X_n}, F_{X_0}) \triangleq \inf_{\varepsilon>0} \{\varepsilon: F_{X_0}(x-\varepsilon)-\varepsilon \leq F_{X_n}(x) \leq F_{X_0}(x+\varepsilon)+\varepsilon, \forall x \in \mathbb{R}\}.
    \]
    By Theorem 2.9 of \cite{huber2009robust}, weak convergence of $X_n$ to $X_0$ implies $\bar \varepsilon_n = d_{\text{Lévy}}(F_{X_n}, F_{X_0}) \rightarrow 0$. For sufficiently large $n$, we have $\bar\varepsilon_n < \beta $ and
    \[
    \begin{aligned}
    & \frac{1}{\beta} \left(-\bar \varepsilon_n + \int_{\bar \varepsilon_n}^\beta \operatorname{VaR}_{\gamma - \bar \varepsilon_n}(X_0) \mathrm{d}\gamma\right) \leq \frac{1}{\beta}\int_{ \bar \varepsilon_n}^ \beta \operatorname{VaR}_\gamma(X_n) \mathrm{d}\gamma  \leq \operatorname{In-CVaR}_0^\beta(X_n) \leq \frac{1}{\beta } \left( \bar \varepsilon_n + \int_0^\beta \operatorname{VaR}_{\gamma+\bar \varepsilon_n} (X_0) \mathrm{d} \gamma \right).
    \end{aligned}
    \]
    Therefore, 
    \[
    \begin{aligned}
    & \abs{\operatorname{In-CVaR}_0^\beta(X_n) - \operatorname{In-CVaR}_0^\beta(X_0)}  \leq  \frac{1}{\beta}\left(2\bar\varepsilon_n + \int_{0}^{\bar\varepsilon_n} \operatorname{VaR}_{\gamma+\bar \varepsilon_n} (X_0) \mathrm{d}\gamma + \int_{\bar\varepsilon_n}^\beta \left( \operatorname{VaR}_{\gamma+\bar \varepsilon_n} (X_0) - \operatorname{VaR}_{\gamma - \bar \varepsilon_n}(X_0) \right) \mathrm{d}\gamma\right).
    \end{aligned}
    \]
    By letting $n \rightarrow \infty$, $\operatorname{VaR}_{\gamma+\bar \varepsilon_n} (X_0) - \operatorname{VaR}_{\gamma - \bar \varepsilon_n}(X_0) \rightarrow 0 $ for all $\gamma \in (0,1)$, except at countably many discontinuities, yielding $\big|{\operatorname{In-CVaR}_\alpha^\beta(X_n) - \operatorname{In-CVaR}_\alpha^\beta(X_0)}\big| \rightarrow 0$, which completes the proof.    
\Halmos\end{proof}

\end{APPENDICES}

\end{document}